\documentclass[12pt]{amsart}

\usepackage[headheight=1cm, left=2cm, right=2cm, top=2cm, bottom=2cm, marginparwidth=1.7cm, marginparsep=4mm]{geometry}

\usepackage[utf8]{inputenc}
\usepackage[T1]{fontenc}

\usepackage{amsthm}
\usepackage{amsmath}
\usepackage{amsfonts}
\usepackage{amssymb}

\usepackage{xcolor}

\usepackage[colorlinks,
    colorlinks,
    linkcolor={red!80!black},
    citecolor={blue!80!black},
]{hyperref}

\newcommand{\CC}{\mathbb{C}}
\newcommand{\PP}{\mathbb{P}}

\numberwithin{equation}{section}
\counterwithin{figure}{section}
\newtheorem{theorem}{Theorem}[section]
\newtheorem{corollary}[theorem]{Corollary}
\newtheorem{lemma}[theorem]{Lemma}

\newtheorem{proposition}[theorem]{Proposition}
\newtheorem{conjecture}[theorem]{Conjecture}
\theoremstyle{definition}
\newtheorem{definition}[theorem]{Definition}
\newtheorem{remark}[theorem]{Remark}
\newtheorem{example}[theorem]{Example}

\usepackage{todonotes}
\usepackage[dvipsnames]{xcolor}
\usepackage{booktabs}
\usepackage{setspace}
\usetikzlibrary{matrix,arrows, cd, calc, bending, quotes}

\usepackage{enumitem}
\setlist[enumerate,1]{label=(\alph*), ref=(\alph*), itemsep=0em}

\newcommand{\kk}{\Bbbk}
\newcommand{\OO}{\mathcal{O}}
\newcommand{\spann}[1]{\left\langle #1\right\rangle}

\DeclareMathOperator{\End}{End}
\DeclareMathOperator{\GL}{GL}
\DeclareMathOperator{\SL}{SL}

\DeclareMathOperator{\expdim}{expdim}
\DeclareMathOperator{\pr}{pr}

\newcommand{\BBname}{Bia{\l}ynicki-Birula}%
\newcommand{\Gmult}{\mathbb{G}_m}%

\DeclareMathOperator{\Hom}{Hom}
\DeclareMathOperator{\Spec}{Spec}

\DeclareMathOperator{\Gr}{Gr}
\DeclareMathOperator{\Hilb}{Hilb}
\DeclareMathOperator{\Quot}{Quot}
\newcommand{\sm}{\mathrm{sm}}

\DeclareMathOperator{\VSP}{VSP}
\DeclareMathOperator{\cVSP}{cVSP}
\newcommand{\VSPbord}{\underline{VSP}}

\newcommand{\VV}{\mathcal{V}}
\newcommand{\WW}{\mathcal{W}}
\newcommand{\ZZ}{\mathcal{Z}}
\newcommand{\VVhat}{\VV_{\,\hat{i}}}
\newcommand{\Vhatpar}[1]{V_{\,\hat{#1}}}
\newcommand{\Vhat}{\Vhatpar{i}}

\newcommand{\pihat}{\widehat{\pi}}
\newcommand{\rhohat}{\widehat{\rho}}
\newcommand{\alphahat}{\widehat{\alpha}}
\newcommand{\UU}{\mathcal{U}}
\newcommand{\Tp}{\widetilde{T}^i}
\newcommand{\Tppar}[1]{\widetilde{T}^{#1}}
\newcommand{\Tpparfr}[1]{{\widetilde{T}^{\mathrm{fr}#1}}}
\newcommand{\varphifr}{\varphi^{\mathrm{fr}}}

\newcommand{\Xppar}[1]{X'^{#1}}
\newcommand{\GLV}{\GL(V_{\bullet})}
\newcommand{\GLW}{\GL(W_{\bullet})}
\newcommand{\SLV}{\SL(V_{\bullet})}
\newcommand{\tspace}{V_{1, \dots, d}}
\newcommand{\tspacecurly}{\VV_{1, \dots, d}}
\newcommand{\tconcspace}{W_{1, \dots, d}}
\newcommand{\tconcspacecurly}{\WW_{1, \dots, d}}
\newcommand{\tconcspaceconc}{\tconcspace^{\mathrm{conc}}}
\newcommand{\LL}{\mathcal{L}}
\newcommand{\Lmot}{\mathbb{L}}
\DeclareMathOperator{\OpUnres}{Unres}
\newcommand{\unres}{\OpUnres_{i}(T)}
\newcommand{\unrespar}[2]{\OpUnres_{#1}(#2)}
\newcommand{\unresfull}{\unresfullpar{T}}
\newcommand{\unresfullpar}[1]{\OpUnres(#1)}

\newcommand{\unresfrfullpar}[1]{\OpUnres^{\mathrm{fr}}(#1)}
\newcommand{\unresfrfull}{\unresfrfullpar{T}}
\newcommand{\unresfrpar}[2]{\OpUnres^{\mathrm{fr}}_{#1}(#2)}
\newcommand{\into}{\hookrightarrow}
\newcommand{\onto}{\twoheadrightarrow}
\newcommand{\Cen}[1]{\operatorname{Cen}_{#1}}%
\newcommand{\eval}{\epsilon_{\cA}}

\newcommand{\TGr}{T^{\Gr}}

\DeclareMathOperator{\rk}{rk}
\DeclareMathOperator{\brk}{brk}
\DeclareMathOperator{\bcrk}{bcrk}

\DeclareMathOperator{\smrk}{smrk}
\DeclareMathOperator{\Sym}{Sym}
\DeclareMathOperator{\Bl}{Bl}
\DeclareMathOperator{\Sing}{Sing}
\DeclareMathOperator{\Iso}{Iso}
\DeclareMathOperator{\im}{im}
\DeclareMathOperator{\Z}{Zeros}
\DeclareMathOperator{\Seg}{Segre}
\newcommand{\op}{\circ}

\newcommand{\sigmahat}{\widehat{\sigma}}%
\newcommand{\csigma}{c\sigma}%
\newcommand{\csigmaOneGen}{c\sigma^{1gen}}%
\newcommand{\csigmafr}{c\sigma^{\mathrm{fr}}}%
\newcommand{\ckappa}{c\kappa}%
\newcommand{\Asigma}{\mathrm{A}\sigma}%
\newcommand{\AsigmaHilb}{\Asigma^{\Hilb}}%
\newcommand{\csigmahat}{\widehat{\csigma}}%
\newcommand{\csigmahatfr}{\widehat{\csigma}^{\mathrm{fr}}}%
\newcommand{\ckappahat}{\widehat{\ckappa}}%
\newcommand{\cA}{\mathcal{A}}%

\usepackage{marginnote}
\begin{document}

\title{Unrestrictions and concise secant varieties}
\author{Jakub Jagie{\l}{\l}a}
\address[Jakub Jagie{\l}{\l}a, Joachim Jelisiejew]{Faculty of Mathematics, Informatics and Mechanics, University of Warsaw,
 ul.~Stefana Banacha
2, 02-097 Warsaw, Poland.}
\email{j.jagiella@uw.edu.pl}

\author{Joachim Jelisiejew}
\email{j.jelisiejew@uw.edu.pl}
\thanks{Both authors supported by National Science Centre grant 2023/50/E/ST1/0033.}
\date{\today{}}

\begin{abstract}
    We introduce the concise secant varieties, which are, informally speaking, modular partial desingularisations of secant varieties to
    Segre embeddings. They are projective and birational to the abstract secant
    varieties, yet each of their points corresponds to a \emph{concise} tensor of appropriate
    border rank (that is, to a minimal border rank tensor).

    We discuss implications throughout the theory of tensors,
    including a characterisation of border rank $\leq r$ tensors as
    unrestrictions of minimal border rank $r$ tensors, a characterisation of tensors with cactus rank $\leq r$, a concise version
    of border apolarity including the fixed point theorem, concise Varieties of Sums of Powers, counting points on
    the second secant variety, connections to defectivity and identifiability in the Segre case, to the
    Salmon conjecture etc.
\end{abstract}

\maketitle


\section{Introduction}

For a smooth nondegenerate projective variety $X \subseteq \mathbb{P}^N$, its $r$-th secant variety is
\begin{equation}\label{eq:secantDef}
    \sigma_r(X) = \Sigma_{r-1}(X) = \overline{ \left\{ \spann{x_1, \ldots ,x_r}\ | x_1, \ldots ,x_r\in X \right\}}\subseteq \mathbb{P}^N,
\end{equation}
where $\spann{-}$ denotes the projective span and the overline denotes the
Zariski-closure in $\mathbb{P}^N$. Secant varieties are of classical interest in algebraic
geometry and a source of special varieties (see~\cite{ranestad_schreyer_VSP,
Ranestad_Voisin__VSP, Ein_Niu_Park, Park}) and admit numerous applications, for
example to algebraic complexity theory and theory of
tensors~\cite{Landsberg__tensors, Buczyska_Buczynski__border, Conner_Harper_Landsberg,
Massarenti_Mella}.

The birational geometry of $\sigma_r(X)$ is successfully studied by means of its natural parameterisation, called the abstract secant variety.
For finer geometric applications, one needs to understand the singularities of $\sigma_r(X)$. Here, one is met by twin problems
\begin{itemize}
    \item the points in the closure in~\eqref{eq:secantDef} are very poorly understood already for $r\geq 4$,
    \item the points of $\sigma_{r-1}(X) \subseteq \sigma_r(X)$ are very singular: $\mathbb{P}T_{p}\,\sigma_{r}(X) = \mathbb{P}^N$ for every $p\in \sigma_{r-1}(X)$.
\end{itemize}
For $r=1$, we have $\sigma_1(X) = X$ and no closure is necessary. For $r=2$, is was known classically that the closure involves only adding all lines tangent to $X$. Generalising this, one observes~\cite{Buczyska_Buczynski__cactus} that for every finite smoothable scheme $Z \subseteq X$ of degree $r$, its span is contained in $\sigma_r(X)$, so that
\begin{equation}\label{eq:secantDefBetter}
    \sigma_r(X) = \overline{\left\{ \spann{Z} \ |\ [Z]\in \Hilb_r^{\sm}(X) \right\}},
\end{equation}
where $\Hilb^{\sm}_r(X)$ is the smoothable component of the Hilbert scheme of points.
When $X$ is embedded by a line bundle positive enough with respect to $r$, no closure in~\eqref{eq:secantDefBetter}
is necessary~\cite{Buczyska_Buczynski__cactus}. This leads to a resolution of singularities of $\sigma_r(X)$ by an appropriate tautological bundle
$\mathcal{B}_r(X)$ on $\Hilb_r^{\sm}(X)$ and to numerous deep results~\cite{Ullery, Ein_Niu_Park, Park,
Agostini_Park} for $\dim X \leq 2$ or $r\leq 3$, when $\Hilb_r(X) = \Hilb_r^{\sm}(X)$ is smooth.

Without positivity, the closure in~\eqref{eq:secantDefBetter} is in general necessary for every $r\geq 3$. Moreover, the boundary
of $\sigma_r(X)$ is much more complicated~\cite{Furukawa_Han}, so this case seemed hopeless, with some results only for $r=3$ for certain $X$, see~\cite{Buczynski_Landbserg_sigma3}. This is despite
its applications, for example for the theory of tensors, for which the most important $X$ to consider is
the Segre embedding:
\[
    \Seg = \mathbb{P}V_1 \times \ldots \times \mathbb{P}V_d\into \mathbb{P}\tspace,
\]
where $d\geq 3$ and $\tspace := V_1\otimes \ldots \otimes V_d$. Below we concentrate on
this case and use the abbreviation $\sigma_r := \sigma_r(\Seg)$. The Segre
embedding is not positive enough in the sense above.

For the Segre variety, minimal border rank tensors were recently introduced in~\cite{JLP} and the authors show that they admit desirable properties,
such as a clear structure theory with moduli and deformations. Buczy{\'n}ski
observed (unpublished \cite{Bucz_Communication}) that every tensor in $\sigma_r$ is a restriction of a minimal border rank tensor.

\subsection{High level overview of results}

In this article, we construct the concise secant varieties $\csigma_r$ which are analogues of the bundles $\mathcal{B}_r(X)$ in the Segre case.
To obtain such an analogue, we need a certain shift of perspective (for example, to obtain an analogue of the Hilbert scheme), which
creates ripples throughout the theory.
We summarise this shift informally here. The reader should view this as a ``big picture'', while the finer details are given only below.

%

\begin{enumerate}
    \item\label{it:unres} We show that the geometric relation of
        lying in the span is equivalent to the algebraic relation of restriction, see
        the Cactus Apolarity Lemma~\ref{ref:restriction:introthm}.
    \item\label{it:minbrank} We construct the concise secant variety
        $\rho\colon \csigma_r\to \sigma_r$ as an iterated universal unrestriction scheme,
        see~\S\ref{ref:construction}.
    \item The morphism $\rho$ becomes the universal
        restriction map, which generalises Buczy{\'n}ski's result and makes it geometric,
        see~\S\ref{ref:unrestrictionsIntro}. On the elementary level, this allows for algorithmic
        construction of minimal border rank unrestrictions.
    \item\label{it:proj} The variety $\csigma_r$ is \emph{projective}. Thus we can degenerate tensors in
        $\csigma_r$, for example using torus actions, yet still obtain minimal border rank tensors in the limit!
        This is in stark contrast with the usual secant variety, where any degeneration could in principle fall into the abyss of the closure,
        and much care is necessary to avoid it, see for example~\cite{Galuppi_Oneto}.
    \item To obtain a variety $\csigma_r$ satisfying both~\ref{it:minbrank} and~\ref{it:proj},
        we consider tensors as points of fibres of vector bundles, rather than points of a fixed vector space $\tspace$.
    \item The variety $\csigma_r$ \emph{always} has the expected dimension and
        in fact it is birational to the abstract secant
        variety mentioned above, see~\S\ref{ssec:abstractSecantsIntro}.
        The description of all points of $\csigma_r$ and the possibility of degeneration can be useful
        for birational questions such as defectivity and identifiability of secants to Segre, see~\S\ref{ssec:abstractSecantsIntro} and~\S\ref{ssec:identifiability}.
    \item The fibres of $\rho$ give concise versions of the celebrated \textbf{V}arieties of \textbf{S}ums of \textbf{P}owers of Ranestad-Schreyer,
        see~\S\ref{ssec:VSP}. The description of points, singularities,~etc.~applies to these varieties as well. This is very useful,
        as these varieties come with interesting global geometry~\cite{ranestad_schreyer_VSP, Ranestad_Voisin__VSP}, but checking their smoothness is cumbersome, again due to the need of closure.
    \item The fibres of $\rho$
        also give rise to a new version of the border apolarity~\cite{Buczyska_Buczynski__border}, see~\S\ref{ssec:borderApolarity}. This version bypasses
        the need for the multigraded Hilbert schemes (or Hilbert schemes whatsoever), hence gives hope for new border rank lower bounds such as~\cite{Conner_Harper_Landsberg, Conner_Huang_Landsberg}.
    \item The singularities of $\csigma_r$ are much better behaved than the
        singularities of $\sigma_r$ for all $r$, see \S\ref{ssec:singularities} and~\S\ref{ssec:smoothEquivalences}, especially~\eqref{eq:smoothEquivalenceCsigma}.         The variety $\csigma_r$ is smooth for $r\leq 3$
        and has terminal singularities for $r = 4$, see Theorem~\ref{ref:resolutions:introthm}.
        Smoothness allows for computing the $K_0(\operatorname{Var})$ class of $\csigma_2$ and $\sigma_2$, see Proposition~\ref{ref:motivecsigma2:prop}.
        (A similar computation for $r=3$ will be given elsewhere.)
        It also gives a different approach to the classical Salmon problem~\S\ref{ssec:Salmon}.
\end{enumerate}
We construct also concise cactus varieties and other variants which we omit here.

\subsection{Unrestrictions}\label{ref:unrestrictionsIntro}

From the perspective of tensors, determining $\sigma_r$ is determining the tensors $[T]\in \mathbb{P}\tspace$ that have border rank $\brk(T) \leq r$.
Here, $T$ has border rank $\leq r$ if there is a degeneration $T = \lim_{t\to 0} T_t$, where $[T_t]\in \mathbb{P}\tspace$ have rank $\leq r$.
When all coordinate spaces $V_{\bullet}$ have dimension $r$, Buczy{\'n}ski
observed~\cite{Bucz_Communication} that $\brk(T)\leq r$ exactly when $T$ is a
restriction of a minimal border rank tensor introduced in~\cite{JLP}. This is important, since
restriction is a linear-algebraic operation, much more explicit than passing to the limit.

We show that every degeneration $[T] =
\lim_{t\to 0} [T_t]$ can be split into two steps:
\begin{enumerate}
    \item\label{it:mbrk} the \emph{minimal border rank limit} $\Tppar{} =
        \lim_{t\to 0} T_t$,
    \item\label{it:restr} a restriction from $\Tppar{}$ to $T$.
\end{enumerate}
We prove this in much greater generality in Theorem~\ref{p:algorithm_segre} below. The result is explicit and constructive. Schematically,
the splitting can be presented as
\begin{equation}\label{eq:factorisation}
    \begin{tikzcd}
        & \Tppar{}\ar[d, mapsto, "\mathrm{restr.}"]\\
        T_t \ar[ur, mapsto, "\mathrm{limit}"] \ar[r, mapsto, "\mathrm{limit}"] & T
    \end{tikzcd}
\end{equation}


\subsection{Construction of the concise secant variety}\label{ref:construction}

Let $\sigma_r^{\op}\subseteq \sigma_r$ be the locus of tensors of rank $\leq
r$. This is the locus under the overline in~\eqref{eq:secantDef}, so that
$\sigma_r$ is the closure of $\sigma_r^{\op}$.
The variety $\csigma_r$ fits into the diagram
\[
    \begin{tikzcd}
        & \csigma_r\ar[d, "\rho"]\\
        \sigma_r^{\op} \ar[r, hook, "\mathrm{op}"]\ar[ru, hook, dashed, "\mathrm{op}"] & \sigma_r
    \end{tikzcd}
\]
which upgrades~\eqref{eq:factorisation}:
the limit $\Tppar{} = \lim_{t\to 0} T_t$ gives a degeneration to a special point on a curve on $\csigma_r$,
that lifts the analogous degeneration $T = \lim_{t\to 0} T_t$ on $\sigma_r$.
The dashed arrow denotes a rational map; this map is defined on the concise locus.
The map $\rho$ applied to $\Tppar{}$ becomes the restriction map.

\newcommand{\Taf}{T^{\mathrm{af}}}
\newcommand{\Tproj}{T^{\mathrm{proj}}}

While above we assumed $\dim
V_i \geq r$, the construction of $\csigma_r$ is general
and requires no such assumption. We discuss it now.
The map $\rho\colon \csigma_r \to \sigma_r$ will be a generalisation of the map
$\UU\to V$, where $\UU\to \Gr(r, V)$ is the universal subbundle (see Example~\ref{ex:baby}). Thus, for the
construction we need analogues of the Grassmannian and vector bundles: instead of being
elements of a fixed vector space $\tspace$, our families of tensors will lie on a vector bundle
$\VV_1\otimes \ldots \otimes \VV_d$ on a variety $X$. For example, $\mathbb{P}\tspace$
comes with a family of tensors $\Tproj\colon \OO_{\PP \tspace}(-1) \to \OO_{\mathbb{P}\tspace}\otimes \tspace$
and $\tspace$ comes with a ``trivial'' family $\Taf\colon \OO_{\tspace}\to \OO_{\tspace}\otimes \tspace$.

The construction of $\csigma_r$ is just a special case of a more general and versatile construction of the unrestriction
schemes $\unres{}$, introduced in Definition~\ref{ref:unrestriction:def} below. The
unrestriction scheme depends on a parameter $m$, which we set to $m := r$ here.

Fix a  variety $X$ with a family of tensors $T$, for example $\mathbb{P}\tspace$ with $\Tproj$ above.
The \emph{unrestriction scheme}
$\unrespar{1}{T}\xrightarrow{\pi_1} X$ comes with a family of tensors $\Tppar{1}\in \UU_1 \otimes V_2\otimes \ldots \otimes V_d$,
where $\UU_1$ is a vector bundle of rank $r$ and $\Tppar{1}$ is $\UU_1$-concise. This means that every element of $\Tppar{1}$
is a tensor living in a vector space isomorphic to $\kk^r \otimes V_2\otimes
\ldots \otimes V_d$ and concise with respect to first coordinate. The unrestriction scheme is appropriately universal with
respect to this property, see Definition~\ref{ref:unrestriction:def} for details.

The crucial next step is that we can iterate this construction.
Namely, we take $\unrespar{1, 2}{\Tppar{1}}\to \unrespar{1}{T}$, which comes with a family of tensors $\Tppar{1,2}$, 
concise on the first \emph{two} coordinates.
In this way, we finally arrive at
\[
    \unresfull := \unrespar{1, \ldots ,d}{\Tppar{1,\dots,d-1}} \xrightarrow{\pi} X.
\]
This is a smooth projective variety
(Corollary~\ref{ref:unrestrictionSmoothness:cor}) with a family of tensors
$\Tppar{}$. The map $\pi$ is birational when $\dim V_{\bullet}\geq r$.
For $T = \Tproj$, the concise secant variety $\csigma_r\subseteq \unresfullpar{\Tproj}$ is the closed locus where $\Tppar{}$ has minimal border rank.
The map $\pi\colon \unresfullpar{\Tproj}\to \mathbb{P}\tspace$ restricts to a map
$\rho\colon \csigma_r\to \sigma_r$, which is birational in good situations, see~\S\ref{ssec:identifiability}.
Analogously, taking $T = \Taf$, we obtain that the minimal border rank locus $\csigmahat_r\subseteq \unresfullpar{\Taf}$ has a map to the cone $\sigmahat_r \into \tspace$ over the secant variety $\sigma_r$. We call $\csigmahat_r$ the \emph{arithmetic concise secant variety}. The map $\csigmahat_r\to \sigmahat_r$ is birational in good situations, see~\S\ref{ssec:identifiability}.
We give the details of in Section~\S\ref{sec:unrestrictionSchemes}.

\subsubsection{Framed concise secant variety and singularities}

To describe the points of $\csigma_r$, it is easier to first describe a related variety, the
framed concise secant $\csigmafr_r$, which is constructed in~\S\ref{ssec:framedUnrestrictions}.
The map $\csigmafr_r\to \csigma_r$ is a principal $\GL_r^{\times d}$-bundle, hence the local properties of
$\csigmafr_r$ and $\csigma_r$ agree, but $\csigmafr_r$ is not projective (while $\csigma_r$ is).

Fix \emph{$r$-dimensional} vector spaces $W_1$, \ldots , $W_d$. The $\kk$-points of the framed concise secant are
\begin{align}\label{eq:framedPoints:intro}
    \csigmafr_r(\kk) =
    \left\{ 
        \begin{array}{c | c}
            \Tpparfr{}\in W_1\otimes_{\kk} \ldots \otimes_{\kk} W_d &\Tpparfr{}\mathrm{\ concise,} \brk(\Tpparfr{}) = r\\
            {[\varphi_1]}\in \mathbb{P}\Hom(W_1, V_1)
                &\varphi_1, \ldots , \varphi_d \mbox{ linear maps }\\
              (\varphi_i\in \Hom(W_i, V_i))_{2\leq i\leq d}  & \mbox{satisfying conciseness conditions } \eqref{eq:starCondition}
        \end{array}
    \right\}.
\end{align}
Crucially, the conciseness conditions~\eqref{eq:starCondition} are open, hence there are no equations coming from them (see Proposition~\ref{ref:framedUnrestrictionPoints:prop} for details). The group $\GL(W_1)\times \ldots \times \GL(W_d)$ acts on $\csigmafr_r$
in a natural way. Thanks to~\eqref{eq:starCondition}, this is a free action. The quotient is $\csigma_r$. On the level of points, we obtain
\begin{equation}\label{eq:concisePoints:intro}
    \csigma_r(\kk) = \csigmafr_r(\kk) / \GL(W_1)\times \ldots \times \GL(W_d).
\end{equation}
The map $\rho\colon \csigma_r(\kk)\to \sigma_r(\kk)$ sends $(\Tppar{},
\varphi_1, \ldots ,\varphi_d)$ to the restriction $[T] =
[\varphi_{1,\dots,d}(\Tppar{})] \in \PP\tspace$.

\newcommand{\rhofr}{\rho^{\mathrm{fr}}}%
The variety $\csigmafr_r$ has another map to a secant variety. This is the map
$\alpha\colon \csigmafr_r\to W_1\otimes \ldots \otimes W_d$ that sends $(\Tpparfr{}, \varphi_1, \dots, \varphi_d)$ to $\Tpparfr{}$.
The map $\alpha$ is smooth, with image equal to the concise locus of $\sigma_r\subseteq W_1\otimes \ldots \otimes W_d$.

\subsection{Applications of concise secant varieties}

\subsubsection{Resolutions of singularities}\label{ssec:singularities}

We pass to analysing the singularities of secant varieties. The two main goals here are
to understand the singular locus and resolve singularities.
When the variety is embedded positively enough, there are many positive results~\cite{Ullery, Ein_Niu_Park, Park, Agostini_Park}.

In our setup, without positivity, the situation is more complicated. We are not aware of results about resolution of singularities.
For describing the singular locus, partial results appear in~\cite{Han__singularities, Furukawa_Han}.
A full answer is known only for $r=2$ and $r=d=3$. The paper~\cite{michalek_oeding_zwiernik} describes $\sigma_2$ for every $d$.
For $d=3$, Buczy{\'n}ski and Landsberg~\cite{Buczynski_Landbserg_sigma3} give normal forms
of tensors in $\sigma_3$ and equations are known here, from which one can
describe the singular locus.

In contrast, our method works  \emph{for every} $d\geq 3$ and yields \emph{resolution} of singularities in the Segre case
for $r=2,3,4$, and, conjecturally, for other values of $r$. Importantly, we do not perform any
case-by-case analysis, but rather recast the problem as a question of singularities of Hilbert and Quot schemes,
so that the results follow from known ones. The order $d\geq 3$ of tensors is not important throughout the argument.
\begin{theorem}[Theorem~\ref{ref:singularities:thm}]\label{ref:resolutions:introthm}
    \begin{enumerate}
        \item The variety $\csigma_r$ is nonsingular for $r\leq 3$.
        \item The variety $\csigma_4$ is singular, with
            singular locus of codimension $9$. It is Cohen-Macaulay with
            rational (in fact, terminal) singularities.
            The singular locus is smoothly equivalent to the vertex of the affine cone over
            $\Gr(2, 6)\into \mathbb{P}^{14}$. Its singularities are resolved by a single blow up of
            the singular locus.
        \item Assume that $\rho$ is birational (this happens when $\dim V_{\bullet} = r$ and, in fact, in much
            greater generality, see~\S\ref{ssec:identifiability}). For $r=2,3$ the map $\rho\colon \csigma_r\to \sigma_r$
            is a resolution of singularities of $\sigma_r$. For $r=4$ the map $\Bl_{\Sing(\csigma_4)} \csigma_4 \to \csigma_4\to \sigma_4$
            is a resolution of singularities.
    \end{enumerate}
\end{theorem}
We have a similar description also for $\csigma_5$, which we will discuss elsewhere,
see Remark~\ref{ref:singularitiesFive:rmk}.

Also for general $r$, the geometry of $\csigma_r$ is better than the geometry of
$\sigma_r$. The secant variety suffers  because $\sigma_{r-1}$ forms a ``black hole'' inside:
at any point  $[T]\in \sigma_{r-1}\subseteq \sigma_r$, the tangent space $T_{[T]}\sigma_r$ 
is the whole tangent space to $\mathbb{P}\tspace$ and there is little hope for a meaningful geometry.
In contrast, for any $r$, the geometry of $\csigma_r$ at \emph{any} point is smoothly equivalent to the geometry
of $\sigma_r\subseteq \mathbb{P}\tconcspace$ at a concise point, where $\tconcspace = W_1\otimes \ldots \otimes W_d$
and $\dim W_{\bullet} = r$, see~\S\ref{ssec:smoothEquivalences}.

\subsubsection{Birational geometry of concise secants for $r=4$}\label{ssec:Salmon}

    \newcommand{\tildecsigma}{\widetilde{\csigma}}
    \newcommand{\TT}{\mathbb{T}}
    \newcommand{\nrm}{\mathrm{nrm}}
    In this section we discuss the finer birational geometry of $\csigma_4$ and
    $\tildecsigma_4 = \Bl_{\Sing(\csigma_4)} \csigma_4$. We concentrate on
    the case $d = 3$, $\dim V_{\bullet} = 4$.
    This is the setup for the famous Salmon problem~\cite{Allman_Rhodes}, which asks for the ideal of equations of $\sigma_4\subseteq \mathbb{P}(V_1\otimes V_2\otimes V_3)$.
    The set-theoretic solution is given by Friedland~\cite{Friedland,
    Friedland_Gross, Bates_Oeding}, but the scheme- or ideal-theoretic description
    is still missing, see~\cite{Landsberg_Manivel, Bates_Oeding}.
    The set theoretic solution is based on a reduction to $\dim V_{\bullet} = (3,3,4)$ and
    cannot be upgraded to the ideal-theoretic description. A direct computation of the ideal
    of $\sigma_4$ is (and will remain) computationally infeasible.
    The progress on the problem seems stalled.


    

    Our framework gives a completely different strategy to deal with the problem, using equivariant topology.
    We assume that $\kk$ has characteristic zero.
    Recall the resolution from Theorem~\ref{ref:resolutions:introthm}
    \[
        \begin{tikzcd}
            \tildecsigma_4 \ar[r, "\xi"] &\csigma_4 \ar[r, "\rho"] & \sigma_4.
        \end{tikzcd}
    \]
    The varieties $\tildecsigma_4, \csigma_4$ and $\sigma_4$ admit natural actions
    of $\GL(V_{\bullet}) := \GL(V_1)\times \GL(V_2) \times \GL(V_3)$ and the maps $\xi, \rho$
    are equivariant. 
    Choose a maximal torus $\Gmult^{12}\simeq \TT\into \GL(V_{\bullet})$.
    We explicitly describe the fixed point on $\tildecsigma_4$ and compute their tangent weights. 
    The fixed points are isolated, so we can use equivariant localisation~\cite{Fulton_Anderson} 
    to compute $-K_{\tildecsigma_4}$. Using the description of $\xi$ from~Theorem~\ref{ref:resolutions:introthm}, 
    we calculate $-K_{\csigma_4}$.
    This divisor is revealed to be big and nef, allowing us to use Kawamata-Viehweg vanishing,
    in particular to prove that $\sigma_4$ is arithmetically Gorenstein up to normalisation. 
    This is made precise in the following result.
    \begin{theorem}\label{ref:introLocalisation:thm}
        The $\TT$-fixed points of $\tildecsigma_4$ and $\csigma_4$ are isolated; there are $588224$ fixed points on $\tildecsigma_4$ and $444864$ on
        $\csigma_4$. Hence, the $K_0(\operatorname{Var})$ class of
        $\tildecsigma_4$ is an explicit polynomial in $\Lmot = [\mathbb{A}^1]$,
        given in Example~\ref{ex:motiveSigma4}, in particular the Picard group
        of $\tildecsigma_4$ is $\mathbb{Z}^7$.
        The
        anticanonical bundle of $\tildecsigma_4$ is $\xi^*\rho^*\OO(16) \otimes
        \OO_{\xi}(5)$, while the anticanonical bundle of $\csigma_4$ is $\rho^*\OO(16)$.

        The normalisation $\sigma_4^{\nrm}$ of $\sigma_4$ has rational singularities, it is
        arithmetically Gorenstein with anticanonical sheaf $\OO_{\sigma_4^{\nrm}}(16)$, and $\OO_{\sigma_4^{\nrm}}(v)$ has no higher
        cohomology for $v\geq 0$.
    \end{theorem}

    No finer
    results on the birational geometry of $\sigma_4$ were known before, as there were no tools to obtain them.
    The proof of Theorem~\ref{ref:introLocalisation:thm} is outlined in~\S\ref{ssec:SalmonResults}.
    We have also computed the Hilbert polynomial of
    $\sigma_4^{\nrm}$ using the equivariant Hirzebruch-Riemann-Roch, in particular we have that the degree of $\sigma_4$ is $401 \cdot 701\cdot 576 = 161914176$.
    To count the contributions of $588224$ points and compute relevant data, such as Todd classes, we use \emph{Macaulay2}.
    However, no serious algorithms are employed in the argument.
    We believe that our method may yield the full solution to
    the Salmon problem, by an analysis of syzygies.
    We omit any further details here due to page limit constraints.

\subsubsection{Resolutions of Quot schemes of few points on surfaces}

    \newcommand{\EE}{\mathcal{E}}
    \newcommand{\MM}{\mathcal{M}}
    \newcommand{\Qtw}{\widetilde{Q}}
    Tensors generalise multiplication maps of algebras or modules and, correspondingly, $\csigma_r$ generalises
    the moduli spaces of finite schemes or sheaves of degree $r$. Here we give an example application of the idea of Theorem~\ref{ref:resolutions:introthm}.

    Let $X$ be a smooth quasi-projective surface and $\EE$ be a vector bundle on $X$. The Quot
    scheme $\Quot_r(\EE)$ of $r$ points on $\EE$ parameterises degree $r$ finite quotients of $\EE$.
    For $\rk\EE\geq 2$ it is singular. Artin-Hochster conjecture that the singularities are Cohen-Macaulay,
    and this was refined to asking whether it has rational singularities~\cite{Stark}. Both questions remain open.

    Perhaps surprisingly, the resolution of singularities of $\Quot_r(\EE)$ is known only for $r\leq 2$, see~\cite{Stark}, but the construction does
    not generalise to higher $r$~\cite[bottom of p.2]{Stark2}.
    Here, we give the resolution for $r=3,4$.
    In fact, we generalise from the case of surfaces and allow $X$ to be an arbitrary smooth quasi-projective variety with a vector bundle $\EE$.
    For $\dim X \leq 3$, $r\leq 4$, the scheme $\Quot_r(\EE)$ is irreducible, but this is false for $\dim X \geq 4$, $r=4$~\cite{jelisiejew_sivic}, so
    we will work with its main component $\Quot^{\sm}_r(\EE)$.
    Let $Q_r := \Quot_r^{\sm}(\EE)$ and let
    $\OO_{Q_r} \otimes_{\kk} \EE \to \MM$ be the universal quotient. The $\OO_X$-module structure
    on $\MM$ yields a map $\tau\colon \OO_{Q_r} \otimes_{\kk} \OO_X \to \End_{\OO_{Q_r\times X}}(\MM)$.
    Consider the space $\Qtw_r\to Q_r$ whose $\kk$-points are pairs $(q, R)$, where $q\in Q_r(\kk)$, $R\subseteq \End_{\OO_X}(\MM|_q)$,
    where $R$ is a commutative degree $r$ subalgebra that contains $\tau|_q(\OO_q \otimes_{\kk} \OO_X)\simeq \tau_q(\OO_X)$.
    Observe that $\MM|_q$ is an $R$-module. The space $\Qtw_r$ comes with a natural scheme structure: it represents
    the functor that generalises the above description of $\kk$-points.

    \begin{theorem}\label{ref:introQuotResolution:thm}
        The map $\Qtw_r\to Q_r$ is birational.
        For $r\leq 3$ the space $\Qtw_r$ is smooth, so it resolves the singularities of $Q_r$. The space $\Qtw_4$ is singular along the locus
        of $(R, \MM|_q)$ such that $R\simeq \kk[x_1, x_2, x_3]/(x_1, x_2, x_3)^2$ and the $R$-module $\MM|_q$ is cyclic or cocyclic.
        Blowing up the singular locus of $\Qtw_4$ yields a smooth space.
    \end{theorem}
    The proof is given in~\S\ref{ssec:singularitiesMain}.

\subsubsection{Algebraic characterisation of the cactus rank}

    The following observation is well-known and sometimes called the
    Apolarity Lemma. Recall that a tensor $T \in \tspace$ is a
    \emph{restriction} of a tensor $T'\in
    \tconcspace$ if there are linear maps $\varphi_i\colon W_i\to V_i$ such
    that the induced map $\tconcspace\to \tspace$ maps $T'$ to
    $T$.
    \begin{lemma}[Apolarity Lemma]\label{ref:apolarityLemma:lem}
        Let $T\in \tspace$ be a tensor. The following are equivalent:
        \begin{enumerate}
            \item\label{it:rankOne} the rank of $T$ is at most $r$,
            \item\label{it:rankSpan} there are $r$ points $x_1, \ldots , x_r\in \Seg$ such that
                $T$ lies in $\spann{x_1, \ldots ,x_r}$,
            \item\label{it:rankAlg} the tensor $T$ is a restriction of the unit tensor
                $\spann{r}$.
        \end{enumerate}
    \end{lemma}
    One can generalise~\ref{it:rankSpan}, replacing the tuple $\left\{ x_1,
    \ldots ,x_r \right\}$ by any smoothable finite scheme $Z\into \Seg$ of degree $r$.
    This gives rise to the notion of \emph{smoothable rank} $\smrk(T)$ of $T$.
    It is known that $\smrk(T) \geq \brk(T)$. Equality occurs for example for
    Generalised Additive Decompositions~\cite{Bernardi_Oneto_Taufer__On_GADs} but not in
    general~\cite{bucz_bucz_smoothable_rank_example}. One can go further and drop the smoothability
    requirement for $Z$ above, which leads to the notion of \emph{cactus rank}.
    We generalise Lemma~\ref{ref:apolarityLemma:lem} to this setup.

    \begin{theorem}[Cactus Apolarity Lemma, {Corollary~\ref{ref:spanCharPopular:cor}}]\label{ref:restriction:introthm}
        Let $T\in \tspace$ be a tensor. The following are equivalent:
        \begin{enumerate}
            \item\label{it:cactusrankOne} the cactus rank of $T$ is at most $r$,
            \item\label{it:cactusrankSpan} there is a finite  scheme $Z =
                \Spec(\cA)\to \Seg$
                with $\deg Z = r$ such that
                $T$ lies in $\spann{Z}$,
            \item\label{it:cactusrankAlg} the tensor $T$ is a regular restriction of the
                multiplication tensor in a Gorenstein algebra $\cA'$
                with $\dim_{\kk} \cA' \leq r$.
        \end{enumerate}
        Having $\cA$ as in~\ref{it:cactusrankSpan}, we can choose $\cA'$ to be a
        quotient algebra of $\cA$. In particular, if $\dim_{\kk} V_{\bullet} = r$ and $T$ is concise, then $T$ is
        isomorphic to the multiplication tensor in $\cA$ itself and $\cA$ is Gorenstein.
    \end{theorem}

    A restriction given by $\varphi_i\colon \cA^{\vee}\to V_i$, $i=1, \ldots , d$ is called \emph{regular},
    if the image of each map $\varphi_i^{\vee}\colon V_i^{\vee}\to \cA$ contains an invertible element of $\cA$, see Definition~\ref{ref:regularAndJointlySpanning:def}.
    \begin{example}
        Let $\cA = \Spec \kk[\varepsilon]/(\varepsilon^3)$. Its multiplication tensor is isomorphic to the symmetric tensor
        $x^2z + xy^2$. A restriction of such a tensor is regular if (on every coordinate) it sends $x$ to a nonzero element.
        The algebra $\cA$ also admits Gorenstein quotients $\cA/(\varepsilon^2)$ and $\cA/(\varepsilon)$ with multiplication
        tensors $x^2y$, $x^3$, respectively. More on such tensors in given in~\cite{landsberg_teitler, CFJ, FOS}.
    \end{example}

    Theorem~\ref{ref:restriction:introthm} follows from the more refined
    Theorem~\ref{ref:spanChar:thm}. Partially symmetric versions of this theorem also follow, but
    will be presented elsewhere. An analogous version for smoothable rank exists, see Remark~\ref{ref:smoothableApolarityLemma:rem}.

    \subsubsection{Abstract secant varieties}\label{ssec:abstractSecantsIntro}

The \emph{abstract secant variety}
\[
    \Asigma_r\subseteq \Sym^{r} \Seg \times \mathbb{P}\tspace
\]
can be defined as the closure of the locus
$\Asigma_r^{\op}$ of all pairs
$(\left\{ x_1, \ldots ,x_{r} \right\}, [T])$, where $x_1, \ldots ,x_r\in \Seg$
are linearly independent and $[T]$ lies in the span $\spann{x_1, \ldots
,x_r}\subseteq \mathbb{P}\tspace$.
The projection onto the second coordinate yields a map $\Asigma_r\to \sigma_r$.

The abstract secant variety is absolutely central in the study of identifiability of
tensors and defectivity of secant varieties, see~\cite{Chiantini_Ciliberto__Weakly_defective, Abo_Ottaviani_Peterson_segre, Abo_Brambilla, Casarotti_Mella} and much more references in~\S\ref{ssec:identifiability}. In these
applications, one takes a general point $(Z, [T])$ of $\Asigma_r^{\op}$ and proves that
the forgetful map behaves ``well-enough'' near $(Z, [T])$. For this, it is
very useful to \emph{degenerate} $(Z, [T])$ to have better control over its
geometry, for example for induction. A typical example is to put points of
$Z$ on a hyperplane. Sometimes, points are allowed to collide~\cite{Galuppi_Oneto}.
However, overall, the boundary of $\Asigma_r$ has no clear description, no formula
for a tangent space at a given point and needs to be avoided.

It would be useful to have even more flexibility in degeneration. This is not
possible for the usual abstract secant variety $\Asigma_r$, since we have no
control over the closure. So one could wonder about other varieties
$\mathcal{M}$ which fill into the diagram
\[
    \begin{tikzcd}
        \Asigma_r^{\op} \ar[r, hook, "\mathrm{op}"]\ar[rd] & \mathcal{M}\ar[d]\\
        & \sigma_r
    \end{tikzcd}
\]
We show that the concise secant variety $\csigma_r$ is one such variety. In fact, one can even enlarge
the locus $\Asigma_r^{\op}$ naturally. Consider
\[
    \AsigmaHilb_r = \left\{
        \begin{array}{c l}
            (Z, T)\ | &Z\into \Seg,\ Z \mbox{ zero-dimensional smoothable Gorenstein, }\deg Z =
        r\\
                   &\pr_{2, \ldots ,d}\spann{Z} \simeq \mathbb{P}^{r-1},\ T\in \spann{Z}\mbox{ not in the span of any subscheme of }Z
        \end{array}
    \right\},
\]
where $\pr_{2, \ldots ,d}\spann{Z}$ denotes the span of $Z$ inside $\mathbb{P}V_{2, \ldots ,d}$.
\begin{theorem}[Theorem~\ref{ref:abstractInsideConcise:thm}]\label{refintro:abstractInsideConcise:thm}
    The variety $\AsigmaHilb_r$ embeds as an open subvariety of $\csigma_r$. The map $\rho\colon \csigma_r\to \sigma_r$
    extends the natural projection map $\AsigmaHilb_r\to \sigma_r$.
\end{theorem}
The advantage of $\csigma_r$ over $\Asigma_r$ is that every point of $\csigma_r$ corresponds to a tensor
and the tangent space at such a point can be described under very mild conditions, using a generalisation
of smooth equivalences from~\S\ref{ssec:smoothEquivalences}; this will be explained
in a future work.

\subsubsection{Varieties of sums of powers}\label{ssec:VSP}

Theorem~\ref{refintro:abstractInsideConcise:thm} has immediate consequences for the varieties of sums of powers.
Recall that these varieties were introduced by Ranestad and Schreyer~\cite{ranestad_schreyer_VSP} as closures of
\begin{equation}\label{eq:VSP}
    \VSP(T, r) = \overline{\left\{ [Z]\into \Seg\ |\ Z\mbox{ zero-dimensional smoothable Gorenstein, }\deg Z = r,\ T\in \spann{Z}\right\}}.
\end{equation}
These varieties exhibit very interesting global geometry, especially in the defective
cases, see for example~\cite{Ranestad_Iliev__VSPcubic, Ranestad_Voisin__VSP,
Ranestad_Schreyer__polar_simplices,
Jelisiejew_Ranestad_Schreyer__polar_simplices}, yielding hyperk{\"a}hler varieties etc.
However, as in the case of secant varieties, the need for closure is a source of major technical problems,
as it is very hard to prove that the closure is smooth, for example, or to describe it, see the discussion in~\cite{Jelisiejew_Ranestad_Schreyer__polar_simplices}.
The need of closure is also the same as existence of wild tensors, as in~\cite{bucz_bucz_smoothable_rank_example}.

One solution for this problem was recently introduced using border apolarity~\cite{Buczyska_Buczynski__border}. There, one considers $\VSPbord(T, r)$
which, roughly speaking, consists of all ``good'' multigraded ideals $I$ with fixed
Hilbert function and such that $T$ is apolar $I$.
The word ``good'' above technically means that $[I]$ lies in the \emph{Slip
component} of a multigraded Hilbert scheme of
Haiman-Sturmfels~\cite{Haiman_Sturmfels__multigraded}.
Border apolarity has led to dramatic improvements in border rank estimates, see~\cite{Conner_Harper_Landsberg, Conner_Huang_Landsberg, Mandziuk_Ventura}.

Working with $\VSPbord(T, r)$ poses two challenges:
\begin{enumerate}
    \item For a given ideal $I$, it is notoriously hard to determine if it lies in the Slip component.
        This is because the multigraded Hilbert scheme, apart from containing the Slip component, contains
        quite a number of ``slop'' ideals, easy to generate but containing
        little geometric information. Therefore, it is notoriously hard to
        describe this Hilbert scheme, see~\cite{Jelisiejew_Mandziuk}.
        In broad terms, the same happens for the usual Hilbert scheme, but the pathologies of the multigraded one show up for smaller $r$ and
        seem much worse.
    \item Even when computing the variety $\VSPbord(T, r)$ is fully possible, it contains rational parts, coming
        from ``artefacts'' of the high-degree parts of ideals $I$, see for
        example~\cite[Theorem~6.6, Remark~7.10]{Huang_Michalek_Ventura}. As a result, the
        global geometry of $\VSPbord(T, r)$ is not the expected one: the $\VSPbord(T, r)$
        viewed as a compactification, adds too much.
\end{enumerate}

The concise secant variety allows for an alternative compactification of the set $(T, Z)$, $T\in \spann{Z}$.
Theorem~\ref{refintro:abstractInsideConcise:thm} suggests the following definition.
\begin{definition}\label{ref:conciseVSP:def}
    The \emph{concise variety of sums of powers} $\cVSP(T,r)$ is fibre $\rho^{-1}(T)$, where $\rho\colon \csigma_r\to \sigma_r$
    is the natural map.
\end{definition}
Theorem~\ref{refintro:abstractInsideConcise:thm} implies that $\cVSP(T, r)$ contains all interesting $Z$ such that
$T\in \spann{Z}$, up to mild conditions on the span. The variety $\cVSP(T, r)$ is also much smaller than $\VSPbord(T, r)$
in general.
\begin{example}
    If $T$ is a concise tensor of minimal border rank $r$, then $\cVSP(T, r)$ is a point, regardless of whether
    $T$ is tame or wild. In particular, in both examples~\cite[Theorem~6.6, Remark~7.10]{Huang_Michalek_Ventura},
    the variety $\VSPbord(T, r)$ has high dimension ($r+2$, $\geq 8$, respectively) while $\cVSP(T, r)$ is a point.
\end{example}

\subsubsection{Concise version of border apolarity}\label{ssec:borderApolarity}

    As discussed above, border apolarity has been a successful tool for obtaining lower bounds on the border rank.
    In this section we explain that the same can be achieved with $\csigma_r$. To compare the two,
    we can say that:
    \begin{enumerate}
        \item\label{it:tmpBorder} Border apolarity produces~\cite[Theorems~1.2,
            4.3]{Buczyska_Buczynski__border} a Borel-fixed ideal $I$ inside the
            apolar ideal of the tensor $T$. The main
            question is whether $I$ is the Slip component. The obstacles for proving that are discussed in~\S\ref{ssec:VSP} above.
        \item\label{it:tmpConcise} The procedure below produces a \emph{Borel-fixed} unrestriction $\Tppar{}\geq T$. The unrestriction
            is a concise tensor in $\tconcspace$, where $\dim W_{\bullet} = r$, so that $\brk \Tppar{}\geq r$. The main question is whether $\Tppar{}$ has
            border rank $r$.
    \end{enumerate}
    In our view, there are two advantages of~\ref{it:tmpConcise} over~\ref{it:tmpBorder}. First, the procedure below requires no apolarity, no multigraded Hilbert schemes etc.
    Second, more importantly, the output $\Tppar{}$ is a tensor and can be tested using all existing tools of the theory:
    Koszul and Young flattenings~\cite{Landsberg_Ottaviani}, centroids~\cite{JLP}, etc.
    
    Let us finally describe the procedure more carefully. Let $[T]\in \mathbb{P}\tspace$ and let $\rho\colon \csigma_r\to \sigma_r$. First, almost directly from the definitions, we have the following:
    \begin{enumerate}
        \item $\brk T\leq r$ if and only if $\rho^{-1}([T])$ is nonempty,
        \item the cactus rank of $T$ is $\leq r$ if and only if $\rho^{-1}([T])$ contains a $1$-generic tensor (see~\S\ref{ssec:genericityConditions}),
        \item the smoothable rank of $T$ is $\leq r$ if and only if $\rho^{-1}([T])$ contains a $1$-generic tensor with smoothable centroid,
        \item the rank of $T$ is $\leq r$ if and only if $\rho^{-1}([T])$ contains a $1$-generic tensor whose centroid is a reduced algebra.
    \end{enumerate}
    The power of border apolarity comes from the \emph{fixed ideal theorem (fit)}~\cite[Theorem~4.3]{Buczyska_Buczynski__border},
    which allows one to reduce greatly the number of cases, to look for Borel-fixed ideals.
    The following is the analogue of this theorem.
    \begin{theorem}[Fixed unrestriction theorem]\label{ref:cfit:thm}
        Let $[T]\in \mathbb{P}\tspace$ be a tensor, let $G\subseteq \SLV$ be the stabilizer of $[T]$ and let $B\subseteq G$ be a Borel subgroup.
        Then $\brk T\leq r$ if and only if there exists a $B$-fixed tensor $\Tppar{}$ restricting to $T$.
    \end{theorem}
    To understand Borel-fixed tensors, we need to carefully see how the action of $B$ lifts to an action on $\tconcspace$.
    This is discussed formally in~\S\ref{ssec:groupActions}.

    As a proof of concept, we give the following example.
    \begin{example}
        Let $T$ be the $3\times3$-matrix multiplication tensor. Using border apolarity and extensive computations~\cite{Conner_Harper_Landsberg}
        are able to show that $\brk T\geq 17$. We can prove that their computation can be reinterpreted in our framework, so that Theorem~\ref{ref:cfit:thm}
        also yields $\brk T \geq 17$.
    \end{example}
    Finally, let us remark that the above ``concise version of border apolarity'' does not really use any apolarity. This is
    not necessarily good or bad, but we want to stress it to avoid confusion.

\subsection{Open problems}

Miraculously, concise secant varieties seem to apply to a range of tensor problems of various flavours.
We sketch three possible applications. None of them is currently worked out, so below we mostly illustrate
a possible approach.

\subsubsection{Lower bounds for border ranks, including matrix multiplication}

    This problem is a straightforward continuation of the discussion in~\S\ref{ssec:borderApolarity}. It is also quite mature:
    it is clear what to do, it only leaves the question on how effective it will be.

    Namely, consider the matrix multiplication tensor $M_n\in V_1 \otimes V_2 \otimes V_3$, where $\dim V_{\bullet} = n^2$.
    It is stabilised by the Borel subgroup $B$ of upper-triangular $n\times n$ matrices.

    To prove, for example, that $\brk M_3 > 17$, one can enumerate possible $B$-fixed concise tensors $\Tppar{}$
    in $\rho^{-1}(M_{3})$ with $\dim \Cen{\Tppar{}}\geq 17$, to prove that none of them has minimal border rank tensor.
    The code provided by Austin Conner for the paper~\cite{Conner_Harper_Landsberg} can be adapted to this purpose.
    Then, one has to rule out the possible candidates, showing that none of them has minimal border rank.
    This part is currently unexplored, but it is tempting to employ the suite of already existing tools: so-called
    $(2,1,0)$-tests of~\cite{Conner_Harper_Landsberg}, the more classical Koszul flattenings~\cite{Oeding_Ottaviani} and nonlinear theory of~\cite{Dolevalek_Michalek}.

\subsubsection{Non-defectivity of Segre}

    It is conjectured since~\cite{Abo_Ottaviani_Peterson_segre} that the Segre varieties are nondefective except for
    unbalanced cases and a small number of exceptions, see more discussion in~\S\ref{ssec:identifiability}. The conjecture was verified
    for an enormous number of cases, but it seems open in general.

    Theorem~\ref{ref:abstractInsideConcise:thm} allows us to rephrase nondefectivity as generic finiteness of $\rho\colon \csigma_r\to \sigma_r$.
    This allows for more aggressive degenerations: it would be enough to prove that $\rho^{-1}([T])$ is finite for a single tensor $T$, which
    may be quite degenerate. Such degenerations techniques are a powerful tool, perhaps their most impressive use case is the proof of generic
    Green's conjecture~\cite{Green__Voisin, Green__Weyman}.
    The question here would be: can we find explicit tensors $T$ with zero-dimensional $\rho^{-1}([T])$? How degenerate can they be?

Throughout, we work over an arbitrary field $\kk$.

%
        
        

\subsection{Acknowledgements}

    We thank Jarek Buczy{\'n}ski for very helpful comments on an earlier draft and for the permission to discuss results
    related to his contributions to unrestrictions. We thank Giorgio Ottaviani for questions
    regarding the abstract secant variety and Pierpaola Santarsiero for discussions about identifiability. We thank
    Mat{\v{e}}j Dole{\v{z}}{\'a}lek and Joseph Landsberg for alerting us about multiple typos and Maciej Ga{\l}{\k{a}}zka for observing an issue
    in an earlier version.

\section{Unrestrictions}\label{sec:unrestrictions}
\newcommand{\pstspace}{S^{\nu_1}V_1 \otimes \ldots \otimes S^{\nu_d}V_d}

In this section we prove the results about unrestrictions announced in~\S\ref{ref:unrestrictionsIntro}. 
We only consider the Segre case for space reasons. The results for the Segre-Veronese case are available at~\cite{JJ4arxiv}.

Throughout this work, the symbol ${\bullet}$ denotes an arbitrary index in the range $1,2, \ldots ,d$.

Recall that a tensor $T \in \tspace$ is \emph{concise} if it does not lie in any proper subspace 
of the form $V_1' \otimes \ldots \otimes V_d' \subsetneq \tspace$. The border rank of 
a concise tensor is at least the maximum of dimensions of all $V_\bullet$. In particular,
when all $V_\bullet$ have dimension $m$, each concise tensor has border rank at least $m$. Concise 
tensors of border rank $m$ are called \emph{minimal border rank} tensors. They are better 
understood than general, non-concise border rank $m$ tensors, see the introduction and~\cite{JLP}.

\begin{definition}\label{ref:unrestrictionNaive:def}
    Let $T, \Tppar{}\in \tspace$ be tensors. 
    The tensor $\Tppar{}$ is an \emph{unrestriction} of $T$ if $T$ is a \emph{restriction} of $\Tppar{}$, that is there exist linear maps 
    $\varphi_i \colon V_i \to V_i$
    such that $(\varphi_1 \otimes \dots \otimes \varphi_d)(\Tppar{}) = T$:
    \[
    \begin{tikzcd}
        \Tppar{} \arrow[r, phantom, "\in"] \arrow[d, phantom, "\rotatebox{90}{$\leq$}"]  & V_1 
        \arrow[r, phantom, "\otimes"] \arrow[d, "\varphi_1"] & \dots 
        \arrow[r, phantom, "\otimes"] & V_d \arrow[d, "\varphi_d"]\\ 
        T \arrow[r, phantom, "\in"] & V_1 \arrow[r, phantom, "\otimes"] & 
        \dots \arrow[r, phantom, "\otimes"] & V_d
    \end{tikzcd}
    \]
\end{definition}
\begin{remark}
    In Definition~\ref{ref:unrestrictionNaive:def} and this whole section for simplicity we assume that $\varphi_i$ have same source and target.
    Starting from~\S\ref{sec:unrestrictionSchemes}, we will distinguish these: the source will in general be called $W_i$, while the target $V_i$.
\end{remark}
\begin{example}\label{ex:wedge}
    Assume that $\kk$ has characteristic different from $2, 3$.
    Consider the antisymmetric tensor 
    \[
        T = e_1 \wedge e_2 \wedge e_3 \in \Lambda^3 \kk^3 \subset \kk^3 \otimes \kk^3 \otimes \kk^3
        \into \kk^5 \otimes \kk^5 \otimes \kk^5.
    \]
    The tensor $T$ has both rank and border rank $5$~\cite{Ilten_Teitler, Derksen_Makam}.
    The rank yields a tautological unrestriction $T \leq \spann{5}$ from the unit tensor. There is 
    another minimal border rank unrestriction 
    \[
        T = 
        \begin{bmatrix}
            0 & x_3 & -x_2 & 0 & 0 \\
            -x_3 & 0 & x_1 & 0 & 0 \\
            x_2 & -x_1 & 0 & 0 & 0 \\
            0 & 0 & 0 & 0 & 0 \\
            0 & 0 & 0 & 0 & 0 \\
        \end{bmatrix}
        \leq 
        \begin{bmatrix}
            x_5 & x_3 + x_4 & -x_2 & x_1 & x_2 \\
            -x_3 + x_4 & 0 & x_1 & 0 & 0 \\
            x_2 & -x_1 & 0 & 0 & 0 \\
            x_1 & 0 & 0 & 0 & 0 \\
            x_2 & 0 & 0 & 0 & 0 \\
        \end{bmatrix}
        = \Tppar{}
    \] 
    from the tensor $\Tppar{}$ isomorphic to the tensor $T_{\OO_{54}}$ from the classification
    from~\cite[Theorem 1.7]{JLP}.
\end{example}

%
%
%
%
%
%
To state our results, we discuss the intuitive notion of 
a degeneration of tensors. There are a few closely related ways to formally define it. 
Let us list a few.
\begin{enumerate}
    \item For $\kk = \CC$, a degeneration of tensors $T_t$ to a tensor $T$ can be thought of as a 
        sequence of tensors in $\tspace$ converging to $T$ in the usual topological sense. 
        In this interpretation, the general tensor $T_{t \neq 0}$ is general tensor from this sequence.
    \item For $\kk = \CC$, a degeneration of tensors $T_t$ to a tensor $T$ can be thought of as a 
        curve germ $\Delta \to \tspace$, where $\Delta\subseteq \mathbb{C}$ is a small disk around $0$, such that $0\in \Delta$
        is mapped to $T$. In this interpretation, the general tensor $T_{t \neq 0}$ 
        is the image of general point of $\Delta\setminus\{0\}$.
    \item More algebraically and for any field $\kk$, a degeneration of tensors $T_t$ to a tensor $T$ can be thought of as a 
        curve germ $\Spec \kk[\![t]\!] \to \tspace$ such that the unique $\kk$-point of 
        $\Spec \kk[\![t]\!]$ is mapped to $T$. In this interpretation, the general tensor $T_{t \neq 0}$ 
        is the image of the generic point of $\Spec \kk[\![t]\!]$.
    \item For general $\kk$, a degeneration of tensors $T_t$ to a tensor $T$ can be thought of as an
        single tensor $T_t\in (\tspace)[\![t]\!]$ such that $T_{t = 0}\in \tspace$ 
        coincides with the tensor $T$. In this interpretation, the general tensor $T_{t \neq 0}$
        is the tensor $T_t$ viewed as an element of the $\kk(\!(t)\!)$-vector space $(\tspace)(\!(t)\!)$.
\end{enumerate} 
In our proofs, we formally use the last of these notions. However, the arguments easily adjust 
to other notions of a degeneration of tensors, in particular the ones listed above, so the reader is welcome to think about a degeneration 
of tensors according to his own taste. 
\smallskip

We say that one can \emph{algorithmically construct} an unrestriction $\Tppar{}$ using a degeneration of 
tensors $T_t$ if one can write an algorithm which takes a degeneration $T_t$ as an input and returns a 
corresponding unrestriction $\Tppar{}$ as an output. 
The proofs of results from this section have this feature.

\begin{theorem}\label{p:algorithm_segre}
    Let $T \in \tspace$ be a tensor. Suppose we have a degeneration of tensors
    $T_t \in \tspace[\![t]\!]$ such that $T_{t = 0} = T$ and the general tensor
    $T_{t \neq 0}$ is concise.
    \begin{enumerate}
        \item Using $T_t$, one can algorithmically construct a concise unrestriction 
        $\Tppar{} \in \tspace$ of $T$.
        \item If the general tensor $T_{t \neq 0}$ has border rank $\leq m$, 
        then $\Tppar{}$ has border rank $\leq m$ as well.
    \end{enumerate}
\end{theorem}

Using the geometric perspective, we will deduce that different choices of bases and minors in the 
algorithmic construction from the proof of Theorem~\ref{p:algorithm_segre} lead to isomorphic 
unrestrictions.
\begin{proposition}[Proposition~\ref{prop:uniqueness_of_unres}]
    Let $T \in \sigmahat_m \subset \tspace$ be a tensor of border rank $\leq m$.
    Let $T_t$ be a degeneration of tensors of border rank $\leq m$ to $T$ with 
    general tensor $T_{t \neq 0}$ concise. 
    Consider any two unrestrictions 
    \[(\Tppar{}, \varphi_1, \dots, \varphi_d)\quad\text{and}\quad (\Tppar{'}, \varphi_1', \dots, \varphi_d')\] 
    constructed from $T_t$ using the algorithmic from the proof 
    of Theorem~\ref{p:algorithm_segre} for different choices of bases and minors.
    Then there exists linear isomorphisms $\psi_i \colon V_i \to V_i$ such that 
    \[
        \Tppar{'} = \psi_{1,\dots,d}(\Tppar{}) \quad\text{and}\quad \varphi_i' = \varphi_i \circ \psi_i^{-1} \quad\text{for every} \quad i=1,\dots,d.
    \]
\end{proposition}
%

The construction is iterative. The following lemma describes the general step.
\begin{lemma}\label{l:algorithm_step}
    Let $T_t \in \tspace[\![t]\!]$ be a degeneration of tensors to a tensor $T$ such that the general tensor
    $T_{t \neq 0}$ is $V_i$-concise. 
    One can algorithmically construct
    \begin{itemize}
        \item a degeneration of linear maps $(\varphi_i)_t \colon V_i[\![t]\!] \to V_i[\![t]\!]$ to a linear map $\varphi_i$ and
        \item a degeneration of tensors $\Tppar{i}_t \in \tspace[\![t]\!]$ to a tensor $\Tppar{i}$
    \end{itemize} 
    that satisfy the following conditions:
    \begin{enumerate}
        \item the tensor $\Tppar{i}$ is $V_i$-concise,
        \item the map $\varphi_i$ yields a restriction $\varphi_i(\Tppar{i}) = T$,
        \item the general map $(\varphi_i)_{t \neq 0}$ is a linear isomorphism
            and $(\varphi_i)_{t \neq 0}(\Tppar{i}_{t \neq 0}) = T_{t \neq 0}$.
    \end{enumerate}
\end{lemma}
\begin{proof}
    Fix coordinates on the vector space $\tspace$. Let $M_t$ be the $(m \times n)$-matrix representing the flattening
    $T_t \colon \Vhat^\vee[\![t]\!] \to V_i[\![t]\!]$. The general tensor 
    $T_{t \neq 0}$ is $V_i$-concise, so $m \leq n$ and  $M_t$ has a non-zero maximal minor. 
    Let $(X_i)_t$ be a $(m \times m)$-submatrix of $M_t$ which determinant has minimal valuation 
    among all non-zero minors of $M_t$. Cramer's rule from linear algebra says that entries 
    of the matrix $(X_i)_{t \neq 0}^{-1} \cdot M_{t \neq 0}$ are maximal minors of $M_{t \neq 0}$ 
    divided by $\det (X_i)_{t \neq 0}$, so they have positive valuation.
    Let us view $(X_i)_{t \neq 0}^{-1} \cdot M_{t \neq 0}$ as a matrix $(M_i)_t$ with coefficients in $\kk[\![t]\!]$.
    The submatrix of $(M_i)_t$ indexed by the same columns as the submatrix $(X_i)_t$ of $M_t$ is the identity matrix,
    so $(M_i)_{t = 0}$ has full rank. Moreover, we have $(X_i)_t \cdot (M_i)_t = M_t$,
    so in particular $(X_i)_{t = 0} \cdot (M_i)_{t = 0} = M_{t = 0}$.
    Finally, take $(\varphi_i)_t$ to be the degeneration of maps corresponding to $(X_i)_t$
    and $\Tp_t$ to be the degeneration of tensors corresponding $(M_i)_t$.
\end{proof}
\begin{proof}[Proof of Theorem~\ref{p:algorithm_segre}]
    First, take the degeneration $T_t$. The general tensor $T_{t \neq 0}$ is 
    concise, so we can apply the algorithm from Lemma~\ref{l:algorithm_step} on $V_1$ to construct 
    a degeneration of tensors $\Tppar{1}_t$ and a degeneration of maps $(\varphi_1)_t$.
    The tensor $\Tppar{1}$ is $V_1$-concise and the map $\varphi_1$ gives a restriction $\Tppar{1} \geq T$.
    Then, take the degeneration $\Tppar{1}_t$ and apply the construction from Lemma~\ref{l:algorithm_step}
    on $V_2$ to get degenerations $\Tppar{1,2}_t$ and $(\varphi_2)_t$ yielding a $V_2$-concise tensor $\Tppar{1,2}$ 
    restricting to $T_1$ via $\varphi_2$, and so on. 
    After $d$ steps, we get tensors $\Tppar{1}, \Tppar{1,2} \dots, \Tppar{1,\dots,d}\in \tspace$ and maps 
    $\varphi_1 \in \End(V_1), \dots, \varphi_d \in \End(V_d)$ such that each $\Tppar{1,\dots,i}$ is $V_i$-concise
    and $\varphi_i(\Tppar{1,\dots,i}) = \Tppar{1,\dots,i-1}$. 
    The flattening ranks can only drop under restriction, so the $i$-th tensor $\Tppar{1,\dots,i}$ is in fact 
    $V_1, \dots, V_i$-concise. In particular, the tensor $\Tppar{} := \Tppar{1,\dots,d}$ is concise
    and the maps $\varphi_1, \dots, \varphi_d$ give an unrestriction $T \leq \Tppar{}$.

    Now assume that the general tensor $T_{t \neq 0}$ has border rank $\leq m$.
    The general map $(\varphi_{1, \dots, d})_{t \neq 0}$ yields an isomorphism between 
    $\Tppar{}_{t \neq 0}$ and $T_{t \neq 0}$, so the general tensor $\Tppar{}_{t \neq 0}$
    and its limit $\Tppar{}$ have border rank $\leq m$ as well. 
\end{proof}

\begin{example}[Bini's tensor]
    Let $T$ be the tensor corresponding to multiplication of $2 \times 2$-matrices
    with the bottom-right entry of the first matrix being zero, see~\cite[Subsection 2.1.4]{JM_complexity}.
    This tensor has border rank equal to 5, which is witnessed by the following degeneration of unit tensors:
    \[
    \begin{aligned}
        T_t = &(t a_1 + a_3) \otimes b_2 \otimes (c_2 + t c_4)
        +(t a_1 + a_2) \otimes (b_1 + t b_3) \otimes c_1
        +(a_2 + a_3) \otimes (b_1 + t b_4) \otimes (c_2 + t c_3)\\
        +&(-a_3 + t^2 a_4) \otimes (b_1 + b_2 + t b_4 + t^2 b_5) \otimes c_2
        +(-a_2 + t^2 a_5) \otimes b_1 \otimes (c_1 + c_2 + t c_3 + t^2 c_5).
    \end{aligned}
    \]
    The algorithm applied to $T_t$ produces the degeneration of tensors
    \[
        \Tppar{}_t :=
        \begin{bmatrix}
        x_1 & x_2 & x_3 & 0 & x_5 \\
        0 & -tx_2 + x_3 & 0 & 0 & 0 \\
        0 & x_4 & 0 & -tx_2 + x_3 + tx_4 & 0 \\
        x_5 & 0 & 0 & 0 & tx_5 \\
        0 & -tx_1 + tx_2 + x_5 & -t^2x_1 + tx_3 + tx_5 & 0 & 0
        \end{bmatrix}
    \]
    whose general tensor $\Tppar{}_{t \neq 0}$ is isomorphic to $T_{t \neq 0}$ via the maps
    \[
        \begin{bmatrix}
        1 & 0 & 0 & 1 & t \\
        0 & 0 & 0 & 0 & 1 \\
        0 & 0 & 1 & 0 & 0 \\
        0 & t & 0 & t & 0 \\
        t & t & t^2 & 0 & 0
        \end{bmatrix},\quad
        \begin{bmatrix}
        1 & 0 & 0 & 0 & 0 \\
        0 & 0 & 1 & 0 & 0 \\
        0 & 0 & 0 & 1 & 0 \\
        0 & 0 & 0 & 0 & 1 \\
        0 & -t & 0 & 0 & 0
        \end{bmatrix},\quad
        \begin{bmatrix}
        1 & 0 & 0 & 0 & 0 \\
        0 & 1 & 0 & 0 & 0 \\
        0 & 0 & 1 & 0 & 0 \\
        0 & 0 & 0 & 1 & 0 \\
        t^2 & 0 & 0 & 0 & -t
        \end{bmatrix}.
    \]
    To get the corresponding minimal border rank unrestriction $\Tppar{}$, we simply set $t = 0$.
    In a suitable basis, this unrestriction can be represented as
    \[
        T = \begin{bmatrix}
        x_1 & 0 & x_3 & 0 & 0 \\
        0 & x_1 & 0 & x_3 & 0 \\
        x_2 & 0 & 0 & 0 & 0 \\
        0 & x_2 & 0 & 0 & 0 \\
        0 & 0 & 0 & 0 & 0
        \end{bmatrix}
        \leq 
        \begin{bmatrix}
        x_1 & x_5 & x_3 & 0 & x_2 \\
        0 & x_1+x_4 & 0 & x_3 & 0 \\
        x_2 & 0 & 0 & 0 & 0 \\
        0 & x_2 & 0 & 0 & 0 \\
        0 & x_3 & 0 & 0 & 0
        \end{bmatrix}
        = \Tppar{}.
    \]
    The tensor $\Tppar{}$ is the $1$-degenerate tensor $T_{\mathcal{O}_{56}}$ from the classification from~\cite[Theorem 1.7]{JLP}.
\end{example}

In the algorithmic proof of Theorem~\ref{p:algorithm_segre}, we had to choose an order of coordinates $V_1, \dots, V_d$. 
The following example shows that different choices can lead to non-isomorphic unrestrictions.
\begin{example}\label{ex:order_matters}
    Let $d = 3$ and $\dim_\kk V_\bullet = 2$. Consider the following degeneration of concise tensors, 
    represented in the matrix form as
  \[
    \begin{bmatrix}
    x_1 & t^2 x_2 \\
    t x_2 & t x_1 \\
  \end{bmatrix}.
  \]
  Applying the algorithm first to columns and then to variables yields the constant degeneration
  \[
  \begin{bmatrix}
    x_1 & t^2 x_2 \\
    t x_2 & t x_1 \\
  \end{bmatrix}
  \leadsto 
  \begin{bmatrix}
    x_1 & t x_2 \\
    t x_2 & x_1 \\
  \end{bmatrix}
  \leadsto 
  \begin{bmatrix}
    x_1 & x_2 \\
    x_2 & x_1 \\
  \end{bmatrix}
  \xrightarrow{t \to 0} \begin{bmatrix}
    x_1 & x_2 \\
    x_2 & x_1 \\
  \end{bmatrix}.
  \]
  Starting from rows, we obtain another degeneration
  \[
  \begin{bmatrix}
    x_1 & t^2 x_2 \\
    tx_2 & tx_1 \\
  \end{bmatrix}
  \leadsto
  \begin{bmatrix}
    x_1 & t^2 x_2 \\
    x_2 & x_1 \\
  \end{bmatrix}
  \xrightarrow{t \to 0} \begin{bmatrix}
    x_1 & 0 \\
    x_2 & x_1 \\
  \end{bmatrix}.
  \]
  The first minimal border rank unrestriction is the multiplication tensor in the algebra
  $\kk \times \kk$, written in the basis $(1,1), (1,-1)$. The second one
  comes from multiplication in $\kk[x]/(x^2)$. In particular, they are not isomorphic. 
\end{example}

\section{Unrestriction schemes}\label{sec:unrestrictionSchemes}

In this section we construct the unrestriction schemes. The unrestrictions
from Section~\ref{sec:unrestrictions} yield curves on these varieties.

We use the following conventions. On a $\kk$-vector space $V$ we put an affine space structure by
$V = \Spec\Sym V^{\vee}$. More generally, we identify a locally free sheaf
$\VV$ on a scheme $X$ with the associated vector bundle $\Spec_X \Sym
\VV^{\vee}$.

To cleanly state the result, we need a notion of a family of tensors parameterised by a variety $X$.
In the literature, such a family is sometimes defined to be a morphism $X\to V_1\otimes \ldots \otimes V_d$. However, this setup
is rather restrictive: in this way there is \emph{no} universal family over
$\mathbb{P}(V_1\otimes \ldots \otimes V_d)$, see
Example~\ref{ex:universalFamily} below. We take a different route by using
vector bundles.
We introduce the notions below.
The following dictionary may help to quickly grasp them:
\[
    \begin{array}{c c c}
        & \mathrm{tensor} & \mathrm{family\ of\ tensors}\\
        \toprule
        & T\in V_1\otimes \ldots \otimes V_d & T\colon \LL \to \VV_1 \otimes \ldots \otimes \VV_d\\
        \mathrm{concise} & V_1^{\vee}\to V_2\otimes \ldots \otimes V_d \mbox{ injective} & \VV_1^{\vee} \otimes \LL \into \VV_2\otimes \ldots \otimes \VV_d\mbox{ subbbundle}\\
        \mathrm{concise} & V_2^{\vee}\otimes \ldots \otimes V_d^{\vee}\to V_1 \mbox{ surjective} &  \VV_2^{\vee}\otimes \ldots \otimes \VV_d^{\vee}\otimes \LL \onto \VV_1 \mbox{ surjective}
    \end{array}
\]

A map $\VV'\into \VV$ of vector bundles is a \emph{subbundle} if it is injective and $\VV/\VV'$ is also a vector bundle.
This also has a much more intuitive equivalent definition: it is not only injective, it stays injective after restriction to a point.
We make this explicit by the following lemma.
\begin{lemma}\label{ref:subbudle:lem}
    Let $\VV'\to \VV$ be a map of vector bundles (not a priori assumed injective) on a scheme $X$. The following are equivalent:
    \begin{enumerate}
        \item\label{it:subbundleAll} for every morphism $f\colon Y\to X$, the pullback $f^*\VV'\to f^*\VV$ is injective,
        \item\label{it:subbundlePoint} for every $x\in X$, the induced map $\VV'|_x\to \VV|_x$ is injective,
        \item\label{it:subbundle} $\VV'\to \VV$ is a subbundle (in particular, it is injective).
    \end{enumerate}
\end{lemma}
\begin{proof}
    \def\mm{\mathfrak{m}}%
    The implication $\ref{it:subbundleAll}\implies\ref{it:subbundlePoint}$ is formal.
    The implication $\ref{it:subbundle}\implies\ref{it:subbundleAll}$ follows from the fact that pullback of
    sequences of vector bundles stay exact.
    Assume $\ref{it:subbundlePoint}$ holds. The claim~\ref{it:subbundle} is local on $X$, so we may assume that $X = \Spec(A)$
    for a local ring $(A, \mm)$ and that $\VV$, $\VV'$ come from $A$-modules $V$, $V'$. By assumption, the map $V/\mm V\to V'/\mm V'$
    is injective. Pick elements $v_1, \ldots ,v_k$ of $V$ and $v_{k+1}, \ldots ,v_l$ of $V'$ such that the classes of $v_1, \ldots ,v_k$ form
    a basis of the vector space $V/\mm V$ and $v_1, \ldots , v_l$ form a basis of the vector space $V'/\mm V'$. This implies that $V\to V'$
    is injective. Consider the free submodule $V'' \subseteq V'$ generated by
    $v_{k+1}, \ldots ,v_l$. One verifies that the map $V' \to V'/V$ restricts
    to an isomorphism $V''\to V'/V$, hence $V'/V$ is a free $A$-module.
\end{proof}

Below, when working with vector bundles over $X$, we write $\otimes$ instead of $\otimes_{\OO_X}$. Similarly, when working with
vector spaces, we write $\otimes$ instead of $\otimes_{\kk}$. When $V$ is a vector space, we denote by $\OO_{X}\otimes V$ the trivial
vector bundle on $X$ with fibre $V$.
For vector spaces $V_1, \ldots ,V_d$, we write $\tspace := V_1 \otimes \ldots \otimes V_d$.
For vector bundles $\VV_1, \ldots , \VV_d$, we write $\tspacecurly := \VV_1 \otimes \ldots \otimes \VV_d$.
When $\tspacecurly$ is clear form the context, we write $\VVhat$ for $\VV_1\otimes  \ldots \otimes\VV_{i-1} \otimes \VV_{i+1}\otimes \ldots \otimes \VV_d$.

\begin{definition}\label{ref:tensor:def}
    Let $X$ be a scheme.
    A \emph{family of tensors over $X$} or a \emph{tensor on $X$} is a map
    $T\colon \LL\to \tspacecurly$, where $\LL$ is a line bundle. For a map
    $f\colon Y\to X$, we have an induced family over $Y$ given by $f^*T\colon
    f^*\LL \to f^*\tspacecurly$. When $Y = \{x\}$ is a point of $X$, we use the
    notation $T|_x$ instead of $f^*T$.
The map $T\colon \LL\to \tspacecurly$ is a subbundle $\LL\into \tspacecurly$ if and only if $T|_x\neq 0$ for every $x\in X$.
\end{definition}

\begin{example}\label{ex:universalFamily}
    Let $V_1, \ldots ,V_d$ be vector spaces.
    \begin{enumerate}
        \item
            There is a tensor over $\tspace$
            \[
                \Taf\colon \OO_{\tspace} \to \OO_{\tspace} \otimes \tspace
            \]
            such that for every $\kk$-point $x$ of $\tspace$ the tensor $\Taf|_x\in \tspace$ is the 
            tensor corresponding to $x\in \tspace$.
            In this case $\OO_{\tspace}$ is not a subbundle, since $\Taf|_0 = 0$.
        \item There is a tensor over $\mathbb{P}(\tspace)$ given by the subbundle
            \begin{equation}\label{eq:univTensorProj}
                \Tproj\colon \OO_{\mathbb{P}(\tspace)}(-1)\into \OO_{\mathbb{P}(\tspace)}\otimes \tspace.
            \end{equation}
            Over a  point a $\kk$-point $[T] \in \mathbb{P}(\tspace)$, the induced map on fibers is the inclusion map
            \[
                \Tproj|_{[T]}\colon \OO_{\mathbb{P}(\tspace)}(-1)|_{[T]} = \kk\cdot T \hookrightarrow \tspace = 
                (\OO_{\mathbb{P}(\tspace)} \otimes \tspace)|_{[T]}.
            \]
            Observe that in this case it is not possible to use only trivial vector bundles.
        \item Let $X \into \mathbb{P}(\tspace)$ be a closed subscheme. The restriction of~\eqref{eq:univTensorProj} to $X$
            yields a tensor $\OO_{X}(-1) \into \OO_{X}\otimes \tspace$. We will use this observation for $X$ being a
            secant or a cactus variety.
        \item Let $V$ be a vector space and let $\UU^\vee$ be the universal subbundle on $\Gr(m, V)$. 
            There is a tensor over $\Gr(m,V)$
            \[
                \TGr \colon \OO_{\Gr(m, V)} \to \UU \otimes V
            \]
            corresponding to the subbundle inclusion $\UU^\vee \into \OO_{\Gr(m, V)} \otimes V$.
            This tensor is $\UU$-concise in the sense of Definition~\ref{ref:concise:def}.
    \end{enumerate}
\end{example}

\begin{remark}\label{ref:twistingNotation:rem}
    A tensor of the form $T\colon \OO_X \to \tspacecurly$ is the same as an
    element $T\in H^0(\tspacecurly)$. In the latter form, families of tensors
    appear everywhere throughout the literature, for example in works related
    to secant varieties of any variety or in the toric setup~\cite{Galazka, Buczyska_Buczynski__border}.
    A tensor of the form $T\colon \LL \to \tspacecurly$ can be twisted by $\LL$ to obtain an equivalent
    tensor $T\colon \OO_X \to (\LL^{\vee}\otimes \VV_1)\otimes \VV_{2 \ldots
    d}$ and in what follows we sometimes use both of these equivalent forms.
\end{remark}

\begin{remark}\label{ref:localFamilies:rem}
    Let $T\colon \LL\to \tspacecurly$ be a family of tensors and $x\in X$ be a point.
    There is an open neighbourhood $x\in U \subseteq X$ such that $\LL|_U$, $(\VV_1)|_U$,
     \ldots, $(\VV_d)|_{U}$ are all trivial bundles. The tensor $T|_U$ is isomorphic to a tensor
     of the form $\OO_U \to \OO_U \otimes \tspace$, hence to a map $f\colon U\to \tspace$. In particular,
     $T|_U = f^*\Taf$, where $\Taf$ is the universal tensor on $\tspace$. Informally speaking,
     locally on $X$ the notion of a family agrees with the more common notion of $X\to \tspace$.
\end{remark}

\begin{definition}\label{ref:concise:def}
A family of tensors $T\colon \LL\to \tspacecurly$ is \emph{$\VV_i$-concise} if
the contraction
\[
    T_{\VV_i}\colon  \VV_i^{\vee}\otimes\LL \to \VVhat
\]
is a vector subbundle. In this case $\LL\to \tspacecurly$ is also a subbundle.
The family $T$ is \emph{concise} if it is concise with respect to every
coordinate.
\end{definition}

\begin{definition}\label{def:concise}
    A family of tensors $T \colon \LL \to \tspacecurly$ 
    is a \emph{restriction} of another family of tensors $\Tppar{} \colon \LL \to \WW_{1,\dots,d}$
    if there exist $\OO_X$-linear maps $\varphi_i \colon \WW_i \to \VV_i$ such 
    that $(\varphi_1 \otimes \dots \otimes \varphi_d)(\Tppar{}) = T$.
\end{definition}

\begin{remark}\label{ref:concise:rem}
    A family of tensors $T$ over $X$ is $\VV_i$-concise if and only if for every point $x\in X$, the induced
contraction ${T_{\VV_i}}|_x$ is injective. In particular, if $T$ restricts to a $\VV_i$-concise tensor
(in the sense of Definition~\ref{def:concise}), then
$T$ itself is $\VV_i$-concise. 
\end{remark}


\subsection{Unrestriction schemes}

    In this section we define unrestriction schemes. Special cases of these
    serve as smooth ``ambient'' varieties in which concise secant and cactus
    varieties live.

    We need a piece of notation. For a map $\phi\colon \WW_1\to \WW_2$ of vector bundles on $X$, the \emph{zero locus
    of $\phi$} is the closed subscheme $\Z(\phi)$ of $X$ which is locally given by entries of $\phi$, that is,
    for any open $U\subseteq X$ with isomorphisms $(\WW_1)|_U \simeq \OO_U^{\rk
    \WW_1}$, $(\WW_{2})|_U\simeq \OO_U^{\rk \WW_2}$, the ideal of $\Z(\phi)$ is given by the entries of the matrix
    $\phi \in \OO_{U}^{\rk \WW_1 \cdot \rk \WW_2}$.

The following is our main object of interest. Again, we fix an integer $m$, which can be arbitrary, but in applications it will usually
be equal to the border rank.
\begin{definition}[Unrestriction scheme]\label{ref:unrestriction:def}
    Take a family of tensors $T\colon \LL\to \tspacecurly$. Fix an index $i$.
    Take the bundle of Grassmannians $p \colon \Gr_X(m, \VVhat) \to X$ and
    let $\UU_i^{\vee} \into p^*\VVhat$ denote the universal subbundle (the dual is introduced here to keep the notation consistent).
    The \emph{$i$-th unrestriction scheme} of $T$ is the closed subscheme 
    \[
        \begin{tikzcd}
            \unres \ar[rd]\ar[rr, hook, "\mathrm{cl}"] &&\Gr_X(m, \VVhat)\ar[ld, "p"]\\
            & X
        \end{tikzcd}
    \]
    defined as the zero locus of the following map of vector bundles
    \begin{equation}\label{eq:unrestrictionDef}
        p^*(\VV_i^\vee \otimes \LL) \xrightarrow{p^*T_{\VV_i}} p^*\VVhat \to \frac{p^*\VVhat}{\UU_i^{\vee}}.
    \end{equation}
\end{definition}
The unrestriction scheme comes with a $\UU_i$-concise family of tensors
\[
 \Tppar{i}\colon \OO_{\unres} \into p^*\VVhat \otimes \UU_i.
\]
By equation~\eqref{eq:unrestrictionDef}, the contraction
$p^*T_{\VV_i}\colon p^*(\VV_i^{\vee} \otimes \LL)\to p^*\VVhat$ factors through $\UU_i^{\vee}\into p^*\VVhat$,
so $p^*T_{\VV_i}$ induces a uniquely determined map $p^*(\VV_i^{\vee} \otimes \LL) \to \UU_i^{\vee}$. Its dual map
$\varphi_i\colon \UU_i\to p^*(\VV_i \otimes \LL^{\vee})$ shows that $p^*T$ is a restriction of $\Tppar{i}$; here we twist by $\LL$ as explained in Remark~\ref{ref:twistingNotation:rem}. This explains
the name $\unres$.
It would be more precise to include $m$ in the notation $\OpUnres_i^m(T)$, yet in our setup $m$ is 
fixed, so we prefer the lighter notation.

\begin{lemma}\label{ref:baseChangeUnres:lem}
    Let $f\colon Y\to X$ be a morphism. Then the pullback $Y\times_X \unres$ is isomorphic to the unrestriction scheme $\unrespar{i}{f^*T}$.
\end{lemma}
\begin{proof}
    Both the bundle of Grassmannians and the condition~\eqref{eq:unrestrictionDef} pull back nicely.
\end{proof}

\begin{example}[Scheme of unrestrictions of a fixed tensor $T$]\label{ex:fibre}
    Let $T\colon \LL\into \tspacecurly$ be a family of tensors.
    Let $x\in X$ be a $\kk$-point and $T|_x\colon \LL|_{x}\into {\VV_1}|_x\otimes \ldots \otimes {\VV_d}|_x$ be the restriction.
    Then the fiber of $\unres \to X$ over $x$ is isomorphic to $\unrespar{i}{T|_x}$.
\end{example}

    Lemma~\ref{ref:baseChangeUnres:lem} yields also the functor of points of $\unres$.
    Namely, for a $\kk$-scheme $S$, a morphism $S\to \unres$ is the same as a morphism $f\colon S\to X$,
    together with a subbundle $\UU_i^{\vee}\into f^*\VVhat{}$ such that $f^*T_{\VV_i}\colon f^*(\LL \otimes \VV_i^{\vee})\to
    f^*\VVhat{}$ factors through $\UU_i^{\vee}$.

    In the following, we obtain another characterisation of the functor of points. In particular, applying $S = \kk$,
    we obtain a description of $\kk$-points.
    \begin{proposition}[Points of unrestriction scheme]\label{ref:functorOfUnrestrictions:prop}
        Let $T\colon \LL\into \tspacecurly$ be a family of tensors on $X$.
        For a $\kk$-scheme $S$, a morphism $S\to\unres$ is the same as a morphism $f\colon S\to X$,
        together with a pair $(U_i, \varphi_i)$, where
        \begin{enumerate}
            \item $U_i$ is a subbundle $U_i^{\vee} \into f^*\VVhat{}$, so it yields an $U_i$-concise tensor
                $\Tppar{i}\colon \OO_S \into U_i \otimes f^*\VVhat{}$,
            \item $\varphi_i$ is a map of vector bundles $\varphi_i\colon U_i \to f^*(\LL^{\vee} \otimes \VV_i)$,
        \end{enumerate}
        such that $\varphi_i(\Tppar{i})$ is equal to $f^*T$ as global sections of $f^*(\LL^{\vee} \otimes \tspacecurly)$.
        These data can be visualised as
        \[
            \begin{tikzcd}
                \Tppar{i} \arrow[r, phantom, "\colon"] \arrow[d, phantom, "\leq" rotate = 90] &[-1cm]
                \OO_S \arrow[d, equal] \arrow[r] &
                f^*\VVhat{}\ar[d, equal] \ar[r, phantom, "\otimes"] &[-0.7cm] U_i \arrow[d, "\varphi_i"] \\
                f^*T \arrow[r, phantom, "\colon"] &
                \OO_S \arrow[r] & f^*\VVhat{} \ar[r, phantom, "\otimes"] & f^*(\VV_i \otimes \LL^\vee).
            \end{tikzcd}
        \]
    \end{proposition}
    \begin{proof}
        From a map $S\to \unres$, we obtain the subbundle $U_i$ by pulling back $\UU_i$ and the map $\varphi_i$ by pulling back
        the map $\varphi_i$ from Definition~\ref{ref:unrestriction:def}. Conversely, having $U_i$, we obtain a morphism
        \[
            \begin{tikzcd}
                S \ar[rr, "\Tppar{i}"]\ar[rd, "f"] && \Gr_X(m, \VVhat{})\ar[ld, "p"]\\
                                      & X
            \end{tikzcd}
        \]
        and $\varphi_i$ shows that it factors through the zero locus defining $\unres{}$.
    \end{proof}

The construction of unrestrictions can be iterated: from a family of tensors $T$ on $X$,
we obtained a family $\Tppar{1}$ on $\unrespar{1}{T}\xrightarrow{\pi_1} X$ concise on the first coordinate.
We can iteratively consider the unrestriction scheme 
\[
    \unrespar{2}{\Tppar{1}} \xrightarrow{\pi_2} \unrespar{1}{T}
\]
together with the family of tensors $\Tppar{1,2}$. This family is concise on the second coordinate by definition,
but it is also concise on the first coordinate, because its restriction is $\Tppar{1}$, see Remark~\ref{ref:concise:rem}.
\begin{definition}
    The $i$-th \emph{iterated unrestriction scheme} of $T$ is the scheme 
    \[
        \unrespar{1, \dots, i}{T} := \unrespar{i}{\Tppar{1, \dots, i-1}}.
    \]
    It comes with a family $\Tppar{1, \ldots , i}$ of tensors which is concise on coordinates $1, \ldots ,i$.
\end{definition}
The iterated unrestriction schemes fit into the following sequence
\[
    \unrespar{1, \dots, i}{T} \xrightarrow{\pi_i} \unrespar{1, \dots, i-1}{T} \xrightarrow{\pi_{i-1}}
    \dots \xrightarrow{\pi_3} \unrespar{1, 2}{T} \xrightarrow{\pi_2} \unrespar{1}{T}
    \xrightarrow{\pi_1} X.
\] 
In particular, we have the \emph{(full) unrestriction scheme} $\unresfullpar{T} := \unrespar{1, \ldots ,d}{T}$ with a concise tensor $\Tppar{}$.
The points of this scheme can be described by iterating Proposition~\ref{ref:functorOfUnrestrictions:prop}.

\subsection{Concise and secant varieties}\label{ssec:conciseSecants}

Having the notion of unrestrictions, we can finally define the objects of main interest: concise secant and cactus varieties.
Consider $m$-dimensional vector spaces $W_1, \ldots ,W_d$ and let $T\in \tconcspace$.
Recall from~\cite{JLP} that the tensor $T$ has \emph{minimal border rank} if $T$ is concise and $\brk T = m$.
\begin{definition}\label{ref:secantLocus:def}    
    Let $T\colon \LL\into \tconcspacecurly$ be a concise (Definition~\ref{ref:concise:def}) tensor on $X$, where every $\WW_i$ has rank $m$.
    \begin{enumerate}
        \item The \emph{minimal border rank locus} is the largest reduced closed
            subscheme $X^{\brk=m} \subseteq X$ such that for every $x\in
            X^{\brk=m}$, the tensor $T|_x$ has border rank $m$. Observe that $\brk = m$ is the
            minimal possible value given that $T$ is concise.
        \item The \emph{minimal border cactus rank locus} is the largest closed
            subscheme $X^{\bcrk=m} \subseteq X$ such that for every $x\in
            X^{\bcrk=m}$, the tensor $T|_x$ has cactus border rank $m$. Buczy{\'nski}-Keneshlou~\cite{Buczynski_Keneshlou}
            give a canonical scheme structure on this locus.
    \item
        Consider the following map of vector bundles
            \[
                \End(\WW_1)\times \ldots \times\End(\WW_d)\to \prod_{i=1}^{d-1} \WW_1\otimes \ldots \otimes \WW_d
            \]
            given by $(X_1, \ldots ,X_d)\mapsto (X_1\circ T - X_i\circ T)_{i=2,
            \ldots ,d}$. The \emph{centroid abundant locus} is the closed
            subscheme of $X$ given as the locus where this map has kernel of rank $\geq m$.
        \end{enumerate}
    We say that $T$ \emph{has minimal border rank} / \emph{has minimal border cactus rank} / \emph{is centroid-abundant},
    if respective loci above are equal to the whole $X$.
\end{definition}

\begin{definition}[Concise secant and cactus varieties]\label{ref:conciseSecant:def}
    Let $V_1, \ldots ,V_d$ be vector spaces. Fix an integer $m\geq 1$. Recall that $\Tproj$ denotes the universal tensor 
    over $\mathbb{P}(\tspace)$ and $\Taf$ denotes the universal tensor over $\tspace$, see~Example~\ref{ex:universalFamily}.
    \begin{enumerate}
        \item
            \emph{The $m$-th concise secant variety} $\csigma_m$ is the minimal border rank locus
            for the universal tensor on the unrestriction scheme $\pi\colon \unrespar{}{\Tproj}\to \mathbb{P}(\tspace)$.
            We have $\pi(\csigma_m) = \sigma_m$, where $\sigma_m$ is the $m$-th secant 
            variety of $\Seg \into \PP(\tspace)$. The map $\rho\colon \csigma_m \onto \sigma_m$ is the restriction of $\pi$.
        \item
            \emph{The arithmetic $m$-th concise secant variety}
            $\csigmahat_m$ is the minimal border rank locus
            for the universal tensor on the unrestriction scheme $\pihat\colon \unrespar{}{\Taf}\to \tspace$.
            Again, the restriction of $\pihat$ yields a surjective map $\rhohat\colon \csigmahat_m \onto \sigmahat_m$,
            where $\sigmahat_m$ is the cone over $\sigma_m$.
        \item \emph{The $m$-th concise cactus variety} $\ckappa_m$ is the minimal cactus border rank locus
            for the universal tensor on the unrestriction scheme $\unrespar{}{\Tproj}\to \mathbb{P}(\tspace)$.
            The \emph{$m$-th arithmetic cactus variety} $\ckappahat_m$ is defined analogously.
    \end{enumerate}
    
\end{definition}

        \begin{proof}[Proof of Theorem~\ref{ref:cfit:thm}]
            By construction, the morphism $\rho$ is projective and $B$ acts on the projective scheme $\rho^{-1}([T])$.
            By the Borel Fixed Point Theorem, this action admits a fixed point, which is the required $\Tppar{}$.
        \end{proof}

\begin{example}[Baby case: matrices]\label{ex:baby}
For $d=2$, that is, for matrices, the questions of tensor rank
are quite trivial. This makes it an ideal case for illustrating the new elements of $\csigma_m$.
Fix vector spaces $V_1, V_2$ and consider the determinantal variety
\[
    \sigma_m = \left\{ [M]\in \mathbb{P}(V_1 \otimes V_2)\ |\ \rk M\leq m \right\}.
\]
Take the Grassmannian $\Gr(m, V_2)$ of $m$-dimensional
subspaces of $V_2$ and its universal subbundle $\UU_1^\vee \subseteq \Gr(m,
V_2)\times V_2$. We can form a projective subbundle $\mathbb{P}(V_1\otimes\UU_1^\vee)$ inside
$\mathbb{P}(V_1\otimes V_2) \times \Gr(m, V_2)$, which comes with a map 
$\rho\colon \mathbb{P}(V_1 \otimes \UU_1^\vee)\to \mathbb{P}(V_1\otimes V_2)$.

It will turn out that $\csigma_m$ is $\mathbb{P}(V_1 \otimes \UU_1^\vee)$ in our case.
This is a smooth variety; a projective bundle over $\Gr(r, V_2)$.
To understand it better, we will describe its points (or, to be precise, $\kk$-points).
These are pairs. First, there is a point of the Grassmannian $\Gr(m,
V_2)$,
which we denote by $[U_1^{\vee}]\in \Gr(m, V_2)$. It corresponds to an
injection $U_1^{\vee} \into V_2$, so to a rank $m$ tensor $\Tppar{}\in
U_1\otimes V_2$. Second, there is an element of
$\mathbb{P}(V_1 \otimes \UU_1^\vee)|_{[U_1^{\vee}]}$, that is, a class $[\varphi_1]\in
\mathbb{P}(V_1 \otimes U_1^{\vee})$, where $\varphi_1\colon U_1\to V_1$ is a
nonzero linear map.

The map $\rho$ sends $(\Tppar{}, [\varphi_1])$ to the tensor $[\varphi_1(\Tppar{})]\in \mathbb{P}(V_1 \otimes V_2)$.
The image of $\rho$ is $\sigma_m$. If $V_1, V_2$ are at least $m$-dimensional,
then the map $\rho$ is birational and a resolution of singularities of $\sigma_m$. It is much employed 
in the geometry of syzygies, where it is attributed to Kempf~\cite{Kempf__Images_of_vector_bundles}, see~\cite[(6.1.9)]{Weyman__Cohohomology_of_vector_bundles}.
\end{example}

\subsection{Structure of unrestrictions and $\csigma_m$ over $\mathbb{P}(\tspace)$ and $\tspace$}

    In this section we discuss group actions and bundles which appear on unrestrictions starting from
    the ``universal'' tensors over $\mathbb{P}(\tspace)$ and $\tspace$. These structures appears essentially
    because $\tspace$ is a trivial vector bundle.

    Let $\Taf\colon \OO_{\tspace}\to \OO_{\tspace} \otimes \tspace$ be the universal tensor over $\tspace$ and let
    \[
        \unresfullpar{\Taf}\to \ldots \to \unrespar{1}{\Taf} \to \tspace.
    \]
    be its unrestriction schemes. Similarly, let $\Tproj\colon \OO_{\mathbb{P}(\tspace)}(-1) \into
    \OO_{\mathbb{P}(\tspace)}\otimes\tspace$ be the universal tensor of
    $\mathbb{P}(\tspace)$ and let 
    $
    \unresfullpar{\Tproj}\to \ldots \to \unrespar{1}{\Tproj} \to \mathbb{P}(\tspace)
    $ be its unrestriction schemes.

    \subsubsection{Group actions}\label{ssec:groupActions}

        Let $\GLV := \GL(V_1)\times \ldots \times \GL(V_d)$ be the general linear group.
        It acts naturally on $\tspace$ and on $\Vhatpar{1} = V_2 \otimes \dots \otimes V_d$, hence it acts on $\Gr(m, \Vhatpar{1}) \times \tspace$.
        The action restricts to $\unrespar{1}{\Taf}$ and makes $\unrespar{1}{\Taf}\to \tspace$
        into a $\GLV$-equivariant map. The action of $\GLV$ on $\Vhatpar{1}$ gives an action of $\GLV$ 
        on the (trivial) bundle $\OO_{\unrespar{1}{\Taf}}\otimes \Vhatpar{1}$. This action preserves the universal subbundle,
        hence induces a linearisation of the universal subbundle. Therefore we obtain an action on the Grassmannian bundle over
        $\unrespar{1}{\Taf}$ and on its subscheme $\unrespar{2}{\Taf}$. Continuing, we obtain $\GLV$-actions on 
        all unrestriction schemes $\unrespar{i}{\Taf}$ such that the restriction maps
        $\unresfullpar{\Taf}\to \ldots \to \unrespar{1}{\Taf}\to \tspace$ are equivariant.

        The argument can be repeated in the projective setting. There is one difference though. The action of scalar
        matrices $\Gmult^d \subseteq \GLV$ on $\mathbb{P}(\tspace)$ is trivial and so is their action on $\Gr(m, \Vhatpar{1})$. By induction,
        we conclude that $\Gmult^d$ acts trivially. In practice, it is hence reasonable to restrict to the action of $\SLV$ in the projective
        setting.

        One could ask, if there is a bigger connected group acting. The answer
        is expected to be negative in many cases. Indeed,
        if $\rho$ is birational, then a connected group action on $\csigma_m$
        descends to one on the normalisation of $\sigma_m = \rho(\csigma_m)$,
        see~\cite[Proposition~2.25]{Brion}. This is the case when $\dim
        V_{\bullet} = m$, but also in much larger number of cases discussed
        in~\S\ref{ssec:identifiability}.

    \subsubsection{Bundles}
    Recall that $\Vhatpar{1} = V_2 \otimes \ldots \otimes V_d$ and
    consider the Grassmannian $\Gr(m, \Vhatpar{1})$ with its universal tensor $\TGr$,
    see Example~\ref{ex:universalFamily}. Note that the tensor $\TGr$ is concise on the first coordinate.
    Let $\unresfullpar{\TGr} = \unrespar{2, \ldots ,d}{\TGr}\to \Gr(m, \Vhatpar{1})$ be its unrestriction scheme.
    \newcommand{\ehat}{\widehat{e}}

    The scheme $\unrespar{1}{\Taf}$ is closed inside $\tspace \times \Gr(m, \Vhatpar{1})$, 
    so there is a natural
    \[\ehat\colon \unrespar{1}{\Taf} \to \Gr(m, \Vhatpar{1}).
    \] 
    By construction, we have that the universal tensor $\Tppar{1}$ on $\unrespar{1}{\Taf}$ 
    is the pullback $\ehat^* \TGr$, see Definition~\ref{ref:unrestriction:def}.

    \begin{proposition}[Unrestriction scheme as a bundle]\label{ref:bundleAffine:prop}
        The morphism $\ehat$ makes $\unrespar{1}{\Taf}$ into a vector bundle of rank $m^2$ over $\Gr(m, \Vhatpar{1})$, more precisely
        it is the bundle $\UU_1^{\vee}\otimes V_1$, where $\UU_1^{\vee}$ is the universal subbundle.
        The diagram
        \begin{equation}\label{eq:unresBundle}
            \begin{tikzcd}
                \unresfullpar{\Taf} \ar[r]\ar[d] & \unresfullpar{\TGr}\ar[d]\\
                \unrespar{1}{\Taf} \ar[r, "\ehat"] & \Gr(m, \Vhatpar{1})
            \end{tikzcd}
        \end{equation}
        is cartesian, so $\unresfullpar{\Taf}\to \unresfullpar{\TGr}$ is also a vector bundle of rank $m^2$.
    \end{proposition}
    Observe that $\unresfullpar{\TGr}$ is a projective variety, while $\unresfullpar{\Taf}\to \unresfullpar{\TGr}$ is affine, so
    the above precisely measures the lack of projectivity of $\unresfullpar{\Taf}$.

    \begin{proof}
        The map $\varphi_1$ from Definition~\ref{ref:unrestriction:def} yields a map from $\unrespar{1}{\Taf}$ to
        the bundle $\UU_1^{\vee}\otimes V_1$. Conversely, from an element of $\UU_1^{\vee}\otimes V_1$, we obtain a map $\varphi_1\colon \UU_1\to V_1$
        and using the description in Proposition~\ref{ref:functorOfUnrestrictions:prop}, a map $\UU_1^{\vee}\otimes V_1\to \unrespar{1}{\Taf}$.
        This proves the first part.

        The tensor $\Tppar{1}$ on $\unrespar{1}{\Taf}$ is a pullback via $\ehat$ of the tensor $\TGr$ on $\Gr(m, \Vhatpar{1})$.
        By repeatedly using the base change Lemma~\ref{ref:baseChangeUnres:lem}, we obtain that~\eqref{eq:unresBundle} is cartesian.
    \end{proof}
    In the projective case, the variety $\unrespar{1}{\Tproj}$
    comes with a natural map
    \[e\colon \unrespar{1}{\Tproj}\to \Gr(m, \Vhatpar{1}).\]
    \begin{proposition}\label{ref:bundleProjective:prop}
        The morphism $e$ makes $\unrespar{1}{\Tproj}$ into a projectivisation of vector bundle of rank $m^2$ over $\Gr(m, \Vhatpar{1})$, more precisely
        it is the bundle $\mathbb{P}(\UU_1^{\vee}\otimes V_1)$, where $\UU_1^{\vee}$ it the universal subbundle.
        The diagram
        \[
            \begin{tikzcd}
                \unresfullpar{\Tproj} \ar[r]\ar[d] & \unresfullpar{\TGr}\ar[d]\\
                \unrespar{1}{\Tproj} \ar[r, "e"] & \Gr(m, \Vhatpar{1})
            \end{tikzcd}
        \]
        is cartesian, so $\unresfullpar{\Tproj}\to \unresfullpar{\TGr}$ is also a projectivised vector bundle of rank $m^2$.
    \end{proposition}
    \begin{proof}
        The argument mirrors Proposition~\ref{ref:bundleAffine:prop}, so we repeat the argument. The only difference is that we obtain a restriction map
        $\varphi_1\colon \UU_1 \to V_1\otimes \OO(1)$, hence a surjection $\UU_1\otimes V_1^{\vee}\to \OO(1)$, which yields a map to $\mathbb{P}(\UU_1^{\vee}\otimes V_1)$.
    \end{proof}

\subsection{Framed unrestrictions}\label{ssec:framedUnrestrictions}

Sometimes, it may be useful to work with ``usual'' tensors, that is, elements of a fixed vector space $W_1\otimes \ldots \otimes W_d$ rather than 
of a nontrivial vector bundle. This can be achieved naturally using framings. The result is the framed unrestriction scheme defined below.
The price paid is that the framed unrestriction scheme is no longer projective over the base scheme.

Recall that for a scheme $X$ and vector bundles $\VV'$, $\VV''$ of rank $m$ on $X$, there is a
principal $\GL_{m}$-bundle $\Iso(\VV', \VV'') \to X$ such that for any $\phi\colon S\to X$, the morphisms
$S\to \Iso(\VV', \VV'')$ over $X$ are in bijection with isomorphisms $\phi^*\VV'\to \phi^*\VV''$. It is called the \emph{bundle of isomorphisms}.
In the special case when $\VV' = \OO_X\otimes W$, the variety $\Iso(\VV', \VV'')\to X$ is a $\GL(W)$-bundle.
The bundles of isomorphisms can be applied very generally. To stay focused, we apply it only for the full unrestriction scheme.
Let $\GLW := \GL(W_1)\times \ldots \times\GL(W_d)$.

\begin{definition}
    Fix $m$-dimensional $\kk$-vector spaces $W_1$, \ldots ,$W_d$. 
    Let $T\colon \LL\to \VV_1 \otimes \ldots \otimes \VV_d$ be a
    tensor on $X$ and let $\Xppar{} :=\unresfullpar{T}$ be its full
    unrestriction scheme,
    with a tensor $\Tppar{}\colon \OO_{\Xppar{}} \into \UU_1\otimes \ldots \otimes \UU_d$.
    The \emph{framed unrestriction scheme} of $T$ is the principal $\GLW$-bundle
    \[
        \unresfrfull := \prod_{i=1}^d \Iso_{\Xppar{}}(\OO_{\Xppar{}}\otimes W_i, \UU_i) \xrightarrow{\upsilon} \Xppar{}.
    \]
\end{definition}
On the framed unrestriction scheme, we have in particular canonical isomorphisms
\[
    u_i\colon \OO_{\unresfrfull} \otimes W_i \to \upsilon^*\UU_i,
\]
so the family of tensors $\upsilon^*\Tppar{}$ yields a family of tensors
\[
    \Tpparfr{}\colon \OO_{\unresfrfull}\into \OO_{\unresfrfull}\otimes W_1\otimes \ldots \otimes W_d.
\]

Let us now describe the points of $\unresfrfull{}$.
\begin{proposition}\label{ref:framedUnrestrictionPoints:prop}
    Let $S$ be an $\kk$-scheme with a morphism $f\colon S\to X$. Morphisms
    $S\to \unresfrfull$ over $X$ are in bijection with tuples $(\Tpparfr{},
    \varphifr_{1}, \ldots , \varphifr_d)$, where $\Tpparfr{}\colon \OO_S\into \OO_S\otimes W_1\otimes
    \ldots \otimes W_d$ is a tensor, $\varphifr_1\colon \OO_S\otimes W_1 \to f^*(\VV_1\otimes \LL^\vee)$,
    $\varphifr_2\colon \OO_{S}\otimes W_2 \to f^*\VV_2$, \ldots , $\varphifr_{d}\colon \OO_S \otimes W_d\to f^*\VV_d$ are
    maps of bundles. These are required to satisfy the following condition:
    \begin{align}\label{eq:starCondition}
        \Tpparfr{} & \mbox{ is } W_1\text{-}, \ldots , W_{d-2}\text{-}, W_{d-1}\text{-}, W_d\text{-concise}\\\notag
        \varphifr_d(\Tpparfr{}) & \mbox{ is } W_1\text{-},
        \ldots , W_{d-2}\text{-}, W_{d-1}\text{-concise}\\\notag
        \varphifr_{d-1}\circ{}\varphifr_d(\Tpparfr{}) & \mbox{ is } W_1\text{-},
        \ldots , W_{d-2}\text{-concise}\\\notag
         \ldots \\\notag
         \varphifr_2\circ{} \ldots \varphifr_{d-1}\circ{} \varphifr_d(\Tpparfr{}) & \mbox{ is } W_1\text{-concise}\\\notag
         \varphifr_1\circ{} \varphifr_2\circ{} \ldots \varphifr_{d-1}\circ{}\varphifr_d(\Tpparfr{}) &= T.
    \end{align}
\end{proposition}
The conciseness conditions above are redundant: if a restriction of a tensor is, say, $W_1$-concise, then the tensor
itself is $W_1$-concise. Therefore, it is enough to require that $\Tpparfr{}$ is $W_d$-concise, $\varphifr_d(\Tpparfr{})$ is $W_{d-1}$-concise,
\ldots, $\varphifr_2\circ \ldots \circ \varphifr_{d-1}(\Tpparfr{})$ is $W_1$-concise.
\begin{proof}
    A morphism $\phi\colon S\to \unresfrfull$ yields a morphism $f\colon S\to X$ and a morphism $S\to \Xppar{} = \unrespar{1, \ldots ,d}{T}$ over $X$.
    By pulling back respective universal bundles, the latter yields
    \begin{enumerate}
        \item a subbundle $U_1^{\vee}\into f^*(\VV_2 \otimes \ldots \otimes \VV_d)$,
        \item a subbundle $U_2^{\vee}\into U_1\otimes f^*(\VV_3 \otimes \ldots \otimes \VV_d)$,
        \item and so on, all the way to
        \item a subbundle $U_d^{\vee}\into U_1\otimes  \ldots \otimes U_{d-1}$,
            this last one yielding a tensor denoted by $\Tppar{}|_S \colon \OO_S \into \OO_S \otimes U_1 \otimes \ldots \otimes U_d$,
    \end{enumerate}
    together with restriction maps $\varphi_1\colon U_1\to f^*(\VV_1 \otimes
    \LL^{\vee})$ and $\varphi_i\colon U_i\to f^*\VV_i$ for $i=2, \ldots , d$, such that
    $\varphi_1\circ \ldots \circ \varphi_d(\Tppar{}|_S) = T$.
    By the properties of the bundles of isomorphisms, the morphism $\phi\colon S\to \unresfrfull$ additionally yields isomorphisms
    $(u_i)|_S\colon \OO_S\otimes W_i\to U_i$.
    Now, the compositions $\varphi_i \circ (u_i)|_S$ give the maps $\varphifr_i$ from the statement, for $i=1, \ldots ,d$, while
    $((u_1)|_S)^{-1}\circ \ldots ((u_d)|_S)^{-1}\Tppar{}|_S$ yields $\Tpparfr{}$.
    The fact that respective $U_i^{\vee}$ are subbundles becomes equivalent to the required conciseness conditions.
    This gives one direction.

    The argument can be reversed, as follows. Suppose we are given $(\Tpparfr{}, \varphifr_1, \ldots ,\varphifr_d)$.
    Let
    \[
        \Tpparfr{}_{\leq i} := (\varphi_{i+1}\circ{} \ldots \circ{}
        \varphi_d)(\Tpparfr{})\colon \OO_S \into W_1\otimes
        \ldots \otimes W_i\otimes_{\kk} f^*\VV_{i+1}\otimes \ldots \otimes
        f^*\VV_d
    \]
    be the restriction. The tensor $\Tpparfr{}_{\leq 1}$ is $W_1$-concise, hence yields a subbundle
    $\OO_S \otimes W_1^{\vee}\into f^*\VV_{2}\otimes \ldots \otimes f^*\VV_d$. To be formal, denote by $U_1^{\vee}$ the image of this 
    and by $(u_1)|_S^{\vee}\colon \OO_S \otimes W_1^{\vee}\to U_1^{\vee}$ the isomorphism to the image, so that
    $(u_1)|_S\colon \OO_S \otimes W_1 \to U_1$ is also an isomorphism of vector bundles. Let
    \[
        \varphi_1 := (u_1)|_S^{-1} \circ \varphifr_1\colon U_1\to f^*(\VV_1 \otimes \LL^\vee).
    \]
    Having done this, the next steps follow by induction. Suppose we constructed $U_j^{\vee}$, $\varphi_j$, $(u_j)|_S$ for $j=1, \ldots , i-1$.
    To construct $U_i^{\vee}$, use $u_1, \ldots ,u_{i-1}$ to translate the tensor $\Tpparfr{}_{\leq i}$ to a tensor in
    \[
        U_1\otimes \ldots \otimes U_{i-1}\otimes (\OO_S \otimes W_i)\otimes f^*(\VV_{i+1}\otimes \ldots \otimes \VV_d).
    \]
    Its contraction with respect to $W_i$ yields the desired $U_i$, and then $(u_i)|_S$, $\varphi_i$ are constructed as above.
    We obtain also a concise tensor $\Tppar{}\in U_1\otimes \ldots \otimes U_d$ from $\Tpparfr{}$ and $u_1, \ldots ,u_d$.
    By repetitive applications of Proposition~\ref{ref:functorOfUnrestrictions:prop}, the tuple $(\Tppar{}, \varphi_1, \ldots ,\varphi_d)$ yields
    a map $S\to \unresfullpar{T}$. The isomorphisms $(u_1)|_S$, \ldots, $(u_d)|_S$ then yield a map to $\unresfrfull$.
\end{proof}

\begin{definition}\label{def:framed_concise_secant}
    Let $V_1, \ldots ,V_d$ be vector spaces. Fix an integer $m\geq 1$. Recall that $\Tproj$ denotes the universal tensor 
    over $\mathbb{P}(\tspace)$ and $\Taf$ denotes the universal tensor over $\tspace$. 
    \begin{enumerate}
        \item
            \emph{The $m$-th framed concise secant variety} $\csigmafr_m$ is the minimal border rank locus
            on the framed unrestriction scheme 
            $\upsilon \colon \unresfrfullpar{\Tproj} \to \unresfullpar{\Tproj}$.
            Since the universal tensor on $\unresfrfullpar{\Tproj}$ is a pullback of the one on $\unresfullpar{\Tproj}$,
            we have $\csigmafr_m = \upsilon^{-1}(\csigma_m)$,
            and hence $\upsilon$ restricts to a smooth surjective map $\csigmafr_m\onto \csigma_m$.
        \item
            \emph{The $m$-th framed arithmetic concise secant variety} $\csigmahatfr_m$ is the minimal border rank locus
            on the framed unrestriction scheme 
            $\upsilon \colon \unresfrfullpar{\Taf} \to \unresfullpar{\Taf}$.
            Since the universal tensor on $\unresfrfullpar{\Taf}$ is a pullback of the one on $\unresfullpar{\Taf}$,
            we have $\csigmahatfr_m = \upsilon^{-1}(\csigmahat_m)$,
            and $\upsilon$ restricts to a smooth surjective map $\csigmahatfr_m\onto \csigmahat_m$.
    \end{enumerate}
\end{definition}

\subsection{Geometric interpretation of the algorithm from \S\ref{sec:unrestrictions}}\label{sec:unrestrictions_geometric}
In this section, we explain the geometry behind the algorithm from Theorem~\ref{p:algorithm_segre}
using concise secant varieties. 
\smallskip

Let $T \in \sigmahat_m$ be a tensor of border rank $\leq m$ and let $T_t \in \tspace[\![t]\!]$ 
be a degeneration of border rank $\leq m$ tensors to $T$ with general tensor $T_{t \neq 0}$ concise.
This is the data of a morphism
\[T_t \colon \Spec \kk[\![t]\!] \to \sigmahat_m \subset \tspace\]
such that $T_{t \neq 0} \colon \Spec \kk(\!(t)\!) \to \sigmahat_m$
factors through the embedding of the concise locus $\sigmahat_m^{\mathrm{conc}} \subset \sigmahat_m$ 
\[
\begin{tikzcd}
    \Spec \kk(\!(t)\!) \arrow[d, hook] \arrow[r, "T_{t \neq 0}"] & \sigmahat_m^\mathrm{conc} \arrow[d, hook] \\
    \Spec \kk[\![t]\!] \arrow[r, "T_t"] & \sigmahat_m 
\end{tikzcd}
\]
and $T_{t = 0} \colon \Spec \kk \to \sigmahat_m$ corresponds to $T$.
Consider the concise secant variety $\csigmahat_m$ together with the map $\rhohat \colon \csigmahat_m \to \sigmahat_m$,
see Definition~\ref{ref:conciseSecant:def}. 
The map $\sigmahat$ is an isomorphism over $\sigmahat_m^{\mathrm{conc}}\subset \sigmahat_m$, 
so we can view the general tensor
as a morphism $T_{t \neq 0} \colon \Spec \kk(\!(t)\!) \to \rhohat^{-1}(\sigmahat_m^{\mathrm{conc}}) \subset \csigmahat_m$,
so that we have the following commutative diagram:
\[
\begin{tikzcd}
    \Spec \kk(\!(t)\!) \arrow[d] \arrow[r, "T_{t \neq 0}"] & \csigmahat_m \arrow[d, "\rhohat"] \\
    \Spec \kk[\![t]\!] \arrow[r, "T_t"] & \sigmahat_m 
\end{tikzcd}
\]
The morphism $\rhohat$ is projective, so and hence proper, so by the valuative criterion for properness 
there exists a unique lift 
\begin{equation}\label{eq:lift_cisgma}
\begin{tikzcd}
    \Spec \kk(\!(t)\!) \arrow[d] \arrow[r, "T_{t \neq 0}"] & \csigmahat_m \arrow[d, "\rhohat"] \\
    \Spec \kk[\![t]\!] \arrow[r, "T_t"] \arrow[ru, "\exists!"] & \sigmahat_m 
\end{tikzcd}
\end{equation}
This morphism corresponds to a degeneration of unrestrictions 
\[
    (\Tppar{}_t, (\varphi_1)_t, \dots, (\varphi_d)_t) \colon 
    \Spec \kk[\![t]\!] \to \csigmahat_m.
\]
This is the unique degeneration of unrestrictions that extends the degeneration $T_t$.
For $t=0$, we obtain a minimal border rank unrestriction 
$(\Tppar{}, \varphi_1, \dots, \varphi_d)\in \csigmahat_m$ of the tensor $T$ 
that is uniquely determined by the choice of a degeneration $T_t$.

Let us compare it to the algorithmic proof of Theorem~\ref{p:algorithm_segre}.
Recall that the algorithm takes a degeneration
$T_t \colon \Spec\kk[\![t]\!] \to \sigmahat_m$ and produces a degeneration 
$\Tppar{}_t \colon \Spec\kk[\![t]\!] \to \sigmahat_m$ together with degenerations of restriction maps 
$(\varphi_1)_t, \ldots, (\varphi_d)_t$
by iteratively applying the construction described in the proof of Lemma~\ref{l:algorithm_step}.
Geometrically, the algorithm yields an explicit lift
\begin{equation}\label{eq:lift_csigmafr}
\begin{tikzcd}
     & \csigmahatfr_m \arrow[d, "\upsilon"] \\
     & \csigmahat_m \arrow[d, "\rhohat"] \\
    \Spec \kk[\![t]\!] \arrow[ruu]\arrow[r, "T_t"] & \sigmahat_m
\end{tikzcd}
\end{equation}
A choice of a degeneration $T_t$ does not uniquely determine an unrestriction because 
each step of the algorithm involves a choice of bases and a choice of a minor of the matrix 
describing the flattening of the degeneration obtained in the previous step. 
On the geometric side, this is reflected by the fact that the composite map $\rhohat \circ \upsilon$
is not proper, so there is no reason for a lift as in~\eqref{eq:lift_csigmafr} to be unique.
However, every such lift has the property that the composite map 
$\Spec \kk[\![t]\!] \to \csigmahat_m$ gives a lift as in~\eqref{eq:lift_cisgma},
which is unique. This shows that any degeneration of unrestrictions produced by the algorithm from 
Theorem~\ref{p:algorithm_segre} gives the same isomorphism class in $\csigmahat_m(\kk[\![t]\!])$.
The result below formalizes these observations.
\begin{proposition}\label{prop:uniqueness_of_unres}
    Let $T \in \sigmahat_m \subset \tspace$ be a tensor of border rank $\leq m$. 
    Fix a degeneration of border rank $\leq m$ tensors $T_t \colon \Spec\kk[\![t]\!] \to \sigmahat_m$ 
    to $T$ such that the general tensor is $T_{t \neq 0}$ concise. 
    Consider any two unrestrictions 
    \[
    (\Tppar{}, \varphi_1, \dots, \varphi_d)\quad\text{and}\quad (\Tppar{'}, \varphi_1', \dots, \varphi_d')
    \] 
    constructed from $T_t$ using the algorithm from Theorem~\ref{p:algorithm_segre},
    but for different choices of bases and minors.
    Then there exists linear isomorphisms $\psi_i \colon V_i \to V_i$ such that 
    \[
        \Tppar{'} = \psi_{1,\dots,d}(\Tppar{}) \quad\text{and}\quad \varphi_i' = \varphi_i \circ \psi_i^{-1} \quad\text{for every} \quad i=1,\dots,d.
    \]
\end{proposition}
\begin{proof}
    Recall that $W_1, \dots, W_d$ are the $m$-dimensional vector spaces from the definition of 
    framed concise secant variety, see Definition~\ref{def:framed_concise_secant}.
    Fix bases of $W_1, \dots, W_d$. At the first step, the algorithm fixes a basis of $V_1$ and 
    produces a $V_1$-concise degeneration of tensors 
    $\Tppar{1}_t \colon \Spec \kk[\![t]\!] \to \sigmahat_m \subset \tspace$ and a degeneration of maps 
    $(\varphi_1)_t \colon \Spec \kk[\![t]\!] \to \End(V_1)$ such that 
    $(\varphi_1)_t(\Tppar{1})_t = T_t$. We work in bases, so we can view $(\Tppar{1}_t, (\varphi_1)_t)$
    as a map $\Spec \kk[\![t]\!] \to W_1 \otimes V_{2,\dots,d} \times \Hom(W_1, V_1)$.
    By iterating this procedure, we obtain a map 
    \[
        (\Tppar{}_t, (\varphi_1)_t, \dots, (\varphi_d)_t) \colon \Spec \kk[\![t]\!] \to \csigmahatfr_m \subset W_{1,\dots,d}\times \prod_{i=1}^d \Hom(W_i, V_i).
    \]
    By construction, the composite map $\upsilon \circ (\Tppar{}_t, (\varphi_1)_t, \dots, (\varphi_d)_t)
    \colon \Spec \kk[\![t]\!] \to \csigmahat_m$ is a lift as in~\eqref{eq:lift_cisgma}.
    By properness of $\rhohat$, such a lift is unique, so $\upsilon \circ (\Tppar{}_t, (\varphi_1)_t, \dots, (\varphi_d)_t)$
    does not depend on any choices made in the algorithm. In particular, any two unrestrictions 
    $(\Tppar{}, \varphi_1, \dots, \varphi_d), (\Tppar{'}, \varphi_1', \dots, \varphi_d')$
    constructed from $T_t$ lie in the same fiber of $\upsilon$. The map $\upsilon$ is a principal
    $\GLW$-bundle, so any two unrestrictions in the same fiber differ by some $(\psi_1, \dots, \psi_d) \in \GLW$.
\end{proof}

\subsection{Smooth equivalences}\label{ssec:smoothEquivalences}
The most important consequence of
Proposition~\ref{ref:framedUnrestrictionPoints:prop} is that it allows us to reduce to the usual setup as follows.
Let $\Taf$ be the universal tensor on $\tspace$ and $\unresfrfullpar{\Taf}\to \tspace$ be the framed unrestriction scheme.
For a $\kk$-scheme $S$, to have a morphism $S\to
\tspace$ is the same as to have a tensor $T = \Taf|_S\colon \OO_S\to
\OO_S\otimes \tspace$. Therefore, the functor of points of $\unresfrfullpar{\Taf}$ simplifies:
we can forget the map $S\to X$ \emph{and} the last condition in~\eqref{eq:starCondition}. This means
that $\unresfrfullpar{\Taf}$ parameterises tuples $(\Tpparfr{}, \varphifr_1, \ldots ,\varphifr_d)$ where
$\Tpparfr{}$ is a concise tensor and $\varphifr_i\colon W_i\to V_i$ are
arbitrary restrictions (for $i=1, \ldots ,d$) that satisfy conciseness conditions
from~\eqref{eq:starCondition} without the last one. Therefore, the scheme
$\unresfrfullpar{\Taf}$ is an open subscheme of the affine space
\[
    \tconcspace \times \Hom(W_1, V_1)\times \ldots \times  \Hom(W_d, V_d).
\]
Let us denote the projection $\unresfrfullpar{\Taf}\to \tconcspace$ by $\alphahat$.
Let us denote by $\tconcspaceconc$ the open subvariety of $\tconcspace$ that parameterises
concise tensors.
We obtain the \emph{affine smooth equivalence} diagram
\begin{equation}\label{eq:smoothEquivalenceFramed}
    \begin{tikzcd}
        & \unresfrfullpar{\Taf} \ar[rd, "\alphahat", "\mathrm{sm}"']\ar[ld, two heads, "\upsilon"', "\mathrm{sm}"]\\
        \unresfullpar{\Taf} && \tconcspaceconc.
    \end{tikzcd}
\end{equation}
The map $\alphahat$ is smooth as a projection, while $\upsilon$ is smooth as a structure map of a principal $\GLW$-bundle.
The family of tensors $\Tpparfr{}$ on $\unresfrfullpar{\Taf}$ is defined as the
pullback $\upsilon^*\Tppar{}$.  Moreover, by the construction of the framed
concise secant variety, we obtain $\upsilon^*\Tppar{} \simeq \Tpparfr{} = \alphahat^*\Taf_W$.

An important consequence is of smooth equivalence is the smoothness of full unrestriction schemes for universal tensors on $\tspace$ and $\mathbb{P}(\tspace)$.
\begin{corollary}[Smoothness of universal unrestrictions]\label{ref:unrestrictionSmoothness:cor}
    Let $\Taf$, $\Tproj$ be the usual universal tensors on $\tspace$ and $\mathbb{P}(\tspace)$, respectively.
    The schemes $\unresfullpar{\Taf}$, $\unresfullpar{\Tproj}$ are smooth.
    Moreover, the fibre $\unresfullpar{\Taf} \to \tspace$ over $0\in \tspace$ is smooth.
\end{corollary}
\begin{proof}
    The smoothness of $\unresfullpar{\Taf}$ follows directly from~\eqref{eq:smoothEquivalenceFramed}.
    Propositions~\ref{ref:bundleAffine:prop}-\ref{ref:bundleProjective:prop} show that
    $\unresfullpar{\Tproj}$ is a projectivisation of a total space of vector bundle $\unresfullpar{\Taf}$, hence it is smooth
    as well. The fibre of $\unresfullpar{\Taf}\to \tspace$ over $0$ is the zero section of the aforementioned vector bundle, hence is smooth
    as well.
\end{proof}

\begin{example}\label{ex:fibres}
    \def\bfP{\mathbf{P}}
    In general, it is not true that the unrestriction scheme of any tensor is smooth.
    To show this, by Example~\ref{ex:fibre}, it is enough to show that a fiber of 
    $\pihat\colon \unresfullpar{\Taf}\to \tspace$ over some tensor is not smooth. 
    
    Take $d =3$, $\dim V_\bullet = 2$ and let $T \in V_{1,2,3}$ be a rank one tensor.
    The fiber $\pihat_1^{-1}(T) \subset \Gr(2, V_2 \otimes V_3)$ is isomorphic 
    to the projective plane $\bfP = \PP((V_2 \otimes V_3) / \spann{T(V_1^\vee)})$. 
    There is one point $x_2 \in \pihat_1^{-1}(T)$ such that $\Tppar{1}|_{x_2}$ 
    is not $V_2$-concise and another point $[x_3]$ such that $\Tppar{1}|_{x_3}$ 
    is not $V_3$-concise. Similarly as above, the fibers of $\pihat_2 \circ \pihat_3$ over 
    $x_2$ and $x_3$ are projective planes $\bfP_2, \bfP_3$. Summing up, $\pihat^{-1}(T)$ 
    consists of three projective planes $\bfP, \bfP_2, \bfP_3$ glued by identifying points 
    $x_2' \in \bfP_2$, $x_3'\in \bfP_3$ with points $x_2, x_3 \in \bfP$. 
\end{example}

Another consequence, which is crucial to this article, is that the singularities of $\csigmahat_m$, $\csigma_m$
are smoothly equivalent to the singularities of the \emph{concise locus} of $\sigmahat_m$.
The concise locus is much better behaved than the whole variety $\sigmahat_m$.
Below we use the framed minimal border rank loci, see
Definition~\ref{def:framed_concise_secant}.

\begin{theorem}\label{ref:SmoothEquivalenceCsigmahat:thm}
    The diagram~\eqref{eq:smoothEquivalenceFramed} restricts to a smooth equivalence
    \begin{equation}\label{eq:smoothEquivalenceCsigmahat}
        \begin{tikzcd}
        & \csigmahatfr_m \ar[rd, "\alphahat", "\mathrm{sm}"']\ar[ld, two heads, "\upsilon"', "\mathrm{sm}"]\\
            \csigmahat_m && \sigmahat_m^{\mathrm{conc}}
        \end{tikzcd}
    \end{equation}
    where $\csigmahatfr_m \subset \unresfrfullpar{\Taf}$ is the minimal border rank locus (Definition~\ref{def:framed_concise_secant}) and $\sigmahat_m^{\mathrm{conc}}\subseteq \tconcspaceconc$
    is the concise locus of the secant variety inside $\tconcspace$.
\end{theorem}
\begin{proof}
    By Definition~\ref{ref:secantLocus:def}, the variety $\sigmahat_m^{\mathrm{conc}}$ is equal to $(\tconcspaceconc)^{\brk = m}$ and the variety $\csigmahat_m$
    is equal to $\unresfullpar{\Taf}^{\brk =m}$.
    Since the tensor on $\unresfrfullpar{\Taf}$ is isomorphic to the pullback of the tensor on $\unresfullpar{\Taf}$, we have that
    \[
        \csigmahatfr_m = (\unresfrfullpar{\Taf})^{\brk=m} = \upsilon^{-1}((\unresfullpar{\Taf})^{\brk = m}),
    \]
    hence $\upsilon$ restricts to a smooth map $\csigmahatfr_m\to \csigmahat_m$. The tensor
    on $(\unresfrfullpar{\Taf})$ is also isomorphic to the pullback of the tensor on $\tconcspaceconc$, so the same argument yields that
    \[
        \csigmahatfr_m = \alphahat^{-1}((\tconcspaceconc)^{\brk = m}) = \alphahat^{-1}(\sigmahat_m^{\mathrm{conc}}).\qedhere
    \]
\end{proof}

We can repeat the procedure for the projective versions. Actually, there are two strategies to do this.
First, by Propositions~\ref{ref:bundleAffine:prop}-\ref{ref:bundleProjective:prop}, we know that $\unresfullpar{\Tproj}$
is a free quotient by $\Gmult$ of an open subset of $\unresfullpar{\Taf}$ and one can obtain the smooth equivalence
by quotienting out~\eqref{eq:smoothEquivalenceFramed}.
We take a different route, mostly to show that the result is obtained purely from the general theory.

Consider $\unresfrfullpar{\Tproj}$. Proposition~\ref{ref:framedUnrestrictionPoints:prop} yields its functor of points.
Once again, for a morphism $f\colon S\to \unresfrfullpar{\Tproj}$, we obtain in
particular a tensor $T\colon f^*\OO(-1)\to \tspace$, a $W_2$-concise tensor $T' := \varphifr_2\circ \ldots
\circ\varphifr_d(\Tpparfr{})$ and a restriction
map $\varphifr_1\colon \OO_S\otimes W_1\to f^*\OO(1)\otimes V_1$ such that $\varphifr_1(T') = T$. The contraction $(T')_1^{\vee}$ is a surjective map
$\Vhatpar{1}^{\vee}\onto W_1$. Together with $\varphifr_1$ it yields a surjection
\[
    \OO_S \otimes \tspace^{\vee} = \OO_S\otimes V_1^{\vee}\otimes \Vhatpar{1}^{\vee} \onto f^*\OO(1),
\]
hence a map $S\to \mathbb{P}\tspace$, which coincides with $T$. Therefore, once more, we can discard $T$ and the last
equality of~\eqref{eq:starCondition} from the data and obtain only open conditions on $\varphifr_1, \ldots ,\varphifr_d$.
We conclude that there is a \emph{projective smooth equivalence}
\begin{equation}\label{eq:smoothEquivalenceFramedProj}
    \begin{tikzcd}
        & \unresfrfullpar{\Tproj} \ar[rd, "\alpha", "\mathrm{sm}"']\ar[ld, two heads, "\upsilon"', "\mathrm{sm}"]\\
        \unresfullpar{\Tproj} && \tconcspace^{\mathrm{conc}}.
    \end{tikzcd}
\end{equation}
One might expect to see the projectivisation of $\tconcspace^{\mathrm{conc}}$ on the right hand side. However, observe that
on the concise locus this makes little difference: $\tconcspace^{\mathrm{conc}}$ is a $\Gmult$-bundle over its projectivisation.
Arguing exactly as in Theorem~\ref{ref:SmoothEquivalenceCsigmahat:thm} we obtain the following smooth equivalence
\begin{equation}\label{eq:smoothEquivalenceCsigma}
    \begin{tikzcd}
        & \csigmafr_m \ar[rd, "\alpha", "\mathrm{sm}"']\ar[ld, two heads, "\upsilon"', "\mathrm{sm}"]\\
        \csigma_m && \sigmahat_m^{\mathrm{conc}}
    \end{tikzcd}
\end{equation}
Moreover, we can use the framing on the intermediate secant varieties $c_{1, \ldots i}\sigma_m \subseteq \unrespar{1, \ldots ,i}{\Tproj}$.
For them, we obtain smooth equivalences
\begin{equation}\label{eq:smoothEquivalenceCsigmaPartial}
    \begin{tikzcd}
        & c_{1,\dots,i}\sigma^\mathrm{fr}_m \ar[rd, "\alpha_{1, \ldots ,i}", "\mathrm{sm}"']\ar[ld, two heads, "\upsilon_{1, \ldots ,i}"', "\mathrm{sm}"]\\
        c_{1, \ldots ,i}\sigma_m && \sigmahat_m^{W_1-, \ldots, W_i-\mathrm{conc}}
    \end{tikzcd}
\end{equation}
where $c_{1,\dots,i}\sigma^\mathrm{fr}_m \subset (\unresfrpar{1, \ldots ,i}{\Tproj})^{\brk = m}$ is the minimal border rank locus and $\sigmahat_m^{W_1-, \ldots, W_i-\mathrm{conc}}\subseteq W_1\otimes \ldots \otimes W_i\otimes V_{i+1}\otimes \ldots \otimes V_d$ is the
locus of $W_1$-, \ldots , $W_i$-concise tensors of border rank $m$.

\begin{corollary}[Dimension of concise secant]\label{ref:dimensionOfcsigma:cor}
    For every $V_1, \ldots , V_d$ such that $\csigma_m$ is nonempty, the dimension of $\csigma_m$ is $m\cdot (\sum_{i=1}^d \dim V_{i}) - m(d-1)-1$,
    and the dimension of $\csigmahat_m$ is $m\cdot (\sum_{i=1}^d \dim V_{i}) - m(d-1)$.
\end{corollary}
\begin{proof}
    The spaces $W_1, \ldots ,W_d$ are $m$-dimensional by definition, so
    the secant variety $\sigmahat_m^{\mathrm{conc}}\subseteq \tconcspace$ has expected dimension $dm^2 - m(d-1)$.
    The fibres of $\alphahat$ have dimension $m\cdot \left( \sum_{i=1}^d \dim
    V_i \right)$, the fibres of $\alpha$ one dimension less, and the fibres of
    $\upsilon$ have dimension $dm^2$, hence the claim follows.
\end{proof}

\begin{remark}
    Usually, the definition of smooth equivalence assumes surjectivity of both maps. In our setup, the maps $\alphahat$, $\alpha$ need
    not be surjective. However, they fail to be such only in somewhat pathological situations, so we chose to use the term anyway.
\end{remark}

\subsubsection{Smooth equivalences and singularities}\label{ssec:singularitiesMain}

    Above, we reduced the ``iterated setup'' to the usual one, that is, to the concise locus of $\sigmahat_m \subseteq W_{1,\dots, d}$.
    Part of this locus, which is ``1-generic enough'' is understood thanks to a
    long-investigated connection to the Hilbert
    scheme and Quot scheme~\cite{poonen_moduli_space, Blaser_Lysikov,
    Landsberg_Michalek__Abelian_Tensors, Wojtala_Irreversibility, JLP, GJLM}. This allows us to use the already
    existing knowledge of singularities of Hilbert and Quot schemes to obtain corresponding statements about the singularities of the
    concise secant variety.

    Consider a concise tensor $\Tppar{}\in \tconcspace$. Recall that $\Tppar{}$ is $1_{W_i}$-generic if there exist an element
    $\alpha_i\in W_i^{\vee}$ and an index $j$ such that $\Tppar{}(\alpha_i)\colon W_j^{\vee}\to W_1\otimes \ldots \otimes W_{\hat{i}}\otimes \ldots \otimes W_{\hat{j}}\otimes  \ldots \otimes W_d$ has maximal rank. This is an open condition. It will be discussed more thoroughly in
    Proposition~\ref{ref:genericity:prop}. 
    Currently, we need the following dictionary, which emerged in~\cite{JLP} using the ideas of~\cite{Landsberg_Michalek__Abelian_Tensors}.
    \begin{center}
    \begin{tabular}{c c}
        a tensor generic on coordinates & is isomorphic to\\
        \toprule
        $\geq d-2$ & multiplication tensor in a module\\
        $\geq d-1$ & multiplication tensor in an algebra\\
        $d$ & multiplication tensor in a Gorenstein algebra
    \end{tabular}
\end{center}

    The part of smooth equivalences which we need may be informally summarised as follows:
    \begin{enumerate}
        \item if $T$ is a tensor isomorphic to a multiplication tensor in an algebra $\cA$, then
            $T$ has minimal border rank if and only if $\cA$ is smoothable. Suppose that this condition holds,
            then the geometry of $T\in \sigmahat_m$ is the same as the geometry of any point $[\Spec(\cA)\subseteq \mathbb{A}^N]\in \Hilb_m^{\sm}(\mathbb{A}^N)$,
            up to adding free parameters,
        \item if $T$ is a tensor isomorphic to a multiplication tensor in a module $M$, then
            $T$ has minimal border rank if and only if $M$ is smoothable. Suppose that this condition holds,
            then the geometry of $T\in \sigmahat_m$ is the same as the geometry of any point $[\OO_{\mathbb{A}^N}^r\onto M]\in \Quot_m^{\sm}(\OO_{\mathbb{A}^N}^r)$,
            up to adding free parameters.
    \end{enumerate}
    For $m=2,3, 4$, we will only need the following part.
    \begin{proposition}[\cite{GJLM}, smoothness]\label{ref:Hilbsmoothness:prop}
        \begin{enumerate}
            \item Suppose that a tensor $T$ is isomorphic to a multiplication tensor in an algebra $\cA$.
                Then $T\in \sigmahat_m$ is a smooth point if and only if a point $[\Spec(\cA)\subseteq \mathbb{A}^N]\in \Hilb_m^\sm(\mathbb{A}^N)$
                is smooth.
            \item Suppose that a tensor $T$ is isomorphic to a multiplication tensor in a module $M$.
                Then $T\in \sigmahat_m$ is a smooth point if and only if a point $[\OO_{\mathbb{A}^N}^r\onto M]\in \Quot_m^\sm(\OO_{\mathbb{A}^N}^r)$
                is smooth.
        \end{enumerate}
    \end{proposition}
    This part, together with a smooth equivalence~\eqref{eq:smoothEquivalenceCsigmahat}, allows us to use the existing results 
    on $\Hilb^{\sm}$ and $\Quot^{\sm}$ to gain knowledge of $\csigma_m$. The
    amount of information obtained in this way is quite dramatic, if one
    compares it with what is known about $\sigma_m$.

    \begin{theorem}[Good geometry of $\csigma_m$]\label{ref:singularities:thm}
        The varieties $\csigma_2$ and $\csigma_3$ are smooth.
        The variety $\csigma_4$ is normal, Gorenstein (in particular, Cohen-Macaulay) and has terminal singularities, hence rational ones.
        Its singular locus has codimension $9$ and blowing it up resolves the singularities.
    \end{theorem}
    \begin{proof}
        For $m=2,3$, we employ Proposition~\ref{ref:Hilbsmoothness:prop}.

        For $m=2$, every minimal border rank tensor $T\in \tconcspace$ is either a unit tensor or a generalised
        $W$-state. These two cases are iterated multiplication tensors in algebras $\kk\times \kk$ and $\kk[\varepsilon]/\varepsilon^2$,
        respectively. Both of these algebras can be embedded into
        $\mathbb{A}^1$ and the Hilbert scheme $\Hilb_2(\mathbb{A}^1)\simeq
        \mathbb{A}^2$ is smooth.

        The case $m=3$ is similar. In this case, every minimal border rank tensor $T$ is generic on at least $d-1$ coordinates (for example,
        by the normal forms given in~\cite[p.478]{Buczynski_Landbserg_sigma3}),
        hence corresponds to a multiplication tensor in an algebra $\cA$ with $\dim_{\kk} \cA = 3$. (The unique case not generic
        on all coordinates is the multiplication tensor in $\cA = \kk[x, y]/(x, y)^2$.)

        A scheme $\Spec(\cA)$ has degree $3$, so there exists a closed embedding $\Spec(\cA) \subseteq \mathbb{A}^2$. The Hilbert
        scheme of $3$ points on $\mathbb{A}^2$ is smooth as a very special case of~\cite{fogarty}, so we obtain that $T$ is smooth.

        The argument for $\csigma_4$ is quite similar, but the moduli space considerations become more involved. 
        Every minimal border rank tensor $T$ is generic on at least $d-2$ coordinates,
        hence corresponds to a multiplication tensor in a module $M$.

        The possible modules are classified, for example in~\cite[Appendix~B]{Jagiella_Jelisiejew}. We will keep the indexing consistent
        with this appendix. Let
        $\cA_{2,6}
        = \kk[x_1,x_2,x_3]/(x_1,x_2,x_3)^2$ and let $M_{2,8}$ be the $\kk[x_1,
        x_2, x_3]$-module
        given by
        \[
            M_{2,8} = \frac{\kk[x_1, x_2, x_3]e_1 \oplus \kk[x_1, x_2, x_3]e_2}{(x_1e_1 - x_2e_2,\ x_2e_1,\ x_3^2e_1,\ x_1e_2,\ x_3e_2,\ x_2^2e_2)}.
        \]
        Let $U_{2,6}$ and $U_{2,8}$ be the corresponding tensors.
        Using the package \emph{Versal Deformations}~\cite{Ilten}, we verify that $\OO^{\oplus 2} \onto M_{2,8}$ is a smooth point of the principal
        component of the Quot scheme (we warn the reader that this point lies also on other components of the Quot scheme), so $U_{2,8}$ is a smooth point
        by Proposition~\ref{ref:Hilbsmoothness:prop}.
        In contrast, the point $\Spec(\cA_{2,6})\subseteq \mathbb{A}^3$ of the Hilbert scheme $\Hilb_4(\mathbb{A}^3)$
        is known to be singular, so also $U_{2,6}$ is singular.
        This allows us to completely describe the singular locus. Namely, by~\cite[Appendix~B]{Jagiella_Jelisiejew} again, every tensor of minimal border rank
        $m=4$ is either isomorphic to $U_{2,6}$ or degenerates to $U_{2,8}$. In the latter case, the tensor is smooth.
        Hence, we conclude that the singular locus consists precisely of tensors isomorphic to $U_{2,6}$. The
        schemes isomorphic to $\Spec(\cA_{2,6})\subseteq \mathbb{A}^3$ fill a locus isomorphic to $\mathbb{A}^3 \subseteq \Hilb_4(\mathbb{A}^3)$,
        hence a locus of codimension $9$. Using the smooth equivalence, this shows that also the tensors isomorphic to $U_{2,6}$
        fill a locus of codimension $9$ inside $\csigma_4$. This can also be deduced independently from~\cite[Theorem 1.1]{BGJLM} at least in the case $d=3$.

        Sheldon Katz~\cite{Katz__4points} (see also
        \cite{stevens_deformations_of_singularities, AJR}) described the singularity
        of $\Hilb_4(\mathbb{A}^3)$ and showed that it is a cone over the
        Grassmannian $\widehat{\Gr(2, 6)}\into \mathbb{A}^{15}$. The
        Grassmannian $\Gr(2,6)$ in its Pl{\"u}cker embedding is arithmetically
        Gorenstein, so this singularity is Gorenstein. It is also normal and terminal (hence, rational) by~\cite[Lemma~3.1, \S2.5]{Kollar__Sings}.
        Blowing up the singular locus corresponds to blowing up the cone point of $\widehat{\Gr(2, 6)}$. This yields
        a line bundle over $\Gr(2, 6)$, hence a nonsingular variety, so this blowup resolves the singularities.
    \end{proof}

\begin{corollary}[Equations of concise secant]\label{ref:equationsSmallr:cor}
    Assume that $\kk$ has characteristic zero.
    For every $m\leq 5$, the concise secant $\csigma_m$, $\csigmahat_m$ are
    set-theoretically equal to the centroid abundant loci (see
    Definition~\ref{ref:secantLocus:def}).
    For every $m\leq 4$, these are equalities of schemes.
\end{corollary}

\begin{proof}[Proof-sketch]
    Using~\eqref{eq:smoothEquivalenceCsigma}-\eqref{eq:smoothEquivalenceCsigmahat}, we are reduced to showing that
    the concise part of $\csigmahat_m$ is given by the centroid abundant equations. This is proven in~\cite[Theorem~1.6]{JLP}.
    The equality in~\cite[Theorem~1.6]{JLP} is purely set-theoretic. To obtain the scheme-theoretic one, we use the
    smooth equivalence as in the proof of Theorem~\ref{ref:singularities:thm} above to reduce to showing that appropriate
    Hilbert and Quot schemes are reduced.
\end{proof}

    \begin{remark}\label{ref:singularitiesFive:rmk}
        The variety $\csigma_5$ also seems to have unexpectedly good singularities, in particular its singular locus
        is much smaller than naively expected: the points that correspond to $1$-degenerate tensors (found in~\cite{JLP}) are smooth.
        This implies that also for $m=5$ the equality of Corollary~\ref{ref:equationsSmallr:cor} is scheme-theoretic.
        This will be explained in a future work.
    \end{remark}

    \begin{proof}[Proof of Theorem~\ref{ref:introQuotResolution:thm}]
        The claim is \'etale-local, hence we may assume that the scheme is $\mathbb{A}^n$ and $\EE$ is trivial.
        In this case, Theorem~\ref{ref:introQuotResolution:thm} follows from smooth equivalences. Let $d = 3$.
        As remarked above, the Quot scheme $Q_r$ is smoothly equivalent to the $1_{V_1}$-generic locus of $\sigma_r$; the multiplication tensors are
        not necessarily concise in the $V_1$ coordinate, so we need a slightly more general equivalence, whose proof is the same and rests on~\cite[\S2]{JLP}.
        As explained in~\cite{Landsberg_Michalek__Abelian_Tensors}, any element proving $1_{V_1}$-genericity yields an isomorphism
        of the tensor with a tensor in $V_1 \otimes V_3^{\vee} \otimes V_3$. 
        A $1_{V_1}$-generic tensor is concise on both $V_2$ and $V_3$, hence the unrestriction only changes $V_1$, replacing it
        with a subspace $V_1^{\vee}\subseteq W_1^{\vee} \into V_3^{\vee}\otimes V_3$. By \cite[Proposition~3.2]{JLP}, the resulting tensor in
        $W_1^{\vee}\otimes V_3^{\vee}\otimes V_3$ is centroid-abundant exactly when $W_1^{\vee}$ is a commutative subalgebra.
    \end{proof}
    The above results apply to small $m$. We need to stress, however, that even for larger $m$ there are many partial results
    describing the singularities of Hilbert and Quot schemes, see, for example~\cite{Hu__singular_Hilbert_schemes, AJR}. Hence,
    when enough coordinates are generic, it is most fruitful to use smooth equivalences to pass from $\csigma_m$ to analysing 
    $\Hilb_m^{\sm}$ or $\Quot_m^{\sm}$.

    The question of Cohen-Macaulayness of secant varieties has been addressed in many special cases in the literature,
    see, for example~\cite{landsberg_weyman_secants_ideal_sings_of_sv_of_Segree, Conca_deNegri_Stojanac}.
    Oeding put forward the following very bold, yet elegant, conjecture.
    \begin{conjecture}[{\cite[Conjecture~1.1]{Oeding__Cohen_Macaulay}}]\label{ref:Oeding:conj}
        All secant varieties of Segre product of projective spaces are arithmetically Cohen-Macaulay (that is, the affine cone $\sigmahat_r$ is Cohen-Macaulay).
    \end{conjecture}
    The conjecture is wide open.
    Via the smooth equivalences, it contains the conjectures of Artin-Hochster on the Cohen-Macaulayness of the Quot scheme of $\mathbb{A}^2$,
    and Haiman's conjecture~\cite[Conjecture~5.2.1]{Haiman_macdonald} on the Cohen-Macaulayness of $\Hilb^{\sm}_d(\mathbb{A}^n)$, for any $d,n$.
    Under the assumption of the conjecture, normality of $\sigma_r$ translates into $\sigma_r$
    being regular in codimension one, which is in practise easier to check. Hence, the conjecture suggests also the normality
    of all secant varieties.

\newcommand{\Knot}{K_{0}(\operatorname{Var}_{\kk})}

The following proposition gives the classes of $\csigma_2$ and $\sigma_2$ in the ring $\Knot$, denoted by
$[\csigma_2]$, $[\sigma_2]$, respectively; these are so-called motives. Computing the motive of a smooth projective variety, such as $\csigma_2$, allows to
compute the Betti numbers and the Hodge diamond, in particular. Computing the motive of a singular variety (not necessarily
projective), such as $\sigma_2$, allows one to compute the number of $\mathbb{F}_p$-points of this variety.
We use the notation $\Lmot{} := [\mathbb{A}^1_{\kk}]$.
For page number reasons, we omit the proof, which can be found in the earlier arXiv versions of this paper.

\begin{proposition}[Betti numbers of $\csigma_2$ and $\sigma_2$, {\cite{JJ4arxiv}}]\label{ref:motivecsigma2:prop}
    Let $\csigma_2$ be the second concise secant and let $\sigma_2\subseteq \mathbb{P}(\tspace)$ be the ordinary secant. Their motives are
    given by the formulas
    \begin{align*}
        [\csigma_2] =  (1+\Lmot)&\cdot (1+\Lmot^2)\cdot \left( \frac{1}{2}\left( (1+\Lmot)^{2(d-1)} + (1+\Lmot^2)^{d-1} \right) + \Lmot{}\cdot (1+\Lmot{})^{d-1}\cdot \left( 1+\Lmot{}+\Lmot{}^2+ \ldots +\Lmot{}^{d-3} \right) \right)\\
        [\sigma_2] = [\csigma_2] &- \binom{d}{2}\left( \Lmot{}^3-\Lmot{} \right)\cdot (\Lmot{}+\Lmot{}^2)\cdot (1+\Lmot{})^{d-2}-\\
                                 &- \left( (\Lmot{}+1)^{d-1}+(\Lmot{}+\Lmot{}^2+ \ldots +\Lmot{}^{d-2}) + (d-1)\Lmot{}^2 - 1 \right)\cdot (1+\Lmot{})^{d}.
    \end{align*}
    In particular, both varieties admit a polynomial point count. The Picard rank of $\csigma_2$ is $d$.
\end{proposition}

\subsection{Finer geometry of $\sigma_4$ and $\csigma_4$}\label{ssec:SalmonResults}

In this section we outline the proof of Theorem~\ref{ref:introLocalisation:thm}.
\begin{proposition}\label{ref:finitelyManyTorusFixedPoints:prop}
    The torus fixed points of $\tildecsigma_4$, $\csigma_4$ are isolated and there are $588224$ and $444864$ such points, respectively.
\end{proposition}
\begin{proof}
    The map $\rho\colon \csigma_4\to \sigma_4$ is equivariant, so it is enough to identify the $\TT$-fixed points in a fiber over
    a fixed point $[v_1\otimes v_2\otimes v_3]\in \sigma_4$. For such a point, the first unrestriction $W_1^{\vee}\into V_2 \otimes V_3$
    needs to be a $4$-dimensional torus-fixed subspace of $V_2\otimes V_3$ containing $v_1 \otimes v_2 \otimes v_3$. 
    The $\TT$-weights on $V_2\otimes V_3$ are pairwise different,
    so there are $\binom{16-1}{4-1}$ possibilities.

    After fixing $W_1$, we fix a $4$-dimensional commutative $\TT$-stable subalgebra $R$ in $\End_{\kk}(W_1)$,
    which will be the projection of $\Cen{\Tppar{}}$ of the unrestriction $\Tppar{}$ of $[v_1\otimes v_2\otimes v_3]$.
    There are always $81$ possibilities for isolated $R$. 
    For some $W_1^\vee$, there are also additional $\PP^1$'s of possible $R$.

    Once $R$ is fixed, there are no more choices involved. Namely, the subspace $W_2^{\vee}$ inside $W_1\otimes V_3$ has to
    be $\TT$-stable and closed under the action of $R$, so in particular it must contain the 
    submodule of $W_1 \otimes V_3$ generated by the flattening of previous tensor.
    For each choice of $W_1^\vee$ and $R$ as above, this $R$-submodule has dimension \emph{at least $4$}.
    When the dimension is higher than $4$, the point is discarded (in this way the general points of $\mathbb{P}^1$ families are eliminated, among others).
    When then dimension is $4$, such requirements determine $W_2^{\vee}$ uniquely and
    yield an $R$-module structure on $W_2$. Similarly, $W_3^{\vee}\subseteq
    W_1\otimes W_2$ is obtained uniquely once we require it to be an
    $R$-submodule and $\TT$-stable. Thus we obtain all the fixed points on $\rho^{-1}([v_1\otimes v_2\otimes v_3]) \subset \csigma_4$.
    There are $6951$ of them, so $\csigma_4^{\TT}$ has $64\cdot 6951 = 444864$ fixed points.

    To compute $\tildecsigma_4^{\TT}$, we need to understand which of the fixed points are blown up when resolving as in Theorem~\ref{ref:singularities:thm}.
    A direct computation shows that exactly $160$ out of $6951$ correspond to
    tensors that are isomorphic, perhaps after permuting $W_1$, $W_2$,
    $W_3$, to a multiplication tensor
    on $\cA_{2,6} = \frac{\kk[x_1, x_2, x_3]}{(x_1, x_2, x_3)^2}$. Using Katz' result, we check that over each of these points, there
    are $15$ isolated fixed points, hence in total we have $64\cdot (6951 - 160 + 15\cdot 160) = 588224$ fixed points.
\end{proof}

\begin{proof}[Proof of Theorem~\ref{ref:introLocalisation:thm}]
    \newcommand{\rhonorm}{\rho^{\nrm}}
    \newcommand{\ch}{\operatorname{ch}}
    \newcommand{\Td}{\operatorname{Td}}
    The bundle $\OO(-K_{\tildecsigma_4})$ comes with a natural $\TT$-linearisation and its weight on a $\TT$-fixed point $x\in \tildecsigma_4^{\TT}$
    is equal to minus the sum of $\TT$-weights on the tangent space $T_x\tildecsigma_4$. We check directly point-by-point that $\OO(-K_{\tildecsigma_4})$ has the same
    weights as $\xi^*\rho^*\OO(16) \otimes \OO_{\xi}(5)$;
    for example, for a fixed point over $[e_1 \otimes e_1\otimes e_1]$ away from the exceptional locus of $\xi$, we have always weights
    \[
        (12, -4,-4,-4,12,-4,-4,-4,12,-4,-4,-4),
    \]
    which after rescaling by a constant vector with entry $4$ yield the usual linearisation of $\OO(16)$.
    By equivariant
    formality and localisation~\cite{Fulton_Anderson}, this means that the
    first Chern classes of $\OO(-K_{\tildecsigma_4})$ and $\xi^*\rho^*\OO(16) \otimes \OO_{\xi}(5)$ in $H^2(\tildecsigma_4, \kk)$
    coincide. Since $\tildecsigma_4$ has isolated fixed points, the first Chern class is injective, hence $\OO(-K_{\tildecsigma_4})$ is isomorphic
    to $\xi^*\rho^*\OO(16) \otimes \OO_{\xi}(5)$.
    Since the anticanonical sheaf on $\Gr(2,6)$ is $\OO(6)$, it follows by
    adjunction that $\OO(-K_{\tildecsigma_4}) = \xi^*(\OO(-K_{\csigma_4}))
    \otimes \OO_{\xi}(5)$, so that $\OO(-K_{\csigma_4})\simeq \rho^*\OO(16)$, which is a big and nef line bundle.

    Let $\rhonorm\colon \csigma_4 \to \sigma_4^{\nrm}$ be the natural birational map to the normalisation of $\sigma_4$, so that $\rhonorm_*\OO_{\csigma_4} = \OO_{\sigma_4^{\nrm}}$ by Zariski's Main Theorem. Kawamata-Viehweg vanishing~\cite[Theorem 2.70, Corollary~2.68, Theorem~5.10]{Kollar_Mori} implies that $R\rhonorm_*\OO_{\csigma_4} = \OO_{\sigma_4^{\nrm}}$, that $\sigma_4^{\nrm}$ is Cohen-Macaulay, and in fact Gorenstein with anticanonical sheaf $\rho^{\nrm}_*(\rho^{\nrm})^*\OO(16) \simeq  \OO(16)$. By the same vanishing, the bundles $\OO_{\sigma_4^{\nrm}}(v)$ have
    no higher cohomology for $v\geq -15$. By using Serre's duality on $\csigma_4$, we conclude that for $v \leq -16$ the bundles $\OO_{\sigma_4^{\nrm}}(v)$ have
    no higher cohomology except possibly the top one. By~\cite[Proposition~8.6]{Hartshorne_Deformation_Theory} this implies that $\sigma_4^{\nrm}$ is arithmetically Cohen-Macaulay. Since its canonical sheaf is a multiple of $\OO(1)$, this variety is even arithmetically Gorenstein.
\end{proof}

\begin{example}\label{ex:motiveSigma4}
    The $K_0(\operatorname{Var})$-class of $\tildecsigma_4$ obtained from the \BBname{} decomposition is
    {\tiny
        \begin{align*}
        &\Lmot^{39}+7\Lmot^{38}+32\Lmot^{37}+107\Lmot^{36}+299\Lmot^{35}+711\Lmot^{34}+1510\Lmot^{33}+2876\Lmot^{32}+5002\Lmot^{31}+7969\Lmot^{30}+11749\Lmot^{29}+16095\Lmot^{28}+20669 \Lmot^{27}+\\
        &+25015\Lmot^{26}+28795\Lmot^{25}+31762\Lmot^{24}+33887\Lmot^{23}+35247\Lmot^{22}+36021\Lmot^{21}+36358\Lmot^{20}+36358\Lmot^{19}+36021\Lmot^{18}+35247\Lmot^{17 }+33887\Lmot^{16}+\\
        &+31762\Lmot^{15}+28795\Lmot^{14}+25015\Lmot^{13}+20669\Lmot^{12}+16095\Lmot^{11}+11749\Lmot^{10}+7969\Lmot^{9}+5002\Lmot^{8}+2876\Lmot^{7}+1510\Lmot ^{6}+711\Lmot^{5}+299\Lmot^{4}+\\
        &+107\Lmot^{3}+32\Lmot^{2}+7\Lmot+1.
\end{align*}}
\end{example}

To compute the Euler characteristic of $\OO(v)$, we can use the equivariant
Euler characteristic of $\xi^*\rho^*\OO(v)$, using equivariant
Hirzebruch-Riemann-Roch, see for example, \cite{Weber_equivariant}.

\section{Smoothable and cactus ranks}

    In this section, we produce an open embedding $\AsigmaHilb_m$ into $\csigma_m$.

    To do this, we propose a new way of thinking about several familiar objects. Since this may be confusing,
    let us explain in broad strokes the content.
    We will be working with the following objects: a finite algebra $\cA$, modules $W_1^{\vee}$, \ldots , $W_d^{\vee}$ over $\cA$
    which are isomorphic to $\cA^{\oplus 1}$ itself, linear maps $\varphi_i\colon W_i\to V_i$ and a tensor $\Tppar{}\in \tconcspace$.
    These objects will appear in three configurations.
    \begin{enumerate}
        \item Evaluations tensors associated to an algebra. In this setting $W_1^{\vee} =  \ldots = W_d^{\vee} = \cA^{\oplus 1}$,
            $\Tppar{}$ is an evaluation tensor and the maps
            $\varphi_{\bullet}$ yield a restriction of $\Tppar{}$ to a tensor in $\tspace$.
        \item Line bundles and algebraic characterisation of spans. In this setting, we additionally have a
            finite scheme $Z$ with a map $Z\to \mathbb{P}V_1\times \ldots
            \times \mathbb{P}V_d$. Here $\cA = H^0(\OO_{Z})$,
            the morphisms $Z\to \mathbb{P}V_i$ induce surjections $\cA \otimes_{\kk} V_i^{\vee}\onto
            W_i^{\vee}$ onto 
            line bundles $W_i^{\vee}$, and the surjections yield restriction maps
            $W_i\to V_i$.
        \item Algebras from centroids. In this setting, $\Tppar{}$ is a $1$-generic tensor and the maps $\varphi_i\colon W_i\to V_i$ yield its restriction
            to a tensor in $\tspace$. Here, $\cA = \Cen{\Tppar{}}$, the $\cA$-module structure
            on $W_i^{\vee}$ comes from the centroid action, and the isomorphisms $W_i^{\vee}\simeq \cA^{\oplus 1}$ are consequences of
            $1$-genericity.
    \end{enumerate}
    
    \subsection{Evaluation tensors}\label{ssec:evaluation}
    Let $\cA$ be a finite $\kk$-algebra. By its \emph{(iterated) multiplication tensor} we mean 
    the tensor in $\cA^{\vee}\otimes \ldots \otimes \cA^{\vee}\otimes \cA$ associated to be the multiplication map
    \[
        \underbrace{\cA\times \ldots \times\cA}_{d-1}\to \cA.
    \]
    Let $\epsilon \colon \cA\to \kk$ be a
    functional. The \emph{(order $d$) evaluation tensor} associated to $(\cA, \epsilon)$ is
    an element of $\cA^{\vee}\otimes \ldots \otimes \cA^{\vee}$ given by the composed multilinear map
    \begin{equation}\label{eq:evaluationTensor}
        \begin{tikzcd}
            \eval\colon \underbrace{\cA\times  \ldots \times \cA}_d \ar[r, "\mathrm{mult.}"] & \cA \ar[r, "\epsilon"]
            & \kk.
        \end{tikzcd}
    \end{equation}
    Recall that a finite algebra $\cA$ is \emph{Gorenstein} if there exists a functional $\epsilon \in \cA^{\vee}$
    such that $\cA^{\vee} = \cA\cdot \epsilon$. Every such functional is called a \emph{dual generator} of $\cA$.

    \begin{proposition}\label{ref:evaluationTensors:prop}
        Let $\cA$ be a finite $\kk$-algebra. Let $d\geq 2$ be an integer. The following conditions are
        equivalent for a tensor $T$:
        \begin{enumerate}
            \item\label{it:evalOne} the tensor $T$ is isomorphic to an evaluation tensor
                associated to some $\epsilon\colon \cA\to \kk$,
            \item\label{it:evalTwo} the tensor $T$ is isomorphic to a
                multiplication tensor of a
                Gorenstein quotient algebra $\cA' = \cA/I$.
        \end{enumerate}
        In particular, if there exists a concise evaluation tensor $T\in \cA^{\vee}\otimes \ldots \otimes \cA^{\vee}$, then
        $\cA$ is Gorenstein.
    \end{proposition}
    \begin{proof}
            $\ref{it:evalOne}\implies \ref{it:evalTwo}$. Let $I = \left\{ a\in \cA\
            |\ \epsilon(\cA\cdot a) = 0 \right\}$. This is an ideal of $\cA$. Let
            $\cA' := \cA/I$.

            The functional $\epsilon$ vanishes on $I$, so it
            yields a functional $\epsilon'\colon \cA'\to \kk$ given by
            $\epsilon'(a + I) := \epsilon(a)$.
            Consider the map $\varphi\colon \cA'\to (\cA')^{\vee}$
            given by $\varphi(a')= a'\cdot \epsilon'$. If $\varphi(a') = 0$
            for some $a' = a + I\in \cA'$, then $\epsilon(\cA a) = 0$, so $a\in I$ and
            $a' = 0$. This
            shows that $\varphi$ is injective. The dimensions of $\cA'$,
            $(\cA')^{\vee}$ are the same, so $\varphi$ is bijective, hence $\cA'$
            is Gorenstein with dual generator $\epsilon'$.

            The evaluation tensor can be rewritten as a multilinear map
            \[
                T\colon \underbrace{\cA\times  \ldots \times \cA}_{d-1} \to \cA^{\vee}
            \]
            given by $T(a_1, \ldots ,a_{d-1}) := (a_1 \ldots a_{d-1})\cdot
            \epsilon$. Whenever at least one $a_j$ lies in $I$, then $T(a_1, \ldots
            ,a_{d-1}) = 0$, so $T$ is isomorphic to a multilinear map
            \[
                T'\colon \underbrace{\cA'\times  \ldots \times \cA'}_{d-1} \to (\cA')^{\vee}
            \]
            given by $T'(a_1', \ldots ,a_{d-1}') := (a_1' \ldots a_{d-1}')\cdot
            \epsilon'$. Composing with the isomorphism $\varphi^{-1}\colon
            (\cA')^{\vee} \to \cA'$, we obtain that $T'$ is isomorphic to the
            multiplication tensor in $\cA'$.

            $\ref{it:evalTwo}\implies \ref{it:evalOne}$. The algebra $\cA'$ is
            Gorenstein and finite, hence there exists a functional
            $\epsilon'\colon \cA' \to \kk$ such that $(\cA')^{\vee} = \cA'\cdot
            \epsilon'$. Let $\varphi\colon \cA'\to (\cA')^{\vee}$ be given by
            $\varphi(a') = a'\cdot \epsilon'$ and let $\epsilon\colon \cA\to
            \kk$ be the composition of $\cA\onto \cA'$ and $\epsilon'$. The proof
            above reverses and shows that the multiplication tensor in $\cA'$ is
            isomorphic to the evaluation tensor associated to $(\cA, \epsilon)$.

            To prove the final statement, observe that $T$ is isomorphic to a multiplication
            tensor in a Gorenstein quotient algebra $\cA' = \cA/I$ but by conciseness, we have
            $I = 0$.
    \end{proof}

    \begin{corollary}\label{ref:dualSocle:cor}
        Let $\cA$ be a finite Gorenstein algebra and $\epsilon \colon \cA\to \kk$ be a
        dual generator. The evaluation tensor associated to $(\cA, \epsilon)$ is
        isomorphic to the multiplication tensor of $\cA$. In particular, the
        tensors associated to different $\epsilon$ are all isomorphic.
    \end{corollary}
    \begin{proof}
        Follows from Proposition~\ref{ref:evaluationTensors:prop}.
    \end{proof}
    Corollary~\ref{ref:dualSocle:cor} was also obtained independently in~\cite{Chang_Gesmundo_Zuiddam}.

    \begin{remark}\label{ref:evaluationTensorsLineBundles:rem}
        One can slightly generalise the notion of evaluation tensor. Namely, let $\cA$ be a finite algebra
        and $W_1^{\vee}$, \ldots , $W_d^{\vee}$ be $\cA$-modules isomorphic to $\cA^{\oplus 1}$.
        For a functional $\epsilon \colon W_1^{\vee}\otimes_{\cA} \ldots \otimes_{\cA} W_d^{\vee}\to \kk$,
        we can form a generalised evaluation tensor
        \[
            \begin{tikzcd}
                W_1^{\vee}\otimes_{\kk} \ldots \otimes_{\kk} W_d^{\vee} \ar[r, two heads] & W_1^{\vee}\otimes_{\cA} \ldots \otimes_{\cA} W_d^{\vee}
                \ar[r, "\epsilon"] & \kk.
            \end{tikzcd}
        \]
        It corresponds to an element
        \[
            \epsilon \in (W_1^{\vee}\otimes_{\cA} \ldots \otimes_{\cA} W_d^{\vee})^{\vee}\subset \left( W_1^{\vee}\otimes_{\kk} \ldots \otimes_{\kk} W_d^{\vee} \right)^{\vee} = \tconcspace.
        \]
    \end{remark}

    \subsection{Line bundles and algebraic characterisation of spans}\label{ssec:lineBundlesAndSpans}

    Let $Z$ be a zero-dimensional scheme. We will use the words \emph{finite scheme} and \emph{zero-dimensional} scheme
    interchangeably, preferring the word finite. If $Z = \Spec(\cA)$, then the \emph{degree} of $Z$ is defined as $\dim_{\kk} \cA$.
    We will usually consider schemes of degree $\leq m$.

    Let $Z = \Spec(\cA)$ be a finite scheme with a morphism $Z\to \mathbb{P}V_1
    \times \ldots \times \mathbb{P}V_d$. To give such a morphism is the same as to give
    surjections $\cA \otimes_{\kk} V_i^{\vee} \onto W_i^{\vee}$ for $i=1, \ldots ,d$, where $W_i^{\vee}$ are
    line bundles on $Z$. We identify the line bundles with their global sections, so that $W_i^{\vee}$ becomes
    a vector space which is a $\cA$-module that is isomorphic to $\cA^{\oplus 1}$. In this way $W_i$ is also a $\cA$-module, which
    is, however, not necessarily isomorphic to $\cA^{\oplus 1}$. It may seem odd that $W_i^{\vee}$ plays a bigger role than $W_i$,
    but this convention is necessary to be compatible with unrestrictions, see~\S\ref{sec:abstractSecant}.

    The \emph{linear span} of $Z$,
    denoted $\spann{Z}$ is the smallest linear subspace $\spann{Z}\subseteq
    \tspace$ such that the image of $Z$ lies in $\mathbb{P}\spann{Z}\subseteq
    \mathbb{P}\tspace$. The linear span is usually defined only for closed embeddings.

    \begin{proposition}[Linear span from a tensor product]\label{ref:linearSpanAbstract:prop}
        Let $Z=\Spec(\cA)$ be a finite scheme with a morphism $Z\to \mathbb{P}V_1\times \ldots \times \mathbb{P}V_d$ and let
        $\cA\otimes_{\kk} V_i^{\vee}\onto W_i^{\vee}$ be the associated line bundles. Let $\varphi_i\colon W_i\to V_i$ be dual to
        $V_i^{\vee}\into \cA \otimes_{\kk} V_i^{\vee}\onto W_i^{\vee}$.
        The linear span of $Z$ is equal to the image of
        \[
            (W_1^{\vee}\otimes_{\cA} \ldots \otimes_{\cA}
            W_d^{\vee})^{\vee} \into W_1 \otimes_{\kk} \ldots \otimes_{\kk} W_d
            \to V_1\otimes_{\kk} \ldots \otimes_{\kk} V_d.
        \]
    \end{proposition}
    \begin{proof}
        The ideal of $\spann{Z}$ is given by the kernel of the restriction map
        \[
            V_1^{\vee} \otimes_{\kk} \ldots
            \otimes_{\kk} V_d^{\vee} = H^0\left(
            \OO_{\mathbb{P}V_1\times \ldots \times\mathbb{P}V_d}(1,1, \ldots ,1)
            \right) \to H^0(\OO_{Z}(1,1, \ldots ,1)).
        \]
        To compute it, we complete this diagram to
        \begin{equation}\label{eq:linearSpanDiagram}
            \begin{tikzcd}
                V_1^{\vee}\ar[d, "\varphi_1^{\vee}"]\ar[rr, phantom, "\otimes_{\kk} \ldots
                \otimes_{\kk}"] && V_d^{\vee} \ar[d, "\varphi_d^{\vee}"]\ar[r] & H^0(\OO_{Z}(1,1, \ldots ,1))\\
                W_1^{\vee} \ar[rr, phantom, "\otimes_{\kk}  \ldots \otimes_{\kk}"] && W_d^{\vee} \ar[r, two heads] &
                W_1^{\vee}\otimes_{\cA} \ldots \otimes_{\cA} W_d^{\vee} \ar[u, "\simeq"]
            \end{tikzcd}
        \end{equation}
        Therefore, the ideal is the kernel of the composition
        \[
            V_1^{\vee} \otimes_{\kk} \ldots \otimes_{\kk} V_d^{\vee} \to
            W_1^{\vee}\otimes_{\kk} \ldots \otimes_{\kk} W_d^{\vee} \onto
            W_1^{\vee}\otimes_{\cA} \ldots \otimes_{\cA} W_d^{\vee},
        \]
        so the span is the image of the dual map, as claimed.
    \end{proof}

    \begin{example}[``Universal'' span]\label{ex:universalLinearSpan}
        For any finite scheme $Z = \Spec(\cA)$ we can take line bundles $W_i^\vee := \cA$ 
        and vector spaces $V_i := \cA^\vee$.
        Multiplication maps $\cA \otimes_\kk V_i^\vee \onto W_i^\vee$ yield an embedding 
        $Z \into \PP V_1\times \ldots \times \PP V_d = \PP\cA^{\vee}\times \ldots \times \PP\cA^{\vee}$.
        The associated restrictions $\varphi_i \colon W_i \to V_i$
        are just the identity maps and the map 
        \[
            \cA \otimes_\kk \ldots \otimes_\kk \cA =
            W_1^{\vee}\otimes_{\kk} \ldots \otimes_{\kk} W_d^{\vee} \onto
            W_1^{\vee}\otimes_{\cA} \ldots \otimes_{\cA} W_d^{\vee} =
            \cA \otimes_\cA \ldots \otimes_\cA \cA =
            \cA
        \]
        is just the iterated multiplication in $\cA$. 
        By Proposition~\ref{ref:linearSpanAbstract:prop}, the linear span $\spann{Z}$ is the image of
        the dual map of $\cA \otimes_\kk \ldots \otimes_\kk \cA \xrightarrow{\mathrm{mult.}} \cA$, that is,
        the image of
        \[
            \cA^\vee \to \cA^\vee \otimes_\kk \ldots \otimes_\kk \cA^\vee,\quad
            (\epsilon \colon \cA \to \kk) \mapsto 
            (\eval \colon \cA \otimes_\kk \ldots \otimes_\kk \cA \xrightarrow{\mathrm{mult.}} \cA \xrightarrow{\epsilon} \kk).
        \]
        In words, tensors in $\spann{Z} \subset \cA^\vee \otimes \ldots \otimes \cA^\vee$ are exactly the 
        evaluation tensors associated to functionals $\epsilon \colon \cA \to \kk$.
    \end{example}
    
    To proceed, we introduce two nondegeneracy conditions on restrictions: regularity and begin jointly
    spanning. The latter will not be needed until much later, so the reader is free to ignore it for the time being.
    
    \begin{definition}\label{ref:regularMap:def}
        Let $Z = \Spec(\cA)$ be a finite scheme, $W^{\vee}$ be a line bundle on $\cA$, and $V$ be a vector space with
        a linear map $\varphi\colon W \to V$. We say that $\varphi$ is
        \emph{regular} if the linear map $\cA \otimes_\kk V^{\vee} \to W^{\vee}$ given by $(a,
        v^{*})\mapsto a\varphi^{\vee}(v^*)$ is surjective, so that it gives a morphism 
        $Z \to \PP V$.
    \end{definition}

    \begin{definition}\label{ref:regularAndJointlySpanning:def}
        Let $Z = \Spec(\cA)$ be a finite scheme, $W_1^{\vee}$,  \ldots , $W_{d}^{\vee}$ be line bundles on $\cA$
        and $V_1, \ldots ,V_d$ be vector spaces with linear maps $\varphi_i\colon W_i\to V_i$ for $i=1, \ldots ,d$.
        \begin{enumerate}
            \item We say that $\varphi_1, \ldots ,\varphi_d$ is a \emph{regular restriction} if every $\varphi_1, \ldots ,\varphi_d$
                is regular,
            \item We say that $\varphi_1, \ldots ,\varphi_d$ are \emph{jointly spanning} if the composition
                \[
                    (W_1^{\vee}\otimes_{\cA} \ldots \otimes_{\cA}
                    W_d^{\vee})^{\vee} \into W_1 \otimes_{\kk} \ldots \otimes_{\kk} W_d
                    \to V_1\otimes_{\kk} \ldots \otimes_{\kk} V_d
                \]
                is injective.
        \end{enumerate}
        
    \end{definition}
    It will be useful to rephrase the notion of jointly spanning in dual terms. We have the following diagram
    \begin{equation}\label{eq:jointSpanConcrete}
            \begin{tikzcd}
                V_1^{\vee}\ar[rr, phantom, "\otimes_{\kk} \ldots \otimes_{\kk}"]\ar[d, equal] &&[-0.5cm] V_d^{\vee}\ar[d, equal]\ar[r, "\varphi_{\bullet}^{\vee}"] & W_1^{\vee}\ar[d, "\simeq"] \ar[rr, phantom, "\otimes_{\kk}  \ldots \otimes_{\kk}"] &&[-0.5cm] W_d^{\vee}\ar[d, "\simeq"] \ar[r, two heads, "\mathrm{mult.}"] & W_1^{\vee}\otimes_{\cA} \ldots \otimes_{\cA} W_d^{\vee}\ar[d, "\simeq"]\\
                V_1^{\vee}\ar[rr, phantom, "\otimes_{\kk} \ldots \otimes_{\kk}"] &&[-0.5cm] V_d^{\vee}\ar[r] & \cA \ar[rr, phantom, "\otimes_{\kk}  \ldots \otimes_{\kk}"] &&[-0.5cm] \cA \ar[r, two heads, "\mathrm{mult.}"] & \cA
            \end{tikzcd}
        \end{equation}
        and the restrictions $\varphi_{\bullet}$ are jointly spanning if the resulting linear map $\tspace^{\vee} \to \cA$ is surjective.

    \begin{lemma}\label{ref:regularity:lem}
        If $\varphi_1, \ldots ,\varphi_d$ are jointly spanning, then every $\varphi_1$, \ldots , $\varphi_d$ is regular.
        If $(\varphi_i)_{i\in I}$ are jointly spanning for some $I\subseteq \{1, \ldots ,d\}$ and $\varphi_j$ is regular,
        for some $j\in \left\{ 1, \ldots ,d \right\}$, then $(\varphi_i)_{i\in I\cup \{j\}}$ are also jointly spanning.
    \end{lemma}
    \begin{proof}
        Suppose that $\varphi_i$ is not regular for some $i$. This is
        equivalent to saying that the ideal of $\cA$ generated by the image
        of $\varphi_i^{\vee}$ is not the whole $\cA$. In other words, $\im
        \varphi_i^{\vee}$ lies in an ideal $I\subsetneq \cA$. Then also the
        image of the first map of~\eqref{eq:jointSpanConcrete} is contained in $I$,
        hence the map is not surjective.

        Suppose now that $(\varphi_i)_{i\in I}$ are jointly spanning. In the
        formalism of~\eqref{eq:jointSpanConcrete}, this means that $\prod_{i\in
        I}\varphi_i^{\vee}(V_i^{\vee}) = \cA$. Since $\varphi_j$ is regular, we have
        $\cA\cdot \varphi_j^{\vee}(V_{j}^{\vee}) = \cA$. Putting this together, we have
        \[
            \prod_{i\in I\cup \{j\}} \varphi_i^{\vee}(V_i^{\vee}) = \left(\prod_{i\in I} \varphi_i^{\vee}(V_i^{\vee})\right)\cdot \varphi_j^{\vee}(V_j^{\vee}) = 
            \cA\cdot \varphi_j^{\vee}(V_j^{\vee}) = \cA,
        \]
        so the claim follows by~\eqref{eq:jointSpanConcrete} again.
    \end{proof}

    \begin{example}[Not jointly spanning regular maps]\label{ex:jointSurjectivity}
        Let $\cA = \kk[x]/(x^{5})$.
        Let $W_1^{\vee} = W_2^{\vee} = W_3^{\vee} = \cA$. Let $V_1^{\vee} = V_2^{\vee} = V_3^{\vee} = \spann{1, x}$ and let $V_i^{\vee} \to W_i^{\vee}$
        be the inclusion maps. They correspond to regular restrictions $W_i \to V_i$. The sequence~\eqref{eq:jointSpanConcrete} is
        \[
            \begin{tikzcd}
                \spann{1, x} \otimes_\kk \spann{1, x}\otimes_\kk  \spann{1, x} \ar[r, hook] & \cA \otimes_\kk \cA \otimes_\kk \cA \ar[r, two heads, "\mathrm{mult.}"] & \cA
            \end{tikzcd}
        \]
        which is not surjective, so the restrictions $(W_i\to V_i)_{i=1,2,3}$ are not jointly spanning.
    \end{example}
    
    Having dealt with the conditions on restrictions, we proceed to our main statements.

    \begin{theorem}[Cactus Apolarity Lemma]\label{ref:spanChar:thm}
        Let $T$ in $\tspace$ be a tensor and $\cA$ be a finite $\kk$-algebra. Let $Z = \Spec(\cA)$. The following are equivalent:
        \begin{enumerate}
            \item\label{it:spansGeo} there exists a morphism $Z \to \mathbb{P}V_1\times \ldots
                \times \mathbb{P}V_d$ such that $T$ lies in the linear span
                $\spann{Z}$,
            \item\label{it:spansAlg} there exist regular linear maps $\varphi_i\colon \cA^{\vee} \to V_i$ 
                and a functional $\epsilon \colon \cA \to \kk$ such that the tensor $T$ is associated to the
                multilinear map
                \[
                    \begin{tikzcd}
                        V_1^{\vee} \ar[d, "\varphi_1^{\vee}"]\ar[r, phantom, "\times
                        \ldots \times"] & V_{d}^{\vee}\ar[d,
                        "\varphi_d^{\vee}"] \ar[rr] && \kk\ar[d, equal]\\
                        \cA \ar[r, phantom, "\times \ldots \times"] & \cA \ar[r,
                        "\mathrm{mult.}"]& \cA \ar[r, "\epsilon"] &
                        \kk
                    \end{tikzcd}
                \]
        \end{enumerate}
    \end{theorem}

    \begin{proof}
        Assume~\ref{it:spansGeo} and apply Proposition~\ref{ref:linearSpanAbstract:prop}. Since $T\in \spann{Z}$, we have
        an element $\epsilon\in(W_1^{\vee}\otimes_{\cA} \ldots \otimes_{\cA} W_d^{\vee})^{\vee}$ yielding $T$.
        In terms of Diagram~\eqref{eq:linearSpanDiagram}, the tensor $T$ corresponds to a linear map
        \[
            V_1^{\vee}\otimes_\kk \ldots \otimes_\kk V_d^{\vee}\to
            W_1^{\vee}\otimes_{\kk} \ldots \otimes_{\kk}W_d^{\vee} \onto
            W_1^{\vee}\otimes_{\cA}  \ldots \otimes_{\cA} W_d^{\vee} \xrightarrow{\epsilon} \kk.
        \]
        Fix isomorphisms $W_i^{\vee}\to \cA$ for each $i=1, \ldots ,d$. Then the above linear map becomes
        \[
            V_1^{\vee}\otimes_\kk \ldots \otimes_\kk V_d^{\vee}\to
            \cA\otimes_{\kk} \ldots \otimes_{\kk}\cA \xrightarrow{\mathrm{mult.}}
            \cA \xrightarrow{\epsilon} \kk,
        \]
        as claimed. This yields~\ref{it:spansAlg}. The argument can be reversed.
    \end{proof}

    \begin{corollary}[Algebraic characterisation of
        spans]\label{ref:spanCharPopular:cor}
        Let $T$ in $\tspace$ be a tensor and $\cA$ be a finite $\kk$-algebra.
        Let $Z = \Spec(\cA)$. The following conditions are equivalent:
        \begin{enumerate}
            \item there exists a morphism $Z\to \mathbb{P}V_1\times \ldots
                \times \mathbb{P}V_d$ such that $T\in \spann{Z}$,
            \item there exists a functional $\epsilon\colon \cA\to \kk$ such that $T$ is a
                regular restriction (Definition~\ref{ref:regularAndJointlySpanning:def}) of the evaluation tensor of $(\cA, \epsilon)$,
            \item the tensor $T$ is a regular restriction of the multiplication tensor
                in a quotient algebra $\cA' = \cA/I$ which is Gorenstein.
        \end{enumerate}
    \end{corollary}
    \begin{proof}
        The equivalence of first two conditions is given in
        Theorem~\ref{ref:spanChar:thm}, while the equivalence of the latter
        two is given in Proposition~\ref{ref:evaluationTensors:prop}.
    \end{proof}

    \begin{proof}[Proof of Theorem~\ref{ref:restriction:introthm}]
        The equivalence $\ref{it:cactusrankAlg}\Leftrightarrow\ref{it:cactusrankSpan}$ follows from
        Corollary~\ref{ref:spanCharPopular:cor}.
        The implication $\ref{it:cactusrankOne}\implies\ref{it:cactusrankSpan}$ is purely formal.
        To obtain $\ref{it:cactusrankSpan}\implies\ref{it:cactusrankOne}$, take a map $Z\to \mathbb{P}V_1\times \ldots \times \mathbb{P}V_d$
        and consider its scheme-theoretic image $Z'\into \mathbb{P}V_1\times \ldots \times \mathbb{P}V_d$. By definitions,
        the spans of $Z$ and $Z'$ coincide. The argument can be given also in purely algebraic terms. In the setting of
        Theorem~\ref{ref:spanChar:thm}, let $\mathcal{B}\subseteq \cA$ be the subalgebra generated by the images of $\varphi_1^{\vee}$,
        \ldots, $\varphi_d^{\vee}$. One obtains that $T$ is a restriction
        of the evaluation tensor associated to $(\mathcal{B},
        \epsilon|_{\mathcal{B}})$ and that $\Spec(\mathcal{B})\to
        \mathbb{P}V_1\times \ldots \times \mathbb{P}V_d$ is a closed embedding.
    \end{proof}

    \begin{remark}\label{ref:smoothableApolarityLemma:rem}
        There is an analogue of Theorem~\ref{ref:restriction:introthm} for smoothable rank.
        It is not immediate, since the smoothability of $Z$ and $Z'$ given in the proof needs
        to be taken into account. The result
        is obtained by Buczy{\'n}ski (unpublished) and we do not mention it here further.
    \end{remark}
    
    \subsection{Algebras from centroids}\label{ssec:genericityConditions}

        Above, we started from a finite scheme $Z = \Spec(\cA)$ and obtained tensors $T\in\spann{Z}\subseteq \tspace$. In
        this section we reverse this procedure. The reversal is possible thanks to the theory of centroids, as developed in~\cite{JLP}.
        Namely, for a $1$-generic (defined below) centroid-abundant tensor $\Tppar{}\in \tconcspace$,
        we will obtain $\cA$ as $\Cen{\Tppar{}}$.

    The following theorem in the case of three factors was proved in~\cite[Proposition~5.3]{JLP}.
    The proof generalises immediately as soon as one takes the correct notion of ``full rank''.
    \begin{proposition}[genericity conditions]\label{ref:genericity:prop}
        Let $d\geq 3$ and $\dim_{\kk} W_{\bullet} = m$.
        Let $\Tppar{}\in W_1 \otimes  \ldots \otimes W_d$ be a concise tensor with $\dim_{\kk} \Cen{\Tppar{}} \geq
        m$. Fix an index $1\leq i\leq d$ and let $\alpha\in W_i^{\vee}$ be a
        functional. The following are equivalent:
        \begin{enumerate}
            \item\label{it:generityRank} for \emph{some} index $j\neq i$, the contraction
                \[
                    \Tppar{}(\alpha)\colon W_{j}^{\vee} \to W_1\otimes  \ldots \otimes
                    W_{\hat{i}}\otimes  \ldots \otimes W_{\hat{j}} \otimes
                    \ldots \otimes W_d
                \]
                is injective,
            \item\label{it:generityRankEvery} for \emph{every} index $j\neq i$, the contraction
                \[
                    \Tppar{}(\alpha)\colon W_{j}^{\vee} \to W_1\otimes  \ldots \otimes
                    W_{\hat{i}}\otimes  \ldots \otimes W_{\hat{j}} \otimes
                    \ldots \otimes W_d
                \]
                is injective.
            \item\label{it:generityGenerator} the $\Cen{\Tppar{}}$-module $W_i^\vee$ is generated by the element
                $\alpha$.
        \end{enumerate}
        If these hold, then $\dim_{\kk} \Cen{\Tppar{}} = m$ and the
        $\Cen{\Tppar{}}$-module $W_i^{\vee}$ is isomorphic to $\Cen{\Tppar{}}$. In this case,
        we say that $\Tppar{}$ is \emph{$1_{W_i}$-generic via $\alpha$} or that $\Tppar{}$
        is \emph{generic at the $i$-th coordinate via $\alpha$}.
    \end{proposition}

    We say that a tensor $\Tppar{}$ as in Proposition~\ref{ref:genericity:prop} is \emph{$1$-generic} if it
    is generic at every of its $d$ coordinates. Being $1$-generic is an open
    condition among concise tensors with $\dim_{\kk} \Cen{\Tppar{}} \geq m$.
    Being $1$-generic is tightly connected to the algebraic characterisation of spans from~\S\ref{ssec:lineBundlesAndSpans},
    as we now show.

    \begin{corollary}\label{ref:1generic:cor}
        Let $d\geq 3$ and $\dim_{\kk} W_{\bullet} = m$. Let $\Tppar{}\in W_1 \otimes  \ldots
        \otimes W_d$ be a concise tensor with $\dim_{\kk} \Cen{\Tppar{}} \geq m$.
        The following are equivalent:
        \begin{enumerate}
            \item $\Tppar{}$ is $1$-generic,
            \item $\Tppar{}$ is isomorphic to an evaluation tensor associated to some $(\cA, \epsilon)$.
        \end{enumerate}
        If the above condition hold, then $\cA$ is Gorenstein, isomorphic to $\Cen{\Tppar{}}$, the functional $\epsilon$ is its dual generator, the line bundles $W_1^{\vee}, \ldots ,W_{d}^{\vee}$ yield
        an embedding of $\Spec(\Cen{\Tppar{}})$ into $\mathbb{P}W_1\times \ldots \times \mathbb{P}W_d$, and $\Tppar{}$ lies in $\spann{\Spec(\Cen{\Tppar{}})} = \Cen{\Tppar{}}\cdot \Tppar{}$.
    \end{corollary}
    \begin{proof}
        Assume that $\Tppar{}$ is $1$-generic.
        Let $\cA = \Cen{\Tppar{}}$. We use Proposition~\ref{ref:genericity:prop}.
        For every $i=1, \ldots ,d$, choose a generator $\alpha_i\in W_i^{\vee}$ of the $\cA$-module $W_i^\vee$.
        The map $\sigma_i\colon \cA\to W_i^{\vee}$ given by $\sigma_i(a) = a\alpha_i$ is an isomorphism.
        View $\Tppar{}$ as the contraction $\Tppar{}\colon W_1^{\vee}\otimes  \ldots \otimes W_d^{\vee}\to \kk$.
        By definition of the centroid, this map factors as
        \[
            \begin{tikzcd}
                W_1^{\vee}\otimes_{\kk} \ldots \otimes_{\kk} W_d^{\vee} \ar[r, two heads] & W_1^{\vee}\otimes_{\cA} \ldots \otimes_{\cA} W_d^{\vee}
                \ar[r, "\epsilon"] & \kk
            \end{tikzcd}
        \]
        where $\epsilon$ is some functional. Applying the isomorphisms $\sigma_1, \ldots ,\sigma_d$, we get
        an isomorphism $W_1^{\vee}\otimes_{\cA} \ldots \otimes_{\cA} W_d^{\vee}\simeq \cA$ and so $\Tppar{}$
        comes from an evaluation tensor associated to $\epsilon$, see also Remark~\ref{ref:evaluationTensorsLineBundles:rem}.
        Conversely, a concise evaluation tensor $\Tppar{}$ is by definition $1$-generic.

        To prove the final statements, observe that by
        Proposition~\ref{ref:evaluationTensors:prop} the algebra $\cA$ is Gorenstein
        and $\epsilon$ is a dual generator of $\cA$.
        The line bundles $W_1^{\vee}$, \ldots , $W_d^{\vee}$ on $\Spec(\cA)$ yield an embedding
        $\Spec(\cA)\to \mathbb{P}W_1\times \ldots \times\mathbb{P}W_d$.
        As in Example~\ref{ex:universalLinearSpan}, the linear span of this
        embedding coincides with $\cA^{\vee} = \cA\cdot \epsilon$. In particular, $1\cdot \epsilon$ is
        the element corresponding to $\Tppar{}$.
    \end{proof}
    For a tensor $\Tppar{}$ as above, we can consider the very same nondegeneracy conditions on restrictions as in Definition~\ref{ref:regularAndJointlySpanning:def},
    by applying these definitions to $\cA = \Cen{\Tppar{}}$ and line bundles $W_1^\vee, \dots, W_d^\vee$.

    \begin{lemma}\label{ref:consiseness1generic:lem}
        Let $\dim W_{\bullet} = m$ and let $\Tppar{}\in \tconcspace$ be a $1$-generic tensor (so that in particular $\dim_{\kk} \Cen{\Tppar{}}\geq m$)
        and let $\varphi_1, \ldots ,\varphi_d$ be regular restrictions.
        Fix a subset $I\subseteq \{1, \ldots ,d\}$ and let $J$ be its complement. Assume that $(\varphi_i)_{i\in I}$ are jointly spanning and let
        \[
            T = \varphi_I(\Tppar{})\in \bigotimes_{i\in I}V_{i} \otimes \bigotimes_{j\in J} W_j
        \]
        be the restriction. Then $T$ is $W_j$-concise for every $j\in J$.
    \end{lemma}
    \begin{proof}
        Pick an index $j \in J$. Up to permuting coordinates, we can and do assume that $j=1$.
        Let $\cA = \Cen{\Tppar{}}$. By Corollary~\ref{ref:1generic:cor}, we may identify $W_j^{\vee}$ with $\cA$ for every $j\in J$
        and $\Tppar{}$ with the evaluation tensor associated to a dual generator $\epsilon\in \cA^{\vee}$.
        By Lemma~\ref{ref:regularity:lem}, the restrictions $\varphi_2, \ldots ,\varphi_d$ are jointly spanning.
        Since we are analysing $j=1$, we can restrict further to $W_1\otimes V_{2, \ldots ,d}$, that is, assume that $I = \left\{ 2, \ldots ,d \right\}$.
        
        If $T$ is not $W_1$-concise, then there exists an element $r\in \cA = W_1^\vee$ such that $T(r) = 0$.
        This implies that $T(r, v_2^*, \ldots , v_d^*) = 0$ for every $v_i^*\in V_i^{\vee}$, $i=2, \ldots ,d$. Since $T$ is an evaluation
        tensor, this translates to
        \[
            0 = \epsilon\left( r\cdot \varphi_2^{\vee}(V_2^{\vee})\cdot  \ldots \cdot \varphi_d^{\vee}(V_d^{\vee}) \right) = \epsilon(r\cdot \cA).
        \]
        In the last equality, we used~\eqref{eq:jointSpanConcrete}. The equality $\epsilon(r\cdot \cA)=0$ means that $r\epsilon = 0$. 
        But $\epsilon$ is a dual generator, so $r = 0$, which concludes the proof.
    \end{proof}

    \subsection{Abstract secant varieties}\label{sec:abstractSecant}

    The aim of this section is to show that the abstract secant variety
    (defined without closure) is an open subvariety of the concise secant
    variety. This requires a shift of perspective: traditionally, the points of
    an abstract secant (or cactus) variety are viewed as pairs $Z = \Spec(\cA)$, $T$ such that
    \[
        T\in \spann{Z}.
    \]
    This perspective is very useful under positivity assumptions, when all concise tensors become $1$-generic (see~\cite{Buczyska_Buczynski__cactus}).
    Without these assumptions, it becomes very hard to interpret the above away from the $1$-generic locus. We argue that one should rewrite the above as
    \[
        T \leq \eval,
    \]
    where $\eval$ is an evaluation tensor on $\cA$. For almost all applications, we can assume
    that $\cA$ is Gorenstein and replace $\eval$ by the (iterated) multiplication tensor on $\cA$.

    \bigskip

    We now discuss the spaces of interest.
        Let $\csigma_m$ be the concise secant variety
        with its family of tensors $\Tppar{}$. We define the \emph{tame locus} $\csigmaOneGen_m$
        of $\csigma_m$ as the locus given by the conditions:
        \begin{enumerate}
            \item\label{it:1generic} the restriction of the universal tensor $\Tppar{}$ to $\csigmaOneGen_m$ is $1$-generic, see Definition~\ref{ref:1genericityInFamilies:def},
            \item\label{it:jointSpan} the restriction map $\varphi_1$ is regular and the maps $\varphi_{\mathbf{2}}, \ldots ,\varphi_d$ are jointly spanning for $\cA = \Cen{\Tppar{}}$, see Definition~\ref{ref:regularAndJointlySpanning:def}; note that the indexing starts from $2$.
        \end{enumerate}

        Classically~\cite[p.439]{Chiantini_Ciliberto__On_the_concept}, the open part of the abstract secant variety
        consists of pairs $\{x_1, \ldots ,x_m\}$, $[T]$, where
        $x_1, \ldots ,x_m$ are general points of $\mathbb{P}V_1\times \ldots \times \mathbb{P}V_d$ and $T\in \spann{x_1, \ldots ,x_m}$
        is general. Here, we work with a slightly different locus, defined as follows.
        Consider the reduced subscheme $\AsigmaHilb_m$ of $\Hilb_m^{\sm}(\mathbb{P}V_1\times \ldots \times \mathbb{P}V_d) \times \mathbb{P}(\tspace)$,
        whose $\kk$-points are pairs $(Z, [T])$ where
        \begin{enumerate}
            \item $[Z] = [\Spec(\cA)]\in \Hilb_m^{\sm}(\mathbb{P}V_1\times \ldots
                \times \mathbb{P}V_d)$ is a smoothable Gorenstein degree $m$ subscheme whose 
                linear span in $\mathbb{P}(V_{\mathbf{2}, \ldots ,d})$ is isomorphic to $\mathbb{P}^{m-1}$; note that the
                indexing starts from $2$.
            \item $T\in \spann{Z} \subset \PP(\tspace)$ is a tensor, which is a restriction of the evaluation tensor on a dual generator
                $\epsilon\colon \cA\to \kk$.
        \end{enumerate}
        %
        %
        %

        The main result of this section, Theorem~\ref{ref:abstractInsideConcise:thm} below, gives an isomorphism of $\csigmaOneGen_m$ with $\AsigmaHilb_m$.
        Before we go to the formal proof, we discuss what happens on $\kk$-points.

        \subsubsection{Comparison of $\kk$-points}\label{ssec:kpointsAbstractSecant}

        Take a $\kk$-point of $\csigmaOneGen_m$. By Lemma~\ref{ref:regularity:lem}, the maps $\varphi_1, \ldots ,\varphi_d$
        are jointly spanning, in particular regular. Let $\cA = \Cen{\Tppar{}}$ and $Z = \Spec(\cA)$.
        Since $\varphi_2, \ldots ,\varphi_d$ are jointly spanning,
        by~\eqref{eq:jointSpanConcrete}, they yield a closed embedding $Z\into
        \mathbb{P}V_2\times \ldots \times \mathbb{P}V_d$ with $\spann{Z}$ isomorphic to $\mathbb{P}^{m-1}$.
        Since $\varphi_1, \ldots ,\varphi_d$, are jointly spanning, they yield a closed embedding $Z\into
        \mathbb{P}V_1\times \ldots \times \mathbb{P}V_d$, so a point $[Z]\in \Hilb_d(\mathbb{P}V_1\times \ldots \times\mathbb{P}V_d)$.
        The tensor $\Tppar{}$ has minimal border rank, so $\cA$ is smoothable, hence $[Z]$ lies in $\Hilb_d^{\sm}$.
        Let $[T] = \varphi_{1, \ldots ,d}(\Tppar{})\in \mathbb{P}\tspace$ be the restriction.
        Since $\Tppar{}$ is $1$-generic, by Proposition~\ref{ref:1generic:cor}, it is an evaluation tensor associated to
        a dual generator $\epsilon \in \cA^{\vee}$, hence $T$ is a restriction of such a tensor. We obtain a point of $\AsigmaHilb_m$.

        Now, take a $\kk$-point $(Z, [T])$ in $\AsigmaHilb_m$. The linear span of $Z$ in $\mathbb{P}V_{2, \ldots ,d}$ has maximal dimension,
        so also the linear span of $Z$ in $\mathbb{P}\tspace$ has maximal dimension. Let $\cA$ be an algebra such that $Z = \Spec(\cA)$.
        Let $W_1^{\vee}$, \ldots , $W_d^{\vee}$ be the line bundles coming from the embedding $Z\into \mathbb{P}V_1\times \ldots \times \mathbb{P}V_d$ and
        $\varphi_i\colon W_i\to V_i$ be the restriction maps for $i=1, \ldots ,d$.
        Twist $\varphi_1$ by $\OO_{[T]}(-1)$.
        Comparing Proposition~\ref{ref:linearSpanAbstract:prop} and Definition~\ref{ref:regularAndJointlySpanning:def}, we see that
        $\varphi_1, \ldots ,\varphi_d$ are jointly spanning, so there is a \emph{unique} evaluation tensor $\Tppar{}\in W_1\otimes \ldots \otimes W_d$
        such that $\varphi_{1, \ldots ,d}(\Tppar{}) = [T]$.

        Analogously, the maps $\varphi_2, \ldots ,\varphi_d$ are jointly spanning.
        By Lemma~\ref{ref:consiseness1generic:lem}, for every $1\leq i\leq d$, the restriction
        \[
            \varphi_{i+1, \ldots ,d}(\Tppar{}) \in W_{1, \ldots ,i}\otimes V_{i+1,  \ldots , d}
        \]
        is $W_1$-, \ldots , $W_i$-concise. These, essentially speaking, yield a point of $\csigmaOneGen_m$.
        To be more precise, the contraction of the $W_1$-concise tensor $\varphi_{2, \ldots ,d}(\Tppar{})$ yields
        a subspace of $V_{2, \ldots ,d}$, which we call $U_1^{\vee}$ and an isomorphism $W_1^{\vee}\to U_1^{\vee}$. This isomorphism
        allows us to interpret $\varphi_{3, \ldots ,d}(\Tppar{})$ as a tensor $\Tppar{1}$ in $U_1\otimes W_2\otimes V_{3, \ldots ,d}$.
        Its contraction with respect to the second coordinate yields a subspace $U_2^{\vee}\subseteq U_1\otimes V_{3, \ldots ,d}$
        and an isomorphism $W_2^{\vee}\to U_2^{\vee}$. The subspace yields a tensor $\Tppar{1,2}$. Continuing, we obtain tensors all the way up to
        $\Tppar{1, \ldots ,d}$ which is isomorphic to $\Tppar{}$, in particular it is $1$-generic.
        Arguing as above, we see that $\varphi_2, \ldots ,\varphi_d$ are jointly spanning, so we obtain a $\kk$-point of $\csigmaOneGen_m$.

        \begin{example}
            Above, it is not enough to assume that $\varphi_1, \ldots ,\varphi_d$ are jointly surjective, the condition on $\varphi_2, \ldots ,\varphi_d$
            is necessary. Indeed, consider the situation from Example~\ref{ex:jointSurjectivity}. Namely, take $d = 4$, $\cA = \kk[x]/(x^5)$,
            $W_1^{\vee} =  \ldots = W_4^{\vee} = \cA$ and $V_i^{\vee}\to
            W_i^{\vee}$ is the inclusion of $\spann{1, x}$ and $\varphi_1 =
            \ldots = \varphi_4\colon W_i\to V_i$ be the associated
            restriction maps.
            Let $\epsilon = (x^4)^*$ be the functional in the monomial basis of
            $\cA$, and $\Tppar{}\in W_{1, \ldots ,4}$ be the associated tensor.
            Then the restriction $\varphi_{2, \ldots ,4}(\Tppar{})\in W_1\otimes V_{2, \ldots ,4}$ is not $W_1$-concise.
            Indeed, the associated contraction $W_1^{\vee} \to V_{2, \ldots , 4}$ sends $x^4$ to zero.
        \end{example}

        \subsubsection{$S$-points and comparison of functors}

        To prove Theorem~\ref{ref:abstractInsideConcise:thm}, we repeat the argument given above on $\kk$-points,
        but using $S$-points, where $S$ is a $\kk$-scheme.
        To even formulate Theorem~\ref{ref:abstractInsideConcise:thm}, we need to have openness of $\csigmaOneGen_m$ inside $\csigma_m$.

    We begin by generalising the notions from~\S\ref{ssec:lineBundlesAndSpans} to the setting of families of tensors.
        \begin{definition}\label{ref:1genericityInFamilies:def}
            Let $\WW_1, \ldots, \WW_d$ be rank $m$ vector bundles on a scheme $X$ and let
            \[
                T\colon \LL \into \WW_1\otimes \ldots \otimes \WW_d,
            \]
            be a concise
            (Definition~\ref{ref:concise:def}) family of tensors.
            For a fixed $1\leq i\leq d$, we say that $T$ is a family of \emph{$1_{\WW_i}$-generic tensors} if for every point $x\in
            X$, the tensor $T|_x$ is $1_{\WW_i}$-generic in the sense of
            Proposition~\ref{ref:genericity:prop}. We say that $T$ is a family of $1$-generic tensors if
            every $T|_x$ is $1$-generic.
        \end{definition}

        \begin{proposition}\label{ref:centroidsInFamilies:prop}
            Let $T$ be a family of $1$-generic tensors on a scheme $X$. Assume
            that $T$ is centroid-abundant
            (Definition~\ref{ref:secantLocus:def}). Then there is a subbundle
            \[
                \Cen{T}\into \End(\WW_1)\times \ldots \times\End(\WW_d)
            \]
            of commutative subalgebras such that $(\Cen{T})|_x \simeq \Cen{T|_x}$ for every $x\in X$. We call
            it the \emph{centroid} of $T$ or \emph{the family of centroids}.
            The vector bundles $\WW_1^{\vee}$, \ldots ,$\WW_d^{\vee}$ on $X$ are $\Spec(\Cen{T})$-line bundles.
            There is also a subbundle $\Cen{T}\cdot T\into \WW_1\otimes \ldots
            \otimes \WW_d$, which satisfies $(\Cen{T}\cdot T)|_x \simeq
            (\Cen{T|_x})\cdot T|_{x}$ for every $x\in X$.

        \end{proposition}
        \begin{proof}
            Define $\Cen{T}$ as the kernel of the map of vector bundles
            \[
                \End(\WW_1)\times \ldots \times\End(\WW_d)\to \prod_{i=2}^{d} \WW_1\otimes \ldots \otimes \WW_d
            \]
            given by $(X_1, \ldots ,X_d)\mapsto (X_1\circ T - X_i\circ T)_{i=2, \ldots ,d}$.
            We claim that this map has constant rank. 
            By assumption, the kernel at every point is at least $m$-dimensional.
            By Proposition~\ref{ref:genericity:prop}, at every point the centroid is exactly $m$-dimensional,
            hence the map has constant rank, equal to $d\cdot m^2 - m$.
            By the constancy of rank, the kernel of this map is
            a vector subbundle which commutes with base change. In particular,
            $\Cen{T}$ satisfies $(\Cen{T})|_x \simeq \Cen{T|_x}$ for every
            $x\in X$. Similarly, we take
            the map $\Cen{T} \to \WW_1\otimes \ldots \otimes \WW_d$ given by $\Cen{T}\ni r\mapsto r\cdot T$. It is a map of vector bundles and it is
            injective on each fiber, hence it is a subbundle and we obtain $\Cen{T}\cdot T$.

            Finally, each vector bundle $\WW_i$ comes with a $\Cen{T}$-action, hence $\WW_i$ and $\WW_i^{\vee}$ are $\Cen{T}$-modules. For every point $x\in X$,
            by Proposition~\ref{ref:genericity:prop}, after perhaps shrinking $X$ to a neighbourhood of $x$, we have an element $\alpha\in \WW_i^{\vee}$
            such that the map $\Cen{T}\to \WW_i^{\vee}$ given by $r\mapsto r\cdot \alpha$ is surjective. Both $\Cen{T}$ and $\WW_i^{\vee}$ are rank $m$ vector
            bundles on $X$, so this map is an isomorphism. This means that $\WW_i^{\vee}$ is a line bundle on $\Spec_X(\Cen{T})\to X$.
        \end{proof}

        \begin{proposition}[relative spans in the ``good case'']\label{ref:linearSpanAbstractRelative:prop}
            Let $\ZZ\to S$ be a finite flat family with a map $\ZZ\to S \times \mathbb{P}V_1 \times \ldots \times \mathbb{P}V_d$.
            Let $\varphi_i\colon \WW_i\to \OO_S\otimes V_i$ be the associated maps of vector bundles on $\OO_S$ and assume that $\varphi_1, \ldots ,\varphi_d$ are jointly spanning
            at every point of $S$. Then the composition
            \begin{equation}\label{eq:relativeSpanEq}
                (\WW_1^{\vee}\otimes_{\OO_{\ZZ}} \ldots \otimes_{\OO_{\ZZ}}
                \WW_d^{\vee})^{\vee} \into \WW_1 \otimes_{\OO_S} \ldots \otimes_{\OO_S} \WW_d
                \to \VV_1\otimes_{\OO_S} \ldots \otimes_{\OO_S} \VV_d
            \end{equation}
            is a vector subbundle and for every $s\in S$ and its restriction yields the linear span $\spann{\ZZ|_s}$.
            This subbundle is called the \emph{relative linear span} of $\ZZ$.
        \end{proposition}
        \begin{proof}
            Since $\varphi_1, \ldots ,\varphi_d$ are jointly spanning, the restriction of~\eqref{eq:relativeSpanEq}
            to every point is injective. Hence~\eqref{eq:relativeSpanEq} is a vector subbundle. Its connection with
            linear spans is proved in Proposition~\ref{ref:linearSpanAbstract:prop}, see also~\eqref{eq:jointSpanConcrete}.
        \end{proof}

        \begin{corollary}[Openness of $1$-generic locus]\label{ref:opennessOf1generic:cor}
            The locus $\csigmaOneGen_m$ is open inside $\csigma_m$.    
        \end{corollary}
        \begin{proof}
            Being $1$-generic is open thanks to Proposition~\ref{ref:genericity:prop}.
            On the $1$-generic locus, there is a flat family of centroids
            (Proposition~\ref{ref:centroidsInFamilies:prop}). Using this family, we obtain a family of linear maps as in
            Definition~\ref{ref:regularAndJointlySpanning:def}. Being injective
            is open in such a family.
        \end{proof}
        We are now in position to embed the abstract secant variety $\AsigmaHilb_m$ into $\csigma_m$.
        %


    \begin{theorem}[Concise secant compactifies the abstract one]\label{ref:abstractInsideConcise:thm}
        Let $\csigma_m$ be the concise secant variety (Definition~\ref{ref:conciseSecant:def}).
        Its tame locus $\csigmaOneGen_m$ is isomorphic to the abstract secant variety $\AsigmaHilb_m$. 
    \end{theorem}

    \begin{proof}

        To prove that $\csigmaOneGen_m$ and $\AsigmaHilb_m$ are isomorphic, we produce maps in both directions 
        on $S$-points for an any $\kk$-scheme $S$. This generalises the comparison for $S = \Spec(\kk)$ from~\S\ref{ssec:kpointsAbstractSecant}.

        Let $\Tppar{}$ be the universal tensor on $\csigmaOneGen_m$.
        Take an $S$-point of $\csigmaOneGen_m$, that is, a map $f\colon S\to \csigmaOneGen_m$.
        Let $\cA := \Cen{f^*\Tppar{}}$ be the centroid of $f^*\Tppar{}$ as in
        Proposition~\ref{ref:centroidsInFamilies:prop}.
        Let $g\colon \ZZ = \Spec_S(\cA) \to S$. This is a finite flat family of degree $m$ over $S$ and it is Gorenstein,
        that is, has Gorenstein fibres.

        Let $\varphi_1, \ldots ,\varphi_d$ be the restriction maps obtained from Proposition~\ref{ref:functorOfUnrestrictions:prop},
        so that $\varphi_1\colon U_1\to f^*\OO(1) \otimes V_1$ and $\varphi_i\colon U_i\to \OO_S\otimes V_i$ for $i=2, \ldots ,d$.
        The bundles $U_1^{\vee}, \ldots , U_d^{\vee}$ are $\cA$-modules. In fact, they are line bundles by assumption on $1$-genericity.
        To check that the associated maps
        \[
            g^*f^*\OO(-1) \otimes V_1^{\vee} \to
            U_1^{\vee},\quad \OO_{\ZZ}\otimes V_i^{\vee}
            \to U_i^{\vee}\quad\mbox{for}\quad i=2, \ldots ,d
        \]
        are surjective, we can argue point by point, so it follows from~\S\ref{ssec:kpointsAbstractSecant}.
        Each of the maps yields a globally generated line bundle on $\ZZ$, hence jointly they produce a morphism
        \begin{equation}\label{eq:Zmorphism}
            \ZZ\to S\times \mathbb{P}V_1\times \ldots \times \mathbb{P}V_d.
        \end{equation}
        %
        The restrictions $\varphi_2, \ldots
        ,\varphi_d$ are jointly spanning at any point, so
        by~\eqref{eq:jointSpanConcrete} we get that the map
        \begin{equation}\label{eq:surj}
            \OO_S\otimes V_2^{\vee}\otimes \ldots \otimes V_d^{\vee}\to  U_2^{\vee}\otimes_{\cA} \ldots \otimes_{\cA} U_d^{\vee}
        \end{equation}
        is surjective pointwise, so it is surjective. This means that the linear span $\spann{\ZZ}\subseteq \OO_S\otimes V_{2, \ldots ,d}$ defined in Proposition~\ref{ref:linearSpanAbstractRelative:prop} is a
        rank $m$ vector subbundle. Taking into account also $\varphi_1$, this implies in particular that the map~\eqref{eq:Zmorphism} is a closed embedding,
        so it yields a morphism $S \to \Hilb_m(\mathbb{P}V_1\times \ldots \times \mathbb{P}V_d)$. Since members of $f^*\Tppar{}$ have
        minimal border rank, the map actually factors through $\Hilb_m^{\sm}(\mathbb{P}V_1\times \ldots \times \mathbb{P}V_d)$.
        Projecting the tensor $f^*\Tppar{}$ using $\varphi_1, \ldots ,\varphi_d$, we obtain a tensor $T\colon f^*\OO(-1) \to \OO_S\otimes \tspace$,
        which lies in the family of spans of $\ZZ$ by Proposition~\ref{ref:linearSpanAbstractRelative:prop}.
        Since $\Tppar{}$ is $1$-generic pointwise, the tensor $T$ is pointwise a restriction of the evaluation tensor obtained
        from a dual generator, as claimed.
        The tensor $T$ yields a map $S \to \AsigmaHilb_m$.

        Conversely, take a morphism $S\to \AsigmaHilb_m$. It gives a morphism $S\to \Hilb_m^{\sm}(\mathbb{P}V_1\times \ldots \times \mathbb{P}V_d)$,
        which corresponds to a family $f\colon \ZZ\to S$ and its closed embedding $\ZZ\into S\times \mathbb{P}V_1\times \ldots \times \mathbb{P}V_d$.
        Let $\cA := f_*\OO_{\ZZ}$ be the corresponding algebra.
        The induced maps $\ZZ\to \mathbb{P}V_i$ yield maps $\cA\otimes_{\kk} V_i^{\vee}\onto \WW_i^{\vee}$ for some line bundles $\WW_i^{\vee}$ on $\ZZ$.
        Let $\varphi_i\colon \WW_i \to \OO_S \otimes_{\kk} V_i$ be the associated restriction maps; they are regular by their very construction.
        These maps produce a composition
        \begin{equation}\label{eq:tmpRelativeSpan}
            (\WW_1^{\vee}\otimes_{\cA} \ldots \otimes_{\cA} \WW_d^{\vee})^{\vee}\to \WW_1\otimes_{\OO_S} \ldots \otimes_{\OO_S} \WW_{d} \to \OO_S\otimes_{\kk} \tspace.
        \end{equation}
        By assumption and by Proposition~\ref{ref:linearSpanAbstractRelative:prop} this composition is injective pointwise, so (by Lemma~\ref{ref:subbudle:lem})
        it is a subbundle.
        The morphism $S\to \AsigmaHilb_m$ yields also a family of tensors $T\colon \LL\into \spann{\ZZ} \into \OO_S\otimes \tspace$
        which factors from the image of~\eqref{eq:tmpRelativeSpan}. Since~\eqref{eq:tmpRelativeSpan} is a subbundle,
        this family gives a family of tensors
        \[
            \Tppar{}\colon \LL\into
            (\WW_1^{\vee}\otimes_{\cA} \ldots \otimes_{\cA} \WW_d^{\vee})^{\vee}\into \WW_1\otimes_{\OO_S} \ldots \otimes_{\OO_S} \WW_d.
        \]
        By assumption, at every point, the family $\Tppar{}$ is an evaluation on a dual socle generator,
        see Corollary~\ref{ref:dualSocle:cor}, in the smoothable algebra, so every of these tensors has minimal border rank.
        For clarity, we replace $\WW_1$ by $\WW_1\otimes \LL^{\vee}$, so that
        $\Tppar{}$ becomes a map from $\OO_S$ and $\varphi_1\colon \WW_1 \to
        \LL^{\vee}\otimes_{\kk} V_1$.

        Using the restriction maps, we obtain tensors $\varphi_{2, \ldots
        ,d}(\Tppar{})\colon \OO_S\to \WW_1 \otimes_{\kk} V_{2, \ldots ,d}$,
        $\varphi_{3, \ldots ,d}(\Tppar{})\colon \OO_S \into \WW_1\otimes \WW_2\otimes_{\kk} V_{3, \ldots ,d}$, etc.
        The tensor $\varphi_{2, \ldots ,d}(\Tppar{})$ yields a contraction $\WW_1^{\vee}\to \OO_S^{\vee}\otimes_{\kk} V_{2, \ldots ,d}$.
        By~\S\ref{ssec:kpointsAbstractSecant}, this contraction is injective pointwise, so it is a subbundle. We obtain a
        subbundle $U_1^{\vee}\subseteq \OO_S\otimes_{\kk} V_{2, \ldots ,d}$ and an isomorphism $\WW_1^{\vee}\simeq U_1^{\vee}$.
        Similarly, we produce $U_2^{\vee}$, \ldots , $U_d^{\vee}$ and a tensor $\OO_S\into U_{1, \ldots ,d}$ isomorphic to $\Tppar{}$.
        The tensor $\Tppar{}$ is $1$-generic pointwise, hence so is $\OO_S\into U_{1, \ldots ,d}$.

        By assumption and by Proposition~\ref{ref:linearSpanAbstractRelative:prop} also the composition
        \[
            (\WW_2^{\vee}\otimes_{\cA} \ldots \otimes_{\cA} \WW_d^{\vee})^{\vee}\to \WW_2\otimes_{\OO_S} \ldots \otimes_{\OO_S} \WW_{d} \to \OO_S\otimes_{\kk} V_{2, \ldots ,d}.
        \]
        is injective pointwise, so (by Lemma~\ref{ref:subbudle:lem})
        it is a subbundle, hence $\varphi_2, \ldots ,\varphi_d$ are jointly spanning pointwise.
    \end{proof}


    \subsection{Relation with identifiability}\label{ssec:identifiability}

        Theorem~\ref{ref:abstractInsideConcise:thm} gives us plenty of information about the birational behaviour of
        the map $\rho$. The \emph{(nonembedded) expected dimension} of $\sigma_m\subseteq \mathbb{P}\tspace$ is the number 
        $m\cdot (\dim \Seg + 1) - 1 = m\cdot (\sum\dim V_{\bullet} - (d-1)) - 1$. We add the word ``nonembedded'', since
        it is more common to define the \emph{expected dimension} by
        \[
            \expdim \sigma_m = \min\left( m\cdot (\dim \Seg + 1)-1,\ \dim \mathbb{P}\tspace \right).
        \]
        The critical value of $m$ in the above formula is
        \begin{equation}\label{eq:scrit}
            s(\Seg) := \left\lfloor \frac{\prod_{i=1}^{d}\dim V_i}{\sum_{i=1}^d (\dim V_i - 1)+1} \right\rfloor.
        \end{equation}
        Below, we usually assume $m < s(\Seg)$, so that the nonembedded and usual expected dimensions agree.
        In contrast with $\sigma_m$, the dimension of $\csigma_m$ is always the expected one, as follows from Corollary~\ref{ref:dimensionOfcsigma:cor}.

        The Segre variety $\Seg \subseteq \mathbb{P}\tspace$ is \emph{$m$-nondefective} if
        $\dim \sigma_m = \expdim \sigma_m$, or equivalently, the map $\AsigmaHilb_m \to \sigma_m$
        is generically finite. Otherwise, it is \emph{$m$-defective}.
        This variety is \emph{$m$-identifiable} if the map $\AsigmaHilb_m\to \sigma_m$ is birational. 
        This can happen only if $\dim \sigma_m$ is equal to the nonembedded expected dimension.

        \begin{lemma}[Concise secants and defectivity]\label{ref:identifiability:lem}
            Take $V_{1}, \ldots ,V_d$ and $m < s(\Seg)$.
            Then the following hold:
            \begin{enumerate}
                \item $\Seg\subseteq \mathbb{P}\tspace$ is $m$-nondefective if and only if $\rho\colon \csigma_m\to \sigma_m$ is generically finite,
                \item $\Seg \subseteq \mathbb{P}\tspace$ is $m$-identifiable if and only if $\rho\colon \csigma_m\to \sigma_m$ is birational.
            \end{enumerate}
        \end{lemma}
        \begin{proof}
            Directly from definitions and Theorem~\ref{ref:abstractInsideConcise:thm}.
        \end{proof}

        The above is a birational result. Using $\csigma_m$, we can transform it into a result about varieties
        of sums of powers.
        \begin{corollary}\label{ref:infiniteFibers:cor}
            Assume that $\Seg$ is $m$-identifiable and let $[T]\in \sigma_m$ be a normal point.
            Then the concise variety of sums of powers $\cVSP([T])$ is either a point or
            it is positive-dimensional.
        \end{corollary}
        \begin{proof}
            By definition, the scheme $\cVSP([T])$ is equal to $\rho^{-1}([T])$.
            By Zariski Main Theorem, this fibre of $\rho$ is connected.
        \end{proof}
        
        By Oeding's Conjecture~\ref{ref:Oeding:conj}, in ``good'' situations normality should hold.

        The questions of defectivity and identifiability of Segre (and Segre-Veronese) varieties have gathered much attention
        classically and since the Alexander and Hirschowitz theorem~\cite{AlexHirsch},
        the field exploded with new beautiful results. A very much non-exhaustive list, mostly geared towards the Segre case, is \cite{Chiantini_Ciliberto__Weakly_defective, Chiantini_Ciliberto__On_the_concept, Chiantini_Ciliberto__On_the_dimension, Abo_Brambilla, Abo_Ottaviani_Peterson_segre, Araujo_Massarenti_Rischter, Laface_Postinghel, Ballico__Partially_symmetric, Massarenti_Mella, Vannieuwenhoven_Vandebril_Meerbergen, Chiantini_Ottaviani_Vannieuwenhoven, Chiantini_Ottaviani_Vannieuwenhoven__On_generic_identifiability, Laface_Massarenti_Rischter, Ballico_Bernardi_Santarsiero, Dolevalek_Ken}. 
        To begin, for example, see the discussion in \cite{Abo_Brambilla_Galuppi_Oneto} or in~\cite{Casarotti_Mella}. For Segre, the picture, at least conjecturally, is quite complete.
        \begin{enumerate}
            \item For $m\leq 6$, the $m$-defective cases are listed in~\cite{Abo_Ottaviani_Peterson_segre}.
                The authors also formulate Question~6.6 which conjecturally lists \emph{all} defective cases in the Segre case.
            \item Casarotti-Mella~\cite[Theorem~18, Corollary~22]{Casarotti_Mella} (see also~\cite{Massarenti_Mella}) provide
                powerful tools to transform results about nondefectivity to ones about identifiability. We list two special cases below.
            \item If $\dim V_{\bullet} = 2$ then $\rho$ is birational for every $d$ and $m < s(\Seg)$ except for
                $(m, d) = (2, 1), (3,2), (4,2), (5,4), (6, 9)$, see~\cite[Theorem~26]{Casarotti_Mella}.
            \item If $d = 3$, $\dim V_1 = \dim V_2 = \dim V_3$, then $\rho$ is
                birational for every $m < s(\Seg)$,
                see~\cite[Theorem~28]{Casarotti_Mella}.
        \end{enumerate}

{\small
\newcommand{\etalchar}[1]{$^{#1}$}

}

\end{document}